\documentclass[12pt, reqno]{amsart}

\usepackage{amssymb, amsmath, amsthm, mathrsfs}
\usepackage[backref]{hyperref}

\usepackage{amscd}   
\usepackage{fullpage}
\usepackage{tikz-cd}
\usepackage{verbatim}
\usepackage{stmaryrd}

\DeclareFontEncoding{OT2}{}{} 

\usepackage{xcolor}

\newtheorem{lemma}{Lemma}[section]
\newtheorem{theorem}[lemma]{Theorem}
\newtheorem*{theorem*}{Theorem}

\newtheorem{prop}[lemma]{Proposition}
\newtheorem{cor}[lemma]{Corollary}
\newtheorem{conj}[lemma]{Conjecture}

\newtheorem{claim*}{Claim}
\newtheorem{thm}[lemma]{Theorem}

\theoremstyle{definition}
\newtheorem{rmk}[lemma]{Remark}
\newtheorem{example}[lemma]{Example}
\newtheorem{remark}[lemma]{Remark}
\newtheorem{defn}[lemma]{Definition}

\newcommand{\Aff}{{\mathbb A}}

\newcommand{\PP}{{\mathbb P}}
\newcommand{\C}{{\mathbb C}}
\newcommand{\F}{{\mathbb F}}
\newcommand{\Q}{{\mathbb Q}}
\newcommand{\R}{{\mathbb R}}

\newcommand{\Z}{{\mathbb Z}}

\newcommand{\bbS}{\mathbb S}

\newcommand{\Xbar}{{\overline{X}}}
\newcommand{\Qbar}{{\overline{\Q}}}

\newcommand{\Vbar}{{\overline{V}}}

\newcommand{\Adeles}{{\mathbb A}}

\newcommand{\ve}{{\varepsilon}}

\newcommand{\calA}{{\mathcal A}}
\newcommand{\calB}{{\mathcal B}}

\newcommand{\calD}{{\mathcal D}}

\newcommand{\calF}{{\mathcal F}}

\newcommand{\calH}{{\mathcal H}}
\newcommand{\calI}{{\mathcal I}}

\newcommand{\calK}{{\mathcal K}}
\newcommand{\calL}{{\mathcal L}}

\newcommand{\calO}{{\mathcal O}}

\newcommand{\calS}{{\mathcal S}}

\newcommand{\calU}{{\mathcal U}}

\newcommand{\calX}{{\mathcal X}}

\newcommand{\calZ}{{\mathcal Z}}

\newcommand{\fraka}{{\mathfrak a}}
\newcommand{\frakb}{{\mathfrak b}}
\newcommand{\frakc}{{\mathfrak c}}
\newcommand{\frakd}{{\mathfrak d}}

\newcommand{\resH}{\mathcal H}

\DeclareMathOperator{\HH}{H}

\DeclareMathOperator{\Tr}{Tr}

\DeclareMathOperator{\Frob}{Frob}
\DeclareMathOperator{\coker}{coker}

\DeclareMathOperator{\inv}{inv}

\DeclareMathOperator{\im}{im}
\DeclareMathOperator{\Rea}{Re}

\DeclareMathOperator{\Aut}{Aut}
\DeclareMathOperator{\Gal}{Gal}

\DeclareMathOperator{\cores}{cores}
\DeclareMathOperator{\res}{res}

\DeclareMathOperator{\Res}{Res}

\DeclareMathOperator{\Br}{Br}

\DeclareMathOperator{\Cl}{Cl}
\DeclareMathOperator{\divv}{div}

\DeclareMathOperator{\Bl}{Bl}
\DeclareMathOperator{\Sym}{Sym}

\DeclareMathOperator{\Pic}{Pic}
\DeclareMathOperator{\Jac}{Jac}

\DeclareMathOperator{\Spec}{Spec}

\DeclareMathOperator{\tors}{tors}

\DeclareMathOperator{\vol}{vol}

\DeclareMathOperator{\rank}{rk}

\DeclareMathOperator{\Hilb}{Hilb}
\DeclareMathOperator{\NS}{NS}

\newcommand{\isom}{\cong}

\newcommand{\sumS}{\sideset{}{^S}\sum}
\newcommand{\prodS}{\sideset{}{^S}\prod}

\numberwithin{equation}{section}
\numberwithin{table}{section}

 \ExplSyntaxOn

\NewDocumentCommand\stackar{ m O{} }
{
    \int_step_inline:nn { #1 }
    {
        \ar[->,shift~left=\fp_eval:n {(#1-1)/2*3 - (##1-1)*3 }pt,#2]
    }
}

\NewDocumentCommand\altstackar{ m O{} }
{
    \int_step_inline:nn { #1 }
    {
        \int_if_odd:nTF { ##1 }
        {
            \ar[->,shift~left=\fp_eval:n {(#1-1)/2*1.5 - (##1-1)*1.5 }pt,#2]
        }
        {
            \ar[<-,shift~left=\fp_eval:n {(#1-1)/2*1.5 - (##1-1)*1.5 }pt,#2]
        }
    }
}

\ExplSyntaxOff
\title{ Counting quadratic points on Fano varieties}
\author{Francesca Balestrieri}
\author{Kevin Destagnol}
\author{Julian Lyczak}
\author{Jennifer Park}
\author{Nick Rome}

\begin{document}
\maketitle
\begin{abstract}
This paper initiates the systematic study of the number of points of bounded height on symmetric squares of weak Fano varieties. We provide a general framework for establishing the point count on $\Sym^2 X$. In the specific case of surfaces, we relate this to the Manin--Peyre conjecture for $\Hilb^2 X$, and prove the conjecture for an infinite family of non-split quadric surfaces. 
In order to achieve the predicted asymptotic, we show that a type II thin set of a new flavour must be removed.

To establish our counting result for the specific family of surfaces, we generalise existing lattice point counting techniques to lattices defined over rings of integers. This reduces the dimension of the problem and yields improved error terms. Another key tool we develop is a collection of results for summing Euler products over quadratic extensions. We use this to show moments of $L$-functions at $s=1$ are constant on average in quadratic twist families.

\end{abstract}

\setcounter{tocdepth}{1}
\tableofcontents

\section{Introduction}

The central object of study in arithmetic geometry is the set of rational points on varieties over number fields. However, focusing on the set of $F$-points of a variety for a single number field $F$ provides a limited perspective on the arithmetic of the variety. A richer perspective is given by looking at all $\Qbar$-points, while accounting for the action of the absolute Galois group $\Gal(\Qbar/\Q)$. The set of closed points of a variety $X$ defined over $\Q$, corresponding to the orbits under $\Gal(\Qbar/\Q)$ of $X(\Qbar)$, admits a partition by their degree 
$$
\bigsqcup_{r=1}^{\infty}X(\Q,r), \quad  \mbox{where}
\quad 
X(\Q,r) := \{P \in X: [\Q(P):\Q] = r\}.
$$
A point of degree $r$ gives rise to a rational point on the symmetric power $\Sym^r X = X^r /S_r$ and the arithmetic of $\Sym^r X$ is intimately related to the study of $X(\Q,r)$, the set of the degree $r$ points on $X$.

In the case of curves $X=C$, the sets $\Sym^r C(\Q)$ have been studied extensively over many years: Faltings~\cite{MR718935} (rational points on the image of $\Sym^rC$ mapped to $\Jac(C)$), Gross--Rohrlich~\cite{MR491708} ($C$ being the Fermat curve), Hindry~\cite{MR907946} ($r = 2$), Mazur~\cite{MR488287} ($C = X_0(N)$), Harris--Silverman~\cite{MR1055774} ($r=2$), Abramovich--Harris~\cite{MR1104789} ($r = 2, 3, 4$), Vojta~\cite{MR1151542} (heights of points of bounded degree), Masser--Vaaler~\cite{MR2247898} ($C = \PP^1$ and general $r$), and many others. 
Studying fixed degree points has been particularly fruitful for modular curves, yielding applications such as modularity of elliptic curves over totally real number fields \cite{MR3359051, MR4150250,MR4402658}, classification of torsion groups of elliptic curves over number fields \cite{MR1172689}, and the resolution of variants of Fermat's Last Theorem over number fields \cite{MR3383161}.

In contrast, the $\Q$-points on $\Sym^r X$ for higher dimensional $X$ have been less studied. While there have been results on the existence of degree $r$ points, such as the work of Coray \cite{MR429731} and Creutz--Viray \cite{MR4633628, CreutzViray24}, the only results on the quantitative study of $\Sym^r X(\Q)$ so far were restricted to the case when $X$ is either isomorphic to $\PP^n$ or when $X$ is $\PP^1 \times \PP^1$, these results are listed in \textsection\ref{ss:uniform}. 
For the quantitative study of $\Sym^r X$, we want to describe the asymptotic behaviour of the counting function
\[
N_\Q(\calU, B) : = \#\{ x \in \calU : H(x) \leq B\},
\]
where $ \calU \subseteq \Sym^r X (\Q)$ is a suitable subset and $H$ is an anticanonical adelic height on $\Sym^r X$.

\subsection{Main results} 
The purpose of this paper is twofold. The first objective is to provide a systematic framework for the quantitative study of the set of rational points on the symmetric square of any smooth weak Fano variety $X$ over $\Q$. Our second objective is to use this framework to count the rational points on the symmetric squares of non-split quadric surfaces.
We observe that the case of surfaces is of particular interest, as when $\dim X=2$ the counting problem on $\Sym^2 X$ falls under the Manin--Peyre conjecture.
 
 \begin{conj}[Manin--Peyre]\label{conj:manin}
Let $V$ be a smooth weak Fano variety over a number field $F$ such that $V(F)$ is Zariski dense, with a relative height function associated to the anticanonical line bundle. Then, there exists a cothin subset $\calU \subseteq V(F)$ and a constant $c_{V,F}$, called the \emph{Peyre constant} which depends on the height, such that 
\begin{equation}\label{eq:manin}
N_F(\calU, B) \sim c_{V,F} B (\log B)^{\rank (\Pic V)-1}.
\end{equation}
Furthermore, the value of the leading constant takes the form $c_{V,F} = \alpha(V) \beta(V) \tau(V)$, whose definitions we recall in \textsection\ref{S:maninpeyre}.
\end{conj}

 Thus we verify this central conjecture for the symmetric squares of our infinite family.

\subsubsection{General framework for counting quadratic points}\label{sss:framework}
Let $X/\Q$ be a weak Fano variety. If $\dim X =1$ then $\Sym^2X \cong \PP^2$ and the point count is already well understood. If $\dim X \geq 2$ then to count rational points on $\Sym^2 X$ we start by partitioning these points as
\[
\Sym^2X(\Q) = \{[P_1,P_2]: P_1,P_2\in X(\Q)\} \sqcup \bigsqcup_{[K:\Q] = 2}\{[P, \bar{P}]: P \in X(K), [\Q(P):\Q] = 2\}.
\]
The natural strategy for studying quadratic points on $X$, used by Schmidt in the case of $\PP^n$~\cite{MR1330740}, is to count the $K$-points on $X_K$ in each quadratic field $K$ and then sum all these contributions over almost all quadratic extensions $K/\Q$. For a single quadratic field $K$ the number of points of bounded height in $X(K)$ is predicted by the Manin--Peyre conjecture.
In order to execute our game plan to count rational points on $\Sym^2 X$, we must be able to sum the leading term in Conjecture~\ref{conj:manin} over the quadratic extensions $K/\Q$. 
We show that $\rank(\Pic X), \alpha(X)$, and $\beta(X)$ do not vary under base change for all but finitely many quadratic extensions (see Proposition~\ref{peyrequad}). This implies that the central difficulty in summing the Peyre constants over quadratic fields lies in summing the Tamagawa numbers $\tau(X_K)$ over almost all $K$. 

The first main result in this paper is not only an asymptotic for this sum in full generality, but also the direct relation to the Tamagawa number of the crepant resolution $\Hilb^2 X \to \Sym^2 X$ in the case that $X$ is a surface.

\begin{thm}[Theorem~\ref{thm:summingc}] \label{thm:introsumconsts}
Let $X/\Q$ be a smooth weak Fano variety.
\begin{itemize}
\item[(a)] If $\dim X = 2$, that is $X$ is a generalised del Pezzo surface, then as $Y \to \infty$
\[
\sum_{\substack{[K: \Q] = 2\\ \vert \Delta_{K/\Q} \vert \leq Y
}} \tau(X_K) = \tau(\Hilb^2 X) \log Y + O_X(1).
\]
\item[(b)] If $\dim X > 3$, then
\[
\sum_{\substack{[K: \Q] = 2\\ \vert \Delta_{K/\Q} \vert \leq Y
}} \tau(X_K) 
\]
converges absolutely as $Y \to \infty$.
\end{itemize}
\end{thm}

The proof requires a deep understanding of the connections between the Peyre constants of the original variety $X$ and of its quadratic base changes $X_K$. In the case of a surface $X$, we additionally need to relate these invariants to the Peyre constants of $\Hilb^2 X$.

An application of the strategy of counting by field requires a very precise uniform control of the point counts in quadratic fields. In practice, providing such uniform point count estimates over all quadratic extensions $K$ seems to be very difficult -- we refer to \textsection\ref{ss:uniform} for more details. One major takeaway from this paper is that if one can provide uniform enough estimates for the $K$-points in almost all quadratic extensions $K/\Q$, then our machinery immediately gives the point count for $\Sym^2 X$ for any weak Fano variety $X$ of dimension at least $3$. In the case of surfaces, there are is an additional hurdle to overcome, namely determining the correct relation between the $B$ in Conjecture~\ref{conj:manin} and the $Y$ in Theorem~\ref{thm:introsumconsts} which is captured by the \emph{cutoff}(cf.\ \S\ref{ss:introcutoffs}).

\subsubsection{Establishing the quadratic Manin--Peyre conjecture for an infinite family}

For any del Pezzo surface $X$ the crepant resolution of singularities $\Hilb^2 X$ is weak Fano. Thus, the point count on $\Sym^2X$ falls under the remit of the Manin--Peyre conjecture, explaining the appearance of $\Hilb^2 X$ in Theorem~\ref{thm:introsumconsts}(a). Previously, the Manin--Peyre conjecture for $\Sym^2X$ of a surface $X$ had only been established for $X = \PP^2$ and $\PP^1 \times \PP^1$ by Le Rudulier \cite{lerudulier-thesis}, building on the counting results of Schmidt~\cite{MR1330740}.
As the second main objective of this paper, we produce an asymptotic formula for the counting function of $\Sym^2 X(\Q)$ for $X$ belonging to an infinite family of non-split quadric surfaces, thus proving the quadratic Manin--Peyre conjecture for this infinite family. Interestingly, one has to remove a new flavour of type II thin set to obtain the predicted leading constant (see Lemma~\ref{lem:thinset}).

\begin{thm} \label{thm:mainthm1}
Let $d$ be a squarefree integer and let $X_d = \{x^2 - dy^2 = zw\} \subseteq \PP^3$. Then there is an infinite family of anticanonical adelic heights on $X_d$ for which there exists an explicit thin set $\calZ \subseteq \Sym^2 X_d(\Q)$ such that 
\begin{enumerate}
    \item[(a)] We have the asymptotic formula
    \[N_\Q(\Sym^2X_d(\Q) \setminus \calZ,B) \sim c_{\Hilb^2 X_d, \Q} B \log B,\]
    which agrees with the order of magnitude predicted by Manin \emph{and} the leading constant predicted by Peyre;
\item[(b)] There exists an explicit  $c_\calZ>0$ such that \[N_\Q(\calZ, B) \sim c_\calZ B \log B.\] 
\end{enumerate}
\end{thm}

For the heights under consideration, see Definition~\ref{defn:heightonXd} and Remark~\ref{rem:infinitely many heights}. We also note that the thin set $\calZ$ contributes with the same order of magnitude to the total counting function and hence needs to be removed to obtain the predicted leading constant.

The surfaces whose quadratic points are studied in Theorem \ref{thm:mainthm1} are del Pezzo surfaces of degree 8.
The next surfaces for which the quantitative study of quadratic points could be in reach are those for which the Manin--Peyre conjecture is already settled for the rational points.
For example, the conjecture is known to hold for $\Q$-points on del Pezzo surfaces of degree $\geq 6$ \cite{MR1353919}, split del Pezzo surfaces of degree 5 \cite{MR1909606},  del Pezzo surfaces of degree 4 with a conic fibration \cite{MR2838351}, certain singular quartic del Pezzo surfaces~\cite{MR2543667,MR2980925,MR2853047}, and certain singular cubic surfaces \cite{MR1679838,MR2332351,MR2520769,MR2990624}. Moreover, Ch{\^a}telet surfaces have received considerable attention \cite{MR2874644,MR3103132,MR3011504,MR3517531}, culminating in the recent proof of the Manin--Peyre conjecture for all Ch{\^a}telet surfaces by Woo~\cite{Woo}.
Theorem \ref{thm:mainthm1} illustrates that it is now possible to move beyond counting $\Q$-points and investigate systematically the quantitative arithmetic of higher degree points on surfaces.

\begin{remark}
     The Manin--Peyre conjecture has been extended to the realm of algebraic stacks by Ellenberg--Satriano--Zureick-Brown \cite{MR4557890} and Darda--Yasuda \cite{DY}, unifying Conjecture \ref{conj:manin} with other completely distinct arithmetic problems, such as the distribution of elliptic curves with certain properties or Malle's conjecture on the distribution of number fields with a given Galois group. Recently, Yin~\cite{Yin} provided upper and lower bounds of the expected order of magnitude for the stacky point count on the stacky symmetric square of $\PP^1$. The techniques used in Theorem \ref{thm:mainthm1} can be adapted to the stacky setting, and in forthcoming work of the authors, we will prove an asymptotic formula for the stacky version of the count.
\end{remark}

\subsection{Additional notable results} In the course of proving the two main theorems of our paper, we also offer the following new results and insights.

\subsubsection{A new form of thin sets.} In Manin's original prediction, it was observed that one might need to remove accumulating closed subsets. However, due to the many counterexamples that appeared throughout the years (see e.g.\ \cite{MR1401626,lerudulier-thesis,MR4167086, GaoManin}), the modern formulation of the Manin--Peyre conjecture requires  the potential removal of \emph{thin} subsets instead. An important question is about understanding the type of thin sets that should be removed, and a conjectural answer has been provided by Lehmann--Sengupta--Tanimoto~\cite{MR4472281}. In a recent AIM workshop entitled ``Degree $d$ points on algebraic surfaces", it was asked what the shape of the accumulating subset in Manin's conjecture for $\Sym^r X$ ought to be\footnote{This question can be found in the notes from the Problem Session at: \url{http://www.aimpl.org/degreedsurface/1/}}.

Through the cases of $\Sym^2\PP^2$ and $\Sym^2 (\PP^1 \times \PP^1)$, new interest in the study of thin sets and their connection to Manin's conjecture was stimulated by the work of Le Rudulier~\cite{lerudulier-thesis}. In fact, the set described by Le Rudulier is a special case of a thin set which needs to be removed in general in order for the Manin--Peyre conjecture to hold.

\begin{lemma}[Lemma \ref{lem:kptsthin}]
\label{lem:thinset}
Let $X/\Q$ be a variety. For a fixed quadratic extension $K$, the points $[P_1,P_2]\in \Sym^2 X(\Q)$ for conjugate $P_i \in V(K)$ form a thin set of type II.
\end{lemma}

In fact, Le Rudulier's thin set corresponds to the {\'e}tale extension $\Q\times \Q$ and the thin set $\calZ$ in the statement of Theorem~\ref{thm:mainthm1} is the union of the points in Lemma~\ref{lem:thinset} for $K = \Q\times\Q$ and $K=\Q(\sqrt{d})$.
    
\subsubsection{Quadratic points of small height on surfaces}\label{ss:introcutoffs}
The Northcott property tells us that there are only finitely many quadratic points of bounded height on $X$. Therefore if the height of a quadratic point is bounded, then the discriminant of its field of definition should be too. 
Theorem~\ref{thm:introsumconsts}(a) shows that for surfaces the sum of $\tau(X_K)$ over all quadratic fields diverges. For the Manin--Peyre conjecture to hold the sum over fields $K$ must be truncated at a power of $B$, since $\rank \Pic ( \Hilb^2 X) = \rank \Pic( X) + 1$ by Proposition~\ref{prop:rho and beta of Hilb2}. More precisely, there should exist some $\gamma >0$, called a \emph{cutoff}, such that for all $K$ and $P \in X(K)\setminus X(\Q)$
\[
2\gamma \log H(P) \geq \log \vert \Delta_{K/\Q} \vert + O(1).
\]
The factor of 2 is included for convenience in \textsection\ref{S:fringes}.
Understanding precisely which fields can contribute points of bounded height is a subtle point. 
Silverman \cite{MR747871} first studied this problem for points in projective spaces and abelian varieties. We extend the study of cutoffs to weak Fano varieties in \textsection\ref{S:fringes}, to determine where to truncate the sum over quadratic fields. We are able to prove general properties of cutoffs, including their existence, that they only depend on the underlying line bundle of the height function, and their behaviour under certain geometric constructions. In addition, we provide techniques for producing lower bounds for cutoffs and determining the smallest possible cutoff.

\begin{remark} In order to prove Theorem \ref{thm:mainthm1}, it is pivotal that the smallest possible cutoff for $X_d$ be computed precisely, as it contributes to the final leading constant. 
In Lemma~\ref{lem:constsxd} we explicitly compute the cutoff $\gamma = 1/2$. In fact, we show that this is the optimal cutoff and that the alpha constants are related by the expression $\alpha(\Hilb^2 X) = \tfrac12 \cdot \gamma \cdot  \alpha(X)$. While we expect this relationship to hold for some other surfaces $X$, it is much more complex in general, as can be seen in the example of $\PP^1 \times \PP^1$ (see Remark~\ref{rem:p1p1}).
\end{remark}

\subsubsection{Uniform point counting}\label{ss:uniform}
In order to prove Theorem \ref{thm:mainthm1}, we must produce new asymptotic formulae for the point counts of $X_d(K)$ for which the main and error terms are both explicit in their dependence on the underlying field $K$, so that we can sum them over varying $K$. This is the first instance of such uniform counts on varieties of dimension at least 2 and of projective degree at least 2.

\begin{prop}\label{prop:introwidmer}
Let $K/\Q$ be any quadratic field other than $\Q(\sqrt{d})$, and let us write $M=K(\sqrt d)$. For any $B \geq 1$ we have that $\#\{ P \in X_d(K): H_K(P) \leq B, \Q(P) = K\}$ is equal to
\[
c_{X_d, K}B + O_{d,\epsilon}\left(B^{7/8}\frac{(h_{M}R_M)^{3/4} (1+\res_{s=1} \zeta_M(s))}{|\Delta_{K/\Q}|^{3/2}}\right).
\]
\end{prop}
There are many examples of the Manin--Peyre conjecture being proven for varieties defined over number fields of degree larger than $1$ via every known method, whether using universal torsors \cite{MR0557080, MR3552013,derenthalbernert}, Poisson summation \cite{MR1353919, MR1797654} or the circle method \cite{MR1446148,MR3229043,MR3610007}.
However, it is standard to fix the field and allow all error terms to depend on it in a completely unspecified manner. Thus, none of these results can be used directly in the study of symmetric powers.

While it is likely that these methods might provide some uniformity if the proofs are scrutinised closely, there is a further complication for which we need to account. Every quadratic field contains the rational numbers;  hence, if one counts all the points in a given quadratic field, one counts the rational points repeatedly for each field. To avoid this enormous overcount, it is necessary, as in Proposition~\ref{prop:introwidmer}, to  count only the \emph{pure} points -- namely, those points whose field of definition is the quadratic field itself. Finding a way to capture only these points is a major obstacle in extending the standard techniques of Manin's conjecture to the setting of symmetric squares.

These two problems -- uniform point counting and counting pure points -- have previously only been simultaneously overcome for (linear subvarieties of) projective space $\PP_F^n$:
\begin{itemize}
    \item Schmidt~\cite{MR1330740}, $r=2$, any $n$, $F=\Q$;
    \item Gao~\cite{MR2693933}, $r \geq 3$, $n \geq r+1$, $F=\Q$;
    \item Widmer~\cite{MR2651154,MR2645051}, any $r$,  $n > \frac{5r}{2} + 4 + \frac{2}{[F:\Q]r}$, any number field $F$.
    \end{itemize}
This allowed Schmidt, Gao and Widmer to count degree $r$ points on $\PP_F^n$ in their respective regimes. Moreover, Guignard~\cite{MR3541701} was able to count degree $r$ points on $\PP^n_F$ for any number field $F$ when $r$ is a prime and $n>r+2$. However the error terms achieved in the $n=1$ cases of these results are always insufficient for the field by field approach to succeed. Instead, Schmidt and Masser--Vaaler~\cite{MR2247898} were able to count degree 2 and general degree $r$ points, respectively, on $\PP^1_\Q$ by parametrising the points via polynomials and counting these.
A similar viewpoint was used by Kass--Thorne~\cite{KassThorne} to produce an asymptotic of the predicted magnitude for $\Hilb^2 \PP^2$, and by Derenthal~\cite{ulrichfourfold} who established Manin's conjecture for the chordal cubic fourfold, which is isomorphic to $\Sym^2 \PP^2_\Q$.

\subsubsection{$\calO_F$-lattices}
Since Schmidt's approach combined with the traditional geometry of numbers machinery falls just short of what is needed to establish Proposition~\ref{prop:introwidmer}, we instead develop the theory of counting in $\calO_F$-lattices culminating in a general version of Davenport's lemma, Proposition~\ref{prop:generalised Davenport}. An ideal in a number field $M$ is not only invariant under multiplication by integers, but also under multiplication by the integers in any fixed subfield. We exploit this extra symmetry by viewing an ideal in  a number field not as a finite dimensional $\Z$-module, but rather as a finite dimensional module over the ring of integers of a fixed subfield $F$ of $M$. In this way, we are able to reduce the dimension of the problem, affording us better control on the successive minima and thus the final error term. 

Indeed if $F/\Q$ is a number field of degree $d$, $\Lambda$ an $\calO_F$-lattice of rank $n$ and $\calS$ a suitable bounded region of diameter $R$, then the error term in the classical Davenport's lemma saves a factor of $\frac{R}{\lambda_{\Q,nd}}$ where $\lambda_{\Q,nd}$ is the largest $\Z$-successive minimum of $\Lambda$. In comparison, when viewing $\Lambda$ over $\mathcal O_F$, the saving is of $\frac{R}{\lambda_{F,n}}$, where $\lambda_{F,n}$ is the largest $\mathcal O_F$-successive minimum,  which in practice makes the error term much easier to control, and in turn bound, effectively.
%

The only previous instances of counting in $\calO_F$-lattices are due to Bombieri--Vaaler~\cite{BV}, Broberg~\cite{broberg} and Glas--Hochfilzer~\cite{GH}. However, each of these assumes we are given an explicit basis of the $\calO_F$-vector space which the lattice spans. This becomes increasingly tricky to do explicitly and uniformly as one varies ideals and as one varies the underlying field. To handle this issue we develop from the ground up an entirely basis-free version of the theory of $\calO_F$-lattice point counting in \textsection\ref{s:lattices}. 

We believe that the advantage gained from the use of $\calO_F$-lattices and from the framework developed in this paper has potential to be useful for a wide variety of problems in arithmetic statistics. Indeed our success suggests the philosophy that one should perform lattice point counting over the largest field whose size does not need to be controlled.

\subsubsection{New analytic results}\label{ss:introanal}
In order to establish Theorem \ref{thm:introsumconsts}, we need to sum Euler products over prime ideals in quadratic number fields for varying quadratic fields. 
This requires us to produce results on summing multiplicative functions  evaluated at discriminants of quadratic fields (Proposition \ref{prop:discsum}) which may have further uses beyond the focus of this paper.
Another key ingredient (Proposition \ref{prop:trunc}) is an on-average truncation of twisted Artin $L$-series which will also likely be of independent interest. 
An immediate consequence of this truncation result is the following formula for all the moments of the value at $s=1$ of Artin $L$-series in quadratic twist families.

\begin{prop}[Special case of Corollary~\ref{cor:galtwists}]
Let $\rho$ be a Galois representation which factors through a finite extension $L/\Q$, and $t \geq 1$ an integer. There exists an explicit constant $C_{\rho, t}>0$ such that for any $Y \geq 1$ and any $\epsilon>0$ we have
    \[
    \sum_{\substack{[K:\Q]=2\\ \vert \Delta_{K/\Q} \vert \leq Y \\ K \text{ lin.\ disjoint from } L}}
    L(1, \rho \otimes \chi_K)^t
    =
    C_{\rho, t} Y + O_{t,\epsilon} \left( Y^{\frac{2}{3} + \epsilon}\right),
    \]
where $\chi_K$ is the quadratic character associated to the field $K$.
\end{prop}

One way to think about this result is that the values $L(1,\rho \otimes \chi_K)$ are constant on average. This is a fact which we take repeated advantage of throughout the paper. For instance, in Lemma \ref{lem:landau}, we prove a uniform bound for the number of ideals of bounded norm in a biquadratic field, in the spirit of recent work of Lowry-Duda--Taniguchi--Thorne~\cite[Thm.\ 3]{MR4381213}. We replace a $\log^{3} X$ term in their error term by the residue of the associated Dedekind zeta function. When this error is summed, we gain important savings by exploiting that this residue is constant on average.

As another example of how these ideas could prove independently useful, in \textsection\ref{ss:moments}, we establish the following corollary on the moments of class numbers of quadratic fields.

\begin{thm}[Consequence of Corollary~\ref{cor:galtwists}]\label{cor:classnumbers}
For any $Y \geq 1$, let $\mathcal D^\pm(Y)$ denote the set of positive or negative fundamental discriminants, respectively, of absolute value less than $Y$.
Let $r\geq 1$ be an integer. There is an explicit constant $\widetilde{c}_r >0$ such that for any $\epsilon > 0$ we have
    \[
\sum_{d \in \mathcal{D}^-(Y)}  h(d)^r
=
\widetilde{c}_r Y^{\frac{r}{2}+1}
+
O_{r, \epsilon}\left(Y^{\frac{r}{2} + \frac{2}{3} + \epsilon}\right),
    \]
    and
    \[
    \sum_{d \in \mathcal{D}^+(Y)}  h(d)^rR(d)^r
=
\widetilde{c}_r Y^{\frac{r}{2}+1}
+
O_{r, \epsilon}\left(Y^{\frac{r}{2} + \frac{2}{3} + \epsilon}\right),
    \] where $h( d)$ denotes the class number of the quadratic field $\Q(\sqrt{ d})$ and $R(d)$ denotes the regulator of $\Q(\sqrt{d})$.
\end{thm}

\noindent
In the case of real fields, this improves upon the $r=2$ error term due to Taniguchi \cite{MR2410385} and is the first example of such asymptotics for $r>3$. Asymptotics for all moments were produced by Raulf \cite{MR3498624} with error terms saving a factor of $Y^\epsilon$, however her sums are ordered by regulator of the field in the spirit of Sarnak \cite{sarnakI, sarnakII} instead of by discriminant as above. Moreover, in the imaginary case, Theorem \ref{cor:classnumbers} improves the best known exponent for the error term in this problem, $\frac r2 +\frac 34 + \epsilon$, established by Wolke~\cite{MR0252322} in 1972.

\subsection{Sketch of the setup and roadmap to the proof} Our setup is the following. We work on the $S_2$-quasitorsor $\pi \colon X^2 \to \Sym^2 X$ ramified along the diagonal $\Delta^{(2)}=\{[P,P] \colon P \in X\} \subseteq \Sym^2 X$, which is an accumulating subset for the rational points.  By using standard properties of torsors (see e.g.\ \cite{Sko01}), away from the diagonal $\Delta^{(2)}$ we can partition the rational points on $\Sym^2 X$ using the twists $\pi^\sigma$ indexed by $\sigma \in \HH^1(\Q, S_2)$. We prove that such a twist is a natural morphism $\pi^\sigma \colon \Res_{K/\Q} X_K \to \Sym^2 X$, where $K/\Q$ is the quadratic \'etale extension determined by $\sigma$. This geometrically justifies counting rational points on $\Sym^2 X (\Q)$ by counting pure $K$-points on $X_K$ for all quadratic extensions $K/\Q$.  Building on the work of Loughran \cite{MR3430268}, we use the connection to the restriction of scalars to understand some of the invariants of $\Sym^2 X$ and $\Hilb^2 X$, such as the Tamagawa number, in terms of those of $X_K$. Similarly, we can exploit this connection to understand the Brauer--Manin set of $\Sym^2 X$ in terms of those of the $X_K$.

Given this setup, we can now detail the roadmap to the proof. 
In \textsection\ref{S:preliminaries}, we recall standard conventions and facts around Peyre's conjecture. Experts may skip ahead, but we note that, in order to facilitate the comparison across different fields, we use a more general notion of adelic metric and reinterpret parts of the Manin--Peyre conjecture in terms of Galois invariant objects over an algebraically closed field. In \textsection\ref{s:symres},  we undertake the comparison between the geometry of the symmetric square and the restriction of scalars, using the twists of the quasitorsors described above. In \textsection\ref{S:PeyreoverK}, we describe the various invariants in the conjecture in terms of the arithmetic of $X$ over a quadratic extension $K/\Q$. In \textsection\ref{S:Tamagawakening},  we elaborate the full Manin--Peyre prediction for $\Hilb^2 X$ for any del Pezzo surface $X$. In \textsection\ref{S:summingtheconstants}, we develop the tools to sum the Tamagawa numbers of $X_K$ for all quadratic extensions, culminating in the proof of Thereom \ref{thm:introsumconsts}. In \textsection\ref{S:fringes}, we develop the theory of cutoffs, and in \textsection\ref{s:lattices} the theory of general $\calO_F$-lattices. Then \textsection\ref{s:nonsplit} is devoted to the infinite family $\{X_d\colon d \in \Z \textrm{ squarefree}\}$ of non-split quadric surfaces. By viewing these as a restriction of scalars, we establish the optimal cutoff and prove uniform counting theorems for $K$-points.

\section*{Acknowledgments}
This project began while the authors were all attending the trimester programme ``{\`A} la red{\'e}couverte des points rationnels" at Institut Henri Poincar{\'e}. We would like to thank the organisers David Harari, Emmanuel Peyre and Alexei Skorobogatov for bringing us together. It would not have been possible to complete the paper without the support of several institutions and their short term research in teams opportunities. In particular, we benefited from the hospitality of the  the Institute for Advanced Studies in Princeton and its Summer Collaborators program, and the International Centre for Mathematical Sciences in Edinburgh and its Research-in-Groups initiative. Moreover, the team was hosted at ISTA by Tim Browning. To all these institutions, and the associated staff, we owe a huge debt of gratitude.

We would like to thank Martin Widmer for a discussion on his paper  \cite{MR2645051}, pointing out an oversight in a previous version of this paper and informing the authors of the paper \cite{debaene}. The gratitude of the authors is also extended to Jakob Glas for introducing them to the theory of $\calO_F$-lattices. NR would like to thank Christopher Frei for discussions on lattice point counting applications in number fields. In addition, we are grateful to Tim Browning, Bjorn Poonen, Bianca Viray and Joseph Silverman for helpful conversations around this project.

During part of this work, FB was supported by the European Union’s Horizon 2020 research and innovation programme under the Marie Sk\l{}odowska-Curie grant 840684. Over the course of this project, JL has been funded by the European Union's Horizon 2020 research and innovation programme under the Marie Sk\l odowska-Curie grant agreement No.~754411, and UKRI MR/V021362/1. JP was funded by NSF DMS-1902199 and NSF DMS-2152182. NR is funded by FWF project ESP 441-NBL. 

\section*{Notation} As is customary in analytic number theory, $\epsilon$ will be a small positive real parameter that will not be fixed but instead is allowed to shrink as the proof goes on. 

Throughout, $F$, $F'$, $K$ and $L$ will be fields. Usually $F$ will be the base field, $F'$ a general field extension, $K$ a quadratic extension and $L$ a field of definition or splitting field. By a variety over $F$, we will mean a separated scheme of finite type over $F$. We will denote by $V$ any $F$-variety and $X$ will refer specifically to a weak Fano variety of dimension at least 2 which for us are geometrically integral, and smooth and projective over the base field.

If $F$ is a field, we fix once and for all a separable closure $\overline{F}$, and we denote by $\Gamma_F:= \Gal(\overline{F}/F)$ the absolute Galois group of $F$.

If $F$ is a number field we will write $\Omega_F$ for the set of places of $F$, and we frequently use $v$ for a general element. We will write $\Omega_F^\infty$ for the subset of archimedean places. For a finite place $v$ of a number field $F$ dividing a rational prime $p$ we define
\[
|x|_v = |N_{F_v/\mathbb Q_p}(X)|^{1/[F_v\colon \mathbb Q_p]}_p,
\]
hence we get a valuation $|.|_p$ on $\bar{\mathbb Q}_p$, or even on its completion $\mathbb C_p$.
For an extension of number fields $L/F$, we write $\Delta_{L/F}$ for the fundamental relative discriminant, which is an ideal in $F$. Fractional ideals of $\Q$ we might identify with their positive generator.
The Haar measures $\mu_v$ on the $F_v$ are normalised such that $\mu_v(\mathcal O_v) = 1$. For the infinite places we use the Lebesgue measure on $\R$ and twice the Lebesgue measure on $\C$.

We denote by $\bbS$ a finite set of places of $\Q$ which contains the archimedean place, and on occasion we will write $S = \prod_{\substack{p \in \bbS \\ p < \infty}}p$. By $\prod^S_p$, we mean a product of primes which are not contained in $\bbS$. Similarly, $\sum^S_n$ refers to a sum over all positive integers which have no prime factors in $\bbS$.

We use the variable $\xi$ to classify quadratic \'etale extensions $F_\xi$ of the base field $F$. In the case of both quadratic extensions of $\mathbb R$ we define $\delta_{\xi,\infty}=\frac12$ and over $\Q_p$ we define the weighted local factor of the discriminant $\delta_{\xi,p} = \frac{1}{2} \vert \Delta_{\Q_{p,\xi}/\Q_p} \vert_p$. For a choice of local extensions $\Xi = (\xi_v) \in \prod_{v \in \bbS} \Q_v^\times/\Q_v^{\times,2}$ we define $\delta_\Xi = \prod_{v\in \bbS} \delta_{\xi,v}$. The collection of all quadratic fields $K/\Q$ for which $K_v/\Q_v$ corresponds to $\xi_v$ for all $v \in \bbS$ is denoted $\mathcal{F}_\Xi$. The set of fields in $\calF_\Xi$ with absolute discriminant bounded by $Y$ is denoted $\calF_\Xi(Y)$. Finally, we write $\calF_\Xi^\circ(Y)$ for those fields in $\calF_\Xi(Y)$ which are linearly disjoint from a known field $L$, usually the splitting field of $\Pic \Xbar$.

As in the introduction, $d$ will always be a squarefree integer and $X_d\subseteq \mathbb P^3_\Q$ the surface defined by $x^2-dy^2=zw$. We refer throughout to the counting functions $N_K(\calU, B)$ and $N_K^*(\calU, B)$ which denote, respectively, the number of $K$-points in the cothin set $\calU$ of height at most $B$ and the number of those points which are pure $K$-points.
Peyre's constant for the variety $V$ over the number field $F$ will be denoted $c_{V,F}$. In general, the dependency on the field is sometimes omitted when $F=\Q$.

We denote by $\Delta$ the diagonal in $X \times X$ which descends to $\Delta^{(2)}$ on $\Sym^2 X$, and pulls back to $\Delta^{[2]}$ on $\Hilb^2 X$. In general, the superscript $^{(2)}$ will be adorned for an object living on $\Sym^2 X$ and the superscript $^{[2]}$ for an object on $\Hilb^2 X$.

\section{Preliminaries}
\label{S:preliminaries}

\subsection{Adelic metrics}

We follow Zhang \cite{MR1311351} in defining the notion of adelic metrics.

\begin{defn}
Let $F$ be a global field and $v$ a place of $F$. Consider a quasi-projective $F$-variety $V$  with a line bundle $(V,\mathcal L)$ together with a collection of norms $\|.\|_v$ on the $\mathbb C_v$-vector spaces $\mathcal L(P_v)$ for all $P_v \in V(\mathbb C_v)$. We say that $\|.\|_v$ is a \emph{$v$-adic metric} if for all open $U \subseteq V$ and every section $s \in \Gamma(U,\mathcal L)_{\mathbb C_v}$ the following two conditions hold.
\begin{itemize}
\item[(i)] The map $\|s(\cdot)\|_v \colon U(\mathbb C_v) \to \mathbb R_{\geq 0}$ is $v$-adically continuous.
\item[(ii)] For any $\sigma \in \Gal(\mathbb C_v/F_v)$, and $P_v \in U(\mathbb C_v)$ we have
\[
\|\sigma(s)(P_v)\|_v = \|s(P_v^\sigma)\|_v.
\]
\end{itemize}
\end{defn}

If $\calL$ is very ample, then for a basis $s_i \in \Gamma(V,\mathcal L)$ the \emph{usual $v$-adic metric} is given for each $P \in U(\C_v)$ by
\[
\|s(P)\|_v = \left(\max_i \left| \frac{s_i(P)}{s(P)}\right|_v\right)^{-1}.
\]

\begin{defn}\label{defn:adelically metrised}
Consider a line bundle $\mathcal L$ together with a collection of $v$-adic metrics $\{\|.\|_v\}_{v \in \Omega_F}$.
\begin{enumerate}
    \item[(a)] If $\calL$ is very ample we say that it is \emph{adelically metrised} if there is a basis $s_i \in \Gamma(V,\mathcal L)$ such that the function    \begin{equation}\label{eq:metrised}
        \left(\max_i \left| \frac{s_i(P)}{s(P)}\right|_v\right)\cdot \|s(P)\|_v
    \end{equation}
    is bounded on $V(\C_v)$ for all $v \in \Omega_F$ and identically $1$ for all but finitely many $v$.
    \item[(b)] A general line bundle $\mathcal L$ is \emph{adelically metrised} if there are two very ample adelically metrised line bundles $(\calL_i,\|.\|_{i,v})$ such that $\calL \cong \calL_1 \otimes \calL_2^{-1}$ and $\|.\|_{v,1}$ is the product of the $v$-adic metrics $\|.\|_v$ and $\|.\|_{2,v}$.
\end{enumerate} 

In addition, if $\|s(P)\|_v$ in (a) is the usual adelic metric defined by a basis $s_i \in \Gamma(V,\mathcal L)$, then we say that $\|.\|_v$ is \emph{induced by a model}. 
\end{defn}

\begin{remark}
    One can show that if \eqref{eq:metrised} is bounded on $V(\C_v)$ for all $v \in \Omega_F$, then it is bounded for any generating set $s_i \in \Gamma(V,\calL)$. Also, it implies that that for a given $P$ and $s \in \Gamma(V,\calL)$ the norm $\|s(P)\|_v=1$ for all but finitely many $v$.

    For proofs of these fact the reader is referred to \cite[\textsection1.1]{lerudulier-thesis}. 
\end{remark}

\begin{remark}
    It is more typical to define adelic metrics following \cite{MR1340296} on $V(F_v)$ instead of $V(\C_v)$; the advantage of our choice in following \cite{MR1311351} is the compatibility of adelic metrics over field extensions. Indeed, we can define an adelic metric on $(V_{F'},\calL_{F'})$ from the one on $(V,\calL)$ for any finite field extension $F'/F$. Let $P_v \in V(\C_v)$ and $\| \cdot \|_v$ the $v$-adic metric on $\calL(P_v)$. Given $v'|v$ a place of $F'$ and $P_{v'} \in V_{F'}(\C_{v'}) = V(\C_v)$, we can simply take $\| \cdot \|_{v'} := \| \cdot \|_v$ as the norm on $\calL(P_{v'})$. Then it is easily seen that this adelic metric is also induced by a model for all but finitely many places.
 \end{remark}

Adelic metrics allow us to define heights.

\begin{defn}
    Let $\calL$ be an adelically metrised line bundle on a quasi-projective variety $V$ over a number field $F$.
    \begin{enumerate}
        \item[(i)] For a finite extension $L/F$ we define the \emph{relative adelic height} on $V(L)$ by
        \[
        H_{\calL,L}(P) = \prod_{v \in \Omega_L} \|s(P)\|_v^{-[L_v\colon\mathbb Q_p]}
        \]
        where $s$ is a local section of $\calL$ which does not vanish at $P$.
        \item[(ii)] The \emph{absolute adelic height} of $P \in V(\overline F)$ is defined by
        \[
        H_{\calL}(P) = H_{\calL,L}(P)^{\frac1{[L \colon \Q]}}
        \]
        where $L/F$ is any finite extension such that $P \in V(L)$.
    \end{enumerate}
\end{defn}

It is well known that the height does not depend on the choice of $s$ and that the absolute height is independent of the particular choice for $L$.

\begin{example}
    In the case that the adelic metric on the line bundle $\mathcal O_{\PP^n}(1)$ consists of the usual metric at all places $v$, we recover the \emph{na\"ive absolute $\calO(1)$-height} of a point $P\in \PP^n(L)$
    \[
    H_{\calO(1)}(P) := \prod_{v \in \Omega_L} \max\{|x_0|_v, \ldots, |x_n|_v\}^{\frac{[L_v \colon \Q_p]}{[L\colon \Q]}}.
    \]
\end{example}

Let $V$ be a quasi-projective variety over a number field $F$. Fix a finite extension $L/F$ and $\mathcal U \subseteq V(L)$, and let $H$ be an absolute anticanonical height function on $V(\bar F)$. For a real parameter $B>1$, we denote the point count 
\[
N_L(\mathcal U,B) := \#\{ P \in \mathcal U \colon H(P) \leq B\}.
\]
We will also require a count of points whose field of definition is precisely $L$,
\[
N_L^*(\mathcal U,B) := \#\{ P \in \mathcal U \colon H(P) \leq B \text{ and } F(P)=L\}.
\]
We reserve the right to drop the subscript if the number fields in question $F$ and $L$ are both just $\Q$. By Northcott's theorem, these two sets are finite if $\mathcal L$ is ample.

\subsection{The Tamagawa measure}\label{ss:tamagawa}

Given a $v$-adic metric $\|.\|_v$ on the anticanonical line bundle of a smooth quasi-projective variety $V/F$, we define the \emph{local Tamagawa measure} $\omega_v$ on $V(F_v)$ as follows.

Let $n = \dim V$. Locally around $P \in V(F_v)$ this measure is constructed in the following way.  Consider an anticanonical top degree form $s=\frac{\partial}{\partial x_1} \wedge \cdots \wedge \frac{\partial}{\partial x_n}$ constructed from local coordinates $x_i$ on an open $U\subseteq V$ around $P$. We define the local measure $\omega_v^{[s]} := \mbox{d}x_1\cdots \mbox{d}x_n$ on $V(F_v)$ as the pullback of $ \mbox{d}\mu(x_1)\cdots \mbox{d}\mu(x_n)$ along $c(F_v) \colon U(F_v) \to \mathbb A_F^n(F_v) = F_v^n$ induced by the chart $c \colon U \to \mathbb A_F^n, P \mapsto (x_i(P))_i$, where $\mu(x_i)$ is the usual Haar measure. The measure $\omega_v^{[s]}$ does depend on the chosen form $s$, but not on its representation. The rescaled measure
\[
\|s\|_v^{[F_v \colon \Q_p]} \omega_v^{[s]}=
\left\| \frac{\partial}{\partial x_1} \wedge \cdots \wedge \frac{\partial}{\partial x_n}\right\|_v^{[F_v \colon \Q_p]}\mbox{d}x_1\cdots \mbox{d}x_n
\]
is independent of the choice of $s$. These local measures glue to give the local Tamagawa measure $\omega_v$ on $V(F_v)$.

We combine the local measures $\omega_v$ to form the \emph{Tamagawa measure} $\omega$ on $V(\Adeles_F)$. The product of $\omega_v(V(F_v))$ does not converge, so we proceed as follows. For a non-archimedean place $v \in \Omega_F$, let $\F_v$ be the associated residue field of characteristic $p$ and let $q_v = \# \F_v$. Let $I_v$ be the inertia group at $v$ for the extension $L/F$ and consider $(\Pic \overline{V})^{I_v}$, the inertia invariant part of the absolute Picard group. The Frobenius morphism $x \mapsto x^p$ on $\F_v$ lifts to a conjugacy class in $D_v/I_v$ and so induces an automorphism on $(\Pic \overline{V})^{I_v}$ which we denote by $\Frob_v$. 
We define
\[
L_{F,v}(s,\Pic \Vbar) := \det\left(1-q_v^{-s} \Frob_v \vert (\Pic \Vbar)^{I_v}\right)^{-1}.
\]
Note that this local factor is invariant under conjugation and hence well defined.
Now for a finite set of places $\mathbb S$, which contains the archimedean places, we introduce
\[
L_{F,\bbS}(s,\Pic \Vbar) = \prod_{v \not\in \bbS} L_{F,v}(s, \Pic \Vbar).
\]
Artin~\cite[Satz 3]{Artin} proved that this Euler product converges for $\Rea(s) > 1$. Define 
\[
\lambda_v =
\begin{cases}
    L_{F,v}(1,\Pic \Vbar) & \text{ if } v \notin \bbS;\\
    1 & \text{ otherwise.}
\end{cases}
\]
Then the product measure
\begin{equation}\label{eq:Tamagawa measure}
\omega = \vert \Delta_F\vert^{-\frac{\dim V}2}\lim_{s \to 1}\left[(s-1)^{\rank(\Pic V)} L_{F,\bbS}(s,\Pic \Vbar)\right] \prod_v \lambda^{-1}_v \omega_v
\end{equation}
converges on $V(\Adeles_F)$ and is independent of $\bbS$.

\subsection{Thin sets}

It has been understood since the genesis of Manin's conjecture that certain accumulating subsets contribute too much to the count and must be removed. The modern formulation, following the counterexamples of Batyrev--Tschinkel, Le Rudulier and Browning--Heath-Brown \cite{MR1401626,lerudulier-thesis,MR4167086}, allows for the removal of a thin set.

\begin{defn}
Let $V$ be a geometrically irreducible, integral, and quasi-projective variety over a field $F$. A subset $\calZ \subseteq V(F)$ is
\begin{itemize}
    \item[(i)] a \emph{thin set of type I} if it is contained in $Z(F)$ for a closed subscheme $Z \subsetneq V$.
    \item[(ii)] a \emph{thin set of type II} if it is a subset of $f(W(F))$ for a generically finite dominant morphism $f \colon W \to V$ of degree $\deg f > 1$ with $W$ geometrically irreducible, integral, and quasi-projective variety over $F$.
    \item[(iii)] a \emph{thin set} if it is contained in a finite union of thin sets of type I and II.
\end{itemize}
\end{defn}

\subsection{The Manin--Peyre conjecture}
\label{S:maninpeyre}
In this section, we present a version of the Manin--Peyre conjecture, which applies to all smooth weak Fano varieties.

\begin{defn}\label{def:weakfano}
We say that a projective, geometrically integral variety $V$ over a number field $F$ is \emph{weak Fano} if it is normal, Gorenstein, and $\omega_V^{\vee}$ is big and nef.    
\end{defn}

While our formulation differs from the usual statement of the Manin--Peyre conjecture \ref{conj:manin}, we will show in Lemma~\ref{lem:rankseq} that these are equivalent. Our version below is particularly useful under base change, as we will see in \textsection\ref{sec:quadconst}.

\begin{conj}[Manin--Peyre]
\label{c:BM}
Let $V$ be a smooth weak Fano variety over a number field $F$ and let $H_F$ be a relative adelic height on the anticanonical line bundle $\omega^\vee_V$ on $V$. Then there exists a cothin subset $\mathcal U \subseteq V(F)$ and a constant $c_{V,F}$ such that 
\[
\#\{ P \in \mathcal U \colon H_F(P) \leq B\} \sim c_{V,F}B(\log B)^{\rho(V)-1} \quad \text{ as } B\to \infty,
\]
where $\rho(V) := \rank( (\Pic \Vbar)^{\Gamma_F})$.
\end{conj}

The order of magnitude was first predicted by Manin and his collaborators~\cite{MR1032922}, \cite{MR974910}; Peyre~\cite{MR1340296} gave an explicit expression for the leading constant (although the~$\beta$ factor appeared later, following the work of Batyrev--Tschinkel~\cite{MR1353919} and of Salberger~\cite{MR1679841}).

We remark that the counting function in the above statement of the Manin--Peyre conjecture uses a relative height. However, we are interested in the counting functions $N_F(\mathcal U,B)$ and $N_F^*(\mathcal U,B)$ which are defined using an absolute height. This is necessary for the heights to remain consistent over finite field extensions, but this choice can lead to potential confusion so the reader is warned to beware.

The explicit form of the constant $c_{V,F} := \alpha(V) \beta(V) \tau(V)$ appearing in Conjecture~\ref{c:BM}  is predicted by Peyre \cite{MR1340296}. Similarly to $\rho(V)$ mentioned above, we introduce a different formulation for $\alpha(V)$, originally defined as a volume in the dual lattice $C^1_{\textup{eff}}(V)^{\vee} \subseteq \left[\Pic V\otimes_{\Z} \R\right]^{\vee}$ by the expression 
\begin{align*}
\alpha(V) := \frac{1}{(\rho(V)-1)!} \int_{C^1_{\textup{eff}}(V)^{\vee}} e^{-\langle K_V, y\rangle} dy,
\end{align*}
where $dy$ is the usual Lebesgue measure. 
We then show in Lemma~\ref{lem:rankseq}(c),(d) that our new formulation of $\alpha$ is equal to the original definition by Peyre.

\begin{conj}[Peyre's constant] 
\label{c:MP}
Using the same setup as above, let $C_{\textup{eff}}^1(\Vbar)$ be the effective cone in $\Pic \Vbar \otimes_{\Z} \R$ and consider the cone in  $ (\Pic \Vbar)^{\Gamma_F} \otimes_{\Z} \R$ given by
\[C_{\textup{eff}}^{1, \Gamma_F} (\Vbar):= C_{\textup{eff}}^1(\Vbar) \cap \left[(\Pic \Vbar)^{\Gamma_F} \otimes_{\Z} \R\right].\]
 Let  $C_{\textup{eff}}^{1, \Gamma_F}(\Vbar)^{\vee}$ be its dual cone in $(\Pic \Vbar)^{\Gamma_F} \otimes_{\Z} \R$, and let $K_\Vbar$ be the canonical class. Then
\[
c_{V,F} = \alpha(V)\beta(V)\tau(V),
\]
where 
\[
\alpha(V) := \frac{1}{(\rho(V)-1)!} \int_{C^{1, \Gamma_F}_{\textup{eff}}(\Vbar)^{\vee}} e^{-\langle K_\Vbar, y\rangle} dy,\]
\[\beta(V) := \#\HH^1(F, \Pic \overline V) \quad \text{ and } \quad  \tau(V) := \omega\bigl(V(\Adeles_F)^{\Br V_F}\bigr)
\] 
and $\omega$ is the Tamagawa measure on $V(\Aff_F)$.
\end{conj}

We conclude with the equivalence of the two formulations of the Manin--Peyre conjecture.

\begin{lemma}\label{lem:rankseq}Let $V$ be a smooth, proper, geometrically integral variety over a number field $F$ with $\Pic \Vbar$ a finitely generated abelian group. Then
\begin{enumerate}
    \item[(a)] $\rank(\Pic V) = \rank((\Pic \Vbar)^{\Gamma_F})$; 
    \item[(b)] $\Pic V \otimes_\Z \R = (\Pic \Vbar)^{\Gamma_F} \otimes_\Z \R$;
    \item[(c)] $C^1_{\textup{eff}}(V) =  C^{1, \Gamma_F}_{\textup{eff}}(\Vbar)$.
    \item[(d)] Under the identification in $(c)$, if $y \in C^1_{\textup{eff}}(V)$ and $y' \in C^{1, \Gamma_F}_{\textup{eff}}(\Vbar)$ correspond to each other, then $\langle K_V, y \rangle = \langle  K_{\Vbar}, y' \rangle$.
    \end{enumerate}
\end{lemma}
   
\begin{proof}
\begin{enumerate}
    \item[(a),(b)] Under the assumptions on $V$ and $F$, the Hochschild--Serre spectral sequence yields the exact sequence
\[
0 \to \Pic V \to (\Pic \Vbar)^{\Gamma_F} \to \Br F.
\]
    It follows that the quotient $(\Pic \Vbar)^{\Gamma_F}/\Pic V $ injects into $\Br F$, which is a torsion group, and we get that $((\Pic \Vbar)^{\Gamma_F}/\Pic V)\otimes_\Z \R  = 0$. We conclude that $\Pic V \otimes_\Z \R \to (\Pic \Vbar)^{\Gamma_F}\otimes_\Z \R $ is an isomorphism of finite-dimensional $\R$-vector spaces, which in turn implies that    
    $\rank(\Pic V) = \rank((\Pic \Vbar)^{\Gamma_F})$.
\item[(c)] We note that, by definition and by (b),
     \[
     C^{1, \Gamma_F}_{\textup{eff}}(\Vbar) := C^1_{\textup{eff}}(\Vbar) \cap ((\Pic \Vbar)^{\Gamma_F} \otimes_\Z \R) = C^1_{\textup{eff}}(\Vbar) \cap (\Pic V \otimes_\Z \R).
     \]
     We claim that $C^1_{\textup{eff}}(\Vbar) \cap (\Pic V \otimes_\Z \R) = C^1_{\textup{eff}}(V)$. 
     First, let $\overline{\calL} \in C^1_{\textup{eff}}(\Vbar) \cap (\Pic V \otimes_\Z \R)$. In particular, $\overline{\calL}$ is effective and there is some $\calL \in \Pic V \otimes_\Z \R$ with $\overline{\calL} = \calL \otimes_F \overline{F}$. Since $V$ is a proper $F$-variety and field extensions are flat we have, by \cite[Prop.\ III.9.3]{Hartshorne}, that the space of global sections is compatible with base change to $\overline{F}$ in the sense that
     \[ \Gamma(V, \calL) \otimes_F \overline{F} = \Gamma(\Vbar, \calL \otimes_F \overline{F}).\]
     This proves that $\calL$ is effective if and only if $\overline{\calL}$ is effective.
\item[(d)] This follows from the identification in (c) and by the fact that for proper schemes the canonical divisor and the intersection pairing behave well under flat base change.\qedhere
\end{enumerate}
\end{proof}

\section{Symmetric squares and restriction of scalars}\label{s:symres}

In this section we will recall the basic properties of the symmetric square of a variety $X$, which will itself be a variety if $X$ is quasi-projective. However, it will be singular along the diagonal $\Delta^{(2)} \subseteq \Sym^2 X$ of dimension $\dim X$. To be able to work with line bundles, we will assume $\dim X \geq 2$, so that all varieties are smooth in codimension $1$.

The goal is to count points on $\Sym^2 X$ which we will accomplish by passing through the restriction of scalars. In the first subsection, we discuss points, line bundles, height functions and Brauer classes on the symmetric square $\Sym^2 X$ and how they relate to the associated objects on $X$. Next, we discuss these same objects on the restriction of scalars. In the third subsection, we state the geometric relationship between the symmetric square and the restriction of scalars, tying together the work of the previous two subsections.

Only in the case that $X$ is a surface is the Hilbert--Chow morphism $\Hilb^2 X \to \Sym^2 X$ a crepant desingularisation \cite[Thm.\ 2.4]{fogarty}. In this case, the point count is predicted by the Manin--Peyre conjecture for $\Hilb^2 X$. In the final subsection, we study the Hilbert scheme of two points on a surface in more detail in preparation for the computation in \textsection\ref{S:Tamagawakening} of the leading term for the point count on $\Sym^2 X$ predicted by the Manin--Peyre conjecture~\ref{conj:manin}.

\subsection{The symmetric square}\label{ss:sym2}

In this section, we are always going to take $X$ to be a smooth, quasi-projective, geometrically integral variety of dimension at least $2$ over a number field $F$.
We state some known and new properties of the symmetric square of $X$, denoted $\Sym^2 X := X^2/S_2$. This is the variety whose points are unordered pairs of points of $X$.

Let us write $[P_1,P_2]$ for the class in $X^2(\overline{F})/S_2$ of $(P_1,P_2) \in X^2(\overline{F})$. The Galois action on $X^2(\overline{F})$ induces one on $X^2(\overline{F})/S_2$. The following properties are well known.  

\begin{enumerate}
\item[(a)] The $F$-points of $\Sym^2X$ are the $\overline{F}$-points which are fixed by the absolute Galois group $\Gal(\overline{F}/F)$. That is:
\[
\Sym^2 X(F) = \{Q :=[P_1,P_2] \in X^2(\overline{F})/S_2 \colon Q \text{ is fixed by} \Gal(\overline{F}/F)\}.
\]
More concretely, the Galois group either fixes each component or it exchanges the two points. This leads to the decomposition
\[
\Sym^2X(F) = \{[P_1, P_2]: P_1, P_2 \in X(F)\} \  \sqcup  \bigsqcup_{[K:F]=2}\{[P,\overline{P}]: P \in X(K), F(P) = K\}.
\]
\item[(b)] The diagonal $\Delta \subseteq X^2$ descends to the diagonal $\Delta^{(2)} \subseteq \Sym^2 X$ parametrising points $[P,P]$ for $P \in X$, and it gives the singular locus of $\Sym^2 X$ (to see this, note that the stabiliser of $S_2$ at these points is non-trivial).
\item[(c)] The quotient $\Sym^2 X$ is Gorenstein and so admits an invertible dualising sheaf $\omega_{\Sym^2 X}$, by applying \cite[Thm.\ 3]{watanabe} to the $S_2$-invariant subrings of local rings of $X^2$.
\end{enumerate}

In what follows, we let $\pi: X^2 \to \Sym^2 X$ be the natural quotient map; since $\pi$ is a finite surjective morphism and $\Sym^2 X$ is normal as the finite quotient of a normal integral quasi-projective variety, see e.g.\ \cite[Ch.\ 0, \S2]{mumford1994geometric}), we have that $\pi$ induces a norm map, which we denote by $N_\pi$, sending line bundles on $X^2$ to line bundles on $\Sym^2 X$, see e.g.\ \cite[\href{https://stacks.math.columbia.edu/tag/0BCX}{Tag 0BCX}]{stacks-project}). We will also denote by $\pi_i: X^2 \to X$ the natural projection onto the $i$-th component, for $i = 1,2$.
\begin{defn}
For a line bundle $\calL$ on $X$, we induce a line bundle on $\Sym^2X$ via
\[
\mathcal L^{(2)} := N_\pi(\pi_1^*\calL).
\]
\end{defn}

We note that this line bundle $\calL^{(2)}$ on $\Sym^2X$ is identical to the one produced by Le Rudulier \cite[p.\ 30]{lerudulier-thesis}. 

\begin{prop}\label{prop:props of symsquare}
Let $X$ be a weak Fano variety of dimension at least $2$ over a number field $F$.
\begin{enumerate}
\item[(a)] The line bundles $\omega^{(2)}_X$ and $\omega_{\Sym^2 X}$ are naturally isomorphic.
\item[(b)] The line bundle $\omega_{\Sym^2 X}^{\vee}$ is big and nef. In particular, $\Sym^2 X$ is weak Fano.
\item[(c)] An adelic metric $\|.\|_v$ on $\calL$ induces a natural adelic metric $\|.\|_v$ on the line bundle $\mathcal L^{(2)}$ on $\Sym^2 X$ defined by
\[
\|s(Q)\|_v = \|\pi^*s(P_1,P_2))\|_v
\]
for $Q=\pi(P_1,P_2)$, where the $v$-adic norm on the right is the induced adelic metric on $\mathcal L^{\boxtimes 2}$ on $X \times X$ obtained by pulling back and taking the product.
\end{enumerate}
\end{prop}

\begin{proof}
\begin{enumerate}
\item[(a)] Since $\dim X \geq 2$, it follows that the diagonal $\Delta \subseteq X^2$ has codimension $\geq 2$ in $X^2$. It follows that $\pi$ is étale in codimension 1 and thus that we have a natural isomorphism $\omega_{X^2} = \pi^\ast\omega_{\Sym^2 X}$. Moreover, we know that $\omega_{X^2} = \pi_1^*\omega_X \otimes \pi_2^*\omega_X$ (\cite{Hartshorne}~Exercise II.8.3(b)), and that $\pi^{\ast}(\omega_X^{(2)}) = \pi^{\ast}(N_{\pi}(\pi^\ast_1 \omega_X)) = \pi^\ast_1 \omega_X \otimes \sigma^\ast \pi^\ast_1 \omega_X = \pi_1^*\omega_X \otimes \pi_2^*\omega_X$, where $\sigma$ is the non-trivial element of $S_2$. Hence,
\( \pi^\ast(\omega_{\Sym^2 X}) = \pi^{\ast}(\omega_X^{(2)}) =  \pi_1^*\omega_X \otimes \pi_2^*\omega_X, \)
and using effective descent to $\Sym^2X$ yields the natural  isomorphism $\omega_{\Sym^2 X} = \omega^{(2)}_X$.
\item[(b)]  To prove nefness, let $C \subseteq \Sym^2 X$ be a curve and consider $\tilde{C} := \pi^{-1}C \subseteq X^2$, then $\pi_* \tilde{C} = 2C$ as $\pi$ is finite in codimension $1$ of degree $2$. Since $\omega^\vee_{X^2} = \pi^{\ast}\omega^\vee_{\Sym^2 X}$, by the projection formula we get
\[
2 \omega^\vee_{\Sym^2 X} \cdot C = \omega^\vee_{\Sym^2 X} \cdot \pi_* \tilde C = \pi^*\omega^\vee_{\Sym^2 X} \cdot \tilde C = \omega^\vee_{X^2} \cdot \tilde C,
\]
and we conclude from the nefness of the anticanonical line bundle on $X^2$.

For nef line bundles, bigness is equivalent to positivity of the volume, see \cite[Thm.\ 2.2.16]{lazarsfeld}. From $\vol \omega^\vee_{X^2} = \vol(\pi^\ast\omega^\vee_{\Sym^2 X}) = \deg \pi \cdot \vol\omega^\vee_{\Sym^2 X}$ by e.g.\ \cite[Lem.\ 3.3.6]{holschbach-thesis} and the bigness of $\omega^\vee_{X^2}$, we conclude that $\omega^\vee_{\Sym^2 X}$ is big.
\item[(c)] The proof for surfaces in \cite[Prop.\ 1.33]{lerudulier-thesis} also applies to the general case. \qedhere
\end{enumerate}
\end{proof}

From the naturally induced adelic metric on $\Sym^2 X$ we obtain a natural height function.

\begin{prop}\label{prop:height on Sym}
    Let $X/F$ be as above and suppose that $X$ is equipped with an adelically metrised line bundle $\mathcal L$. With the respect to the natural adelically metrised line bundle $\mathcal L^{(2)}$ on $\Sym^2 X$ we obtain the absolute height function $H^{(2)}$ on $\Sym^2 X(\bar F)$ which is given by
    \[
    H^{(2)}_{\calL^{(2)}}(Q) := H_{\calL}(P_1)H_{\calL}(P_2)
    \]
    on a point $Q=[P_1,P_2]$.
\end{prop}

We also need to understand how to descend elements of the Brauer group $\Br X$ down to $\Br(\Sym^2 X)$. First we need to understand the following fact about Brauer groups of fields, which we will apply to function fields of geometrically integral varieties.

\begin{prop} \label{prop:coresBr}
    Let $\kappa'/\kappa$ be a Galois extension of fields with Galois group $G$. An element $\sigma \in G$ induces a group homomorphism $\sigma^*  \colon \Br \kappa' \to \Br \kappa'$. The following diagram commutes
\[    \begin{tikzcd}
        \Br \kappa' \arrow[rr, "\sigma^*"] \arrow[dr,"\cores_{\kappa'/\kappa}" swap] & & \Br \kappa' \arrow[dl,"\cores_{\kappa'/\kappa}"]\\
        & \Br \kappa
    \end{tikzcd}\]
\end{prop}

\begin{proof}
    We identify $\Br \kappa$ and the Galois cohomology group $\HH^2(\Gal(\overline{\kappa}/\kappa), \overline{\kappa}^{\times})$ and similarly $\Br \kappa' = \HH^2(\Gal(\overline{\kappa}/\kappa'), \overline{\kappa}^{\times})$. We note that $\cores_{\kappa'/\kappa}\colon \Br \kappa' \to \Br \kappa$ is the norm map on the abelian groups by dimension shifting from the $0$-th cohomology groups $\cores_{\kappa'/\kappa}: \kappa'^{\times} \to \kappa^{\times}$. Since the action of $\sigma^*$ leaves the corestriction map unchanged on the $0$-th cohomology groups, the diagram above commutes.
\end{proof}

\begin{defn}\label{def:cores}
    Consider $\mathcal A \in \Br X$. For a projection $\pi_i \colon X^2 \to X$ we define $\mathcal A^{(2)}:= \cores_{\pi} (\pi_i^* \mathcal A) \in \Br\left(\Sym^2 X\right)$ using the finite morphism $\pi \colon X^2 \to \Sym^2 X$.
\end{defn}

The following proposition can be generalised to arbitrary symmetric powers of quasi-projective geometrically integral varieties over number fields, but we will only require it for symmetric squares.

\begin{prop}\label{prop:BrSym2}
Let $X$ be a variety as above defined over a number field $F$.
\begin{enumerate}
        \item[(a)] The Azumaya class $\mathcal A^{(2)}$ is independent of the chosen projection $\pi_i$.
        \item[(b)] Consider a local point $Q_v \in \Sym^2 X(F_v)$, which has the form
        \[
        \begin{cases}
        [P_{v,1},P_{v,2}] & \text{ for $P_{v,i} \in X(F_v)$, or}\\
        [P_{\mathfrak v},\bar P_{\mathfrak v}] & \text{ for $P \in X(K_{\mathfrak v})$}
        \end{cases}
        \]
        for a quadratic field extension $K_{\mathfrak v}/F_v$ with the unique place $\mathfrak v$. In the respective cases we have
        \[
        \inv_v \calA^{(2)}(Q_v) =
        \begin{cases}
            \inv_v \calA(P_{v,1}) + \inv_v \calA(P_{v,2})\\
            \inv_{\mathfrak v} \calA(P_{\mathfrak v}).
        \end{cases}
        \]
    \end{enumerate}
\end{prop}

\begin{proof}
\begin{enumerate}
    \item[(a)] Let $\sigma: X^2 \to X^2$ denote the map that swaps the two factors, so that $\pi_1 \circ \sigma = \pi_2$. The induced map on the Brauer groups $\sigma^*: \Br(X^2) \to \Br(X^2)$ then satisfies $\sigma^*\pi_1^*(\calA) = (\pi_1 \circ \sigma)^*(\calA) = \pi_2^*(\calA)$.

    Furthermore, $\sigma$ restricts to the trivial automorphism on $\Sym^2 X$, and thus it corresponds to the non-trivial element of the Galois group of the function fields ${\Gal(\kappa(X^2)/\kappa(\Sym^2 X)) \isom \Z/2\Z}$. The latter isomorphism comes from the fact that $\pi:X^2 \to \Sym^2 X$ is generically $2$-to-$1$. Now the smoothness of $X^2$ implies $\Br(X^2) \subseteq \Br \kappa(X^2)$ and noting that $\cores_{\pi} = \cores_{\kappa(X^2) / \kappa(\Sym^2 X)}$, it follows that $\cores_\pi(\pi_1^* \calA) = \cores_\pi(\pi_2^* \calA)$, that is, $\calA^{(2)}$ does not depend on the choice of the projection $\pi_i$ by Proposition~\ref{prop:coresBr}.

    \item[(b)] Let us consider the case $Q_v=[P_{v,1},P_{v,2}]$ first. Note that since the action of $S_2$ on $X^2$ swaps the factors, it follows that for any $\calB \in \Br(X^2)$ we have that 
    \[
    \cores_\pi \calB ([P_{v,1}, P_{v,2}]) = \calB(P_{v,1}, P_{v,2}) + \calB(P_{v,2}, P_{v,1}).
    \]
    Hence for $\calB = \pi_1^* \calA$ we obtain
    \[
    \inv_v \calA^{(2)}(Q_v) = \inv_v \calA(P_{v,1}) + \inv_v \calA(P_{v,2}).
    \]

    For the second case $[P_{\mathfrak v},\bar P_{\mathfrak v}]$ for $P \in X(K_{\mathfrak v})$ we obtain
    \[
    \calA^{(2)}(Q_v) = \cores_{K_{\mathfrak v}/\Q_v} \calA(P_{\mathfrak v}).
    \]
    Since residue maps (as defined by Serre) commute with corestriction, see \cite[Prop.\ 1.4.7]{brauerbook}, they are the negative of those defined by Witt \cite[Thm.\ 1.4.14]{brauerbook}, and invariant maps are defined in terms of residue maps \cite[Defn.\ 13.1.7]{brauerbook}, we obtain
    \[
    \inv_v \calA^{(2)}(Q_v) = \inv_{\mathfrak v} \calA(P_{\mathfrak v}).\qedhere
    \]
\end{enumerate}
\end{proof}

\subsection{Restriction of scalars}

Let $V'$ be a variety over a number field $F'/F$. We record some facts on the restriction of scalars functor $\Res_{F'/F} V' \colon \mathbf{Sch}_{F} \to \mathbf{Set}, S \mapsto V'(S \times_F F')$.

Let $L \subseteq \overline \Q$ be the normal closure of $F'$. We consider all embeddings $j \colon F' \to L$. We will write $V'^{(j)}$ for the $L$-variety obtained by pulling back $V'$ along $j$. For an $F$-automorphism $\sigma$ of $L$ we have a natural isomorphism $V'^{(\sigma \circ j)} \to V'^{(j)}$ over $\sigma^* \colon \Spec L \to \Spec L$. 

\begin{prop}\label{prop:weilrestriction}
\begin{enumerate}
\item[(a)] The morphisms $\sigma^*$ endow $\prod_j V'^{(j)}$ with a descent datum. If $V'$ is quasi-projective over $F'$, then this datum is effective, and hence descends to an $F$-scheme $\Res_{F'/F} V'$ representing this functor.
\item[(b)] If $V'$ is smooth and (quasi-) projective over $F'$, then $\Res_{F'/F} V'$ is smooth and (quasi-) projective over $F$, respectively. 
\item[(c)] The natural bijection $\nu_{V',A} \colon \Res_{F'/F} V'(A) \xrightarrow{\cong} V'(A\otimes_F F')$ for an $F$-algebra $A$ can be realised as follows: an $A$-point on $\Res_{F'/F} V'$ pulls back to an $A\otimes_F F'$-point on $\left(\Res_{F'/F} V'\right)_{F'}$, which gives a point in $V'(A \otimes_F F')$ under the counit  of adjunction
\[
\epsilon_{V'} \colon \left(\Res_{F'/F} V'\right)_{F'} \to V'.
\]
\end{enumerate}
\end{prop}

\begin{proof}
    \begin{enumerate}
    \item[(a)] See e.g.\ \cite[Thm.\ 6.1.5]{JLMMS}.
    \item[(b)] For the quasi-projective part, see e.g.\ \cite[Prop.\ 6.2.10]{JLMMS}; the projective part follows from e.g.\  \cite[\S7.6, Prop.\ 5]{BLRNeron} since if $V'$ is separated and of finite type over a field $F'$, then its restriction of scalars is also separated and of finite type, and hence a variety, over $F$. If $V'$ is projective, then it is proper and quasi-projective, and so its restriction of scalars is also proper and quasi-projective over $F$, and thus projective over $F$. For the smoothness part, see e.g.\ \cite[\S7.6, Prop.\ 5(h)]{BLRNeron}.
    \item[(c)] An $A$-point on $\Res_{F'/F} V'$ corresponds to an $A\otimes_F L$-point on $\prod_j V'^{(j)}$ which respects the effective descent data. Equivalently, we could have given the $A \otimes_F L$-point on one of the factors $V'_L$. This respects the effective descent data for $L/F'$ and hence descends to coincide with the counit $\left(\Res_{F'/F} V'\right)_{F'} \to V'$.\qedhere
    \end{enumerate}
\end{proof}

We now study line bundles and adelic metrics under restriction of scalars. The main goal is to translate an adelic metric on the anticanonical bundle $\omega^\vee_{V'}$ to one on $\omega^\vee_{\Res_{F'/F}V'}$.

First, consider line bundles and their sections. In the notation from above, we see that a line bundle $\calL$ on $V'$ over $F'$ pulls back to line bundles $\calL^{(j)}$ on $V'^{(j)}$ for each $j$. The line bundle $\boxtimes_j \calL^{(j)}$ on $\prod_j V'^{(j)}$ is Galois invariant and descends to a line bundle $\Res \calL$ on $\Res_{F'/F} V'$. Similarly, a section $s \in \Gamma(V',\calL)$ induces sections $s^{(j)} \in \Gamma(V'^{(j)},\calL^{(j)})$, and the section $\boxtimes_j s^{(j)}$ descends to a section $\Res s$ of $\Res \calL$.

The following proposition from \cite{MR3430268} makes specific the relationship between line bundles and sections of an $F'$-variety and its restriction of scalars to a base field $F$.

\begin{prop}\label{prop:properties of Weil res}
Let $V'$ be a smooth quasi-projective $F'$-variety. We use the setup and notation from above.
\begin{enumerate}
    \item[(a)] The map $\Pic V' \to \Pic(\Res_{F'/F} V'), \calL \mapsto \Res \calL$ is an injective homomorphism.
    \item[(b)] The map $\Gamma(V', \calL) \to \Gamma(\Res_{F'/F} V', \Res_{F'/F} V'), s \mapsto \Res s$ is an isomorphism of $F$-vector spaces.
    \item[(c)] There is a canonical morphism $\Res \omega_{V'} \to \omega_{\Res V'}$, and this is an isomorphism. 
\end{enumerate}
\end{prop}

\begin{proof}
\begin{enumerate}
    \item[(a)] This was proven in \cite[\textsection 2.2]{MR3430268}.
    \item[(b)] By flat base change of cohomology, the sections of $\Res \calL$ are the Galois invariant sections of $\boxtimes_j \calL^{(j)}$ on $\prod_j V'^{(j)}$. The sections of the product were computed in \cite[proof of Lem.\ 2.2]{MR3430268} to be $\prod_j \Gamma(V'^{(j)},\calL^{(j)})$. Since $\sigma \in \Gal(L/F)$ acts on the components by the natural $\sigma_* \colon \Gamma(V'^{(j)},\calL^{(j)}) \to \Gamma(V'^{(\sigma \circ j)},\calL^{(j)})$, we conclude that the Galois invariant sections are $\Gamma(V',\calL)$.
    \item[(c)] See \cite[Lem.\ 2.2]{MR3430268}.\qedhere
\end{enumerate}
\end{proof}

Now, we construct the natural adelic metric $\{\|.\|_{v}\}_{v \in \Omega_{F}}$ on  $\Res_{F'/F} V'$ (via the line bundle $\Res_{F'/F} \calL$) from an adelic metric $\{\|.\|_{v'}\}_{v' \in \Omega_{F'}}$ on $V'/F'$ using $\calL$. Fix a place $v \in \Omega_F$ and an embedding $F \to \C_v$. The isomorphism in Proposition~\ref{prop:properties of Weil res}(b) identifies vector spaces over $F'$ and $F$ respectively, so we see that it induces a natural isomorphism
\begin{align*}
\Gamma(\Res_{F'/F} V', \Res_{F'/F} \calL')_{\C_v} &= \Gamma(\Res_{F'/F} V', \Res_{F'/F} \calL') \otimes_{F}  \C_v\\
&\cong\Gamma(V', \calL') \otimes_F \C_v\\
&\cong \prod_{\phi \colon F'\to \C_v} \Gamma(V', \calL') \otimes_\phi \C_v = \prod_\phi \Gamma(V'^{(\phi)},\calL^{(\phi)}),
\end{align*}
where we have written $V'^{(\phi)}$ and $\calL^{(\phi)}$ for the pullback of $V'/F'$ and $\calL$ along the $F$-embedding $\phi \colon F' \to \mathbb C_v$. From the adelic metric on $\calL$ we obtain $v'$-adic metrics on each $\calL^{(\phi)}$.
Also, by the properties of the restriction of scalars we have
\[
\Res_{F'/F} V'(\C_v) \isom \prod_\phi V'^{(\phi)}(\C_v).
\]
It is important to note that in both isomorphisms the $[F'\colon F]$ factors are indexed by the $F$-embeddings $\phi: F' \to \C_v$. In the part (b) of the following result we express the natural adelic described in (a) in terms of the places $v'\mid v$, as they correspond to $F$-embeddings of $F' \to \C_v$ up to conjugacy.

\begin{prop}\label{prop:adelic metric on Weil res}
\begin{enumerate} 
\item[(a)] Let $t=(s_\phi) \in \Gamma(\Res_{F'/F} V', \Res_{F'/F} \calL')_{\C_v}$ and $Q=(P_\phi) \in \Res_{F'/F} V'(\C_v)$ viewed as a point in $\Res_{F'/F} V'(\C_v)$. Then
    \[
    \|t(Q)\|_v := \prod_{v'\mid v} \prod_\phi \|s_\phi(P_\phi)\|_{v'}^{\frac{[F'_{v'}\colon F_v]}{[F'\colon F]}}
    \]
    defines an adelic metric on $\Res_{F'/F} \calL$.
    \item[(b)] For a point $Q \in \Res_{F'/F} V'(F_v)$, corresponding to $(P_{v'})_{v' \mid v} \in \prod_{v' \mid v} V'(F'_{v'})=V'(F_v \otimes_F F')$, and $s \in \Gamma(V',\calL)$ we have
    \[
    \|\Res s(Q)\|_v = \prod_{v'\mid v} \|s(P_{v'})\|_{v'} ^{[F'_{v'} \colon F_v]}.
    \]
\end{enumerate}
\end{prop}

\begin{proof}
\begin{enumerate}
    \item[(a)]  A $\sigma \in \Gal(\C_v/F_v)$ induces isomorphisms $V'^{(\sigma\circ \phi)}(\C_v) \to V'^{(\phi)}(\C_v), P \mapsto P^\sigma$ and $\Gamma(V',\calL)^{(\phi)} \to \Gamma(V',\calL)^{(\sigma \circ \phi)}, s \mapsto \sigma(s)$. This gives an isomorphism $\sigma^* \colon \calL(P^\sigma_\phi) \to \calL(P_\phi)$ which maps $s_\phi(P^\sigma_\phi)$ to $\sigma^*s_\phi(P_\phi)$ which have the same $v'$-adic norm defined above. We conclude that
    \[
    \|\sigma(t)(Q)\|_v = \|t(Q^\sigma)\|_v,
    \]
    since $\sigma$ just permutes the factors in the definition.

    Similar to showing that the product of adelic metrics is an adelic metric, we can conclude that the collection of $v$-adic metrics $\{\|.\|_v\}$ on $\Res_{F'/F} V'$ is an adelic metric. Indeed, they define a product metric on $\prod_j V'^{(j)}$ which is Galois invariant.
    \item[(b)] The conjugacy class of embeddings $\phi \colon F' \to \C_v$ which factor through $F'_{v'}$, which has size $[F_{v'} \colon F_v]$, correspond to the place $v'$. Tracing through the isomorphisms we see that $P_{v'} \colon F_{v'} \to V'$ factors these $[F_{v'} \colon F_v]$ embeddings $\phi$. The section $\Res s$ pulls back to $(\phi^* s) \in \prod_\phi \Gamma(V'^{(\phi)},\calL^{(\phi)})$ as $s \in \Gamma(V',\calL)$ is defined over the base field. Hence $[F_{v'} \colon F_v]$ of the factors $\|s_\phi(P_\phi)\|_{v'}$ of the factors indexed by $\phi$ equal $\|s(P_{v'})\|_{v'}$. The result follows.\qedhere
\end{enumerate}
\end{proof}

Now, we compare the heights $\resH$ on $V'$ and $H$ on $\Res_{F'/F} V'$ induced by this adelic metric.

\begin{prop}\label{prop:height on Weil res}
A point $Q \in \Res_{F'/F} V'(\bar F)$, corresponding to $(P_\phi) \in \prod_{\phi \colon F'\to \bar F} V'(\bar F)=V'(\bar F\otimes_F F')$, satisfies
\[
H(Q) = \prod_\phi \resH(P_\phi).
\]
In particular, for $Q \in \Res_{F'/F} V'(F)$ corresponding to $P \in V'(F')$ we have
\[
H(Q) = \resH(P)^{[F'\colon F]}.
\]
\end{prop}

\begin{proof}
We compute the height on the restriction of scalars using a section $\Res s$ and get
    \begin{align*}
    H(Q) := & \prod_{v \in \Omega_F} \|\Res s(Q)\|_v^{-\frac{[F_v \colon \Q_p]}{[F\colon \Q]}} = \prod_{v \in \Omega_F} \prod_\phi \prod_{v' \mid v} \|s(P_\phi)\|_{v'}^{-\frac{[F'_{v'} \colon \Q_p]}{[F\colon \Q]}}\\
     = & \prod_\phi \prod_{v' \in \Omega_{F'}} \|s(P_\phi)\|_{v'}^{-\frac{[F'_{v'} \colon \Q_p]}{[F\colon \Q]}} = \prod_\phi \resH(P_\phi).
    \end{align*}
    In the case that $Q$ is an $F$-point, we get $P=P_\phi$ for all $\phi$.
\end{proof}

Thus, an adelic metric induces on $V'$ induces a natural one on $\Res_{F'/F} V'$. Since $\prod_{v'\mid v} V'(F'_{v'})$ and $\Res_{F'/F} V'(F_v)$ are canonically in bijection, we can compare the local Tamagawa measures on both sides.

\begin{theorem}\label{thm:measure on restriction of scalars}
    Under the natural identification
    \[
    \nu_{V',v} \colon \Res_{F'/F} V'(F_v) \xrightarrow{\cong} V'(F'_v) = \prod_{v'\mid v} V'(F'_{v'})
    \]
    the image measure of $\omega_v$ along $\nu$ equals $\sqrt{\vert \Delta_{F'/F}\vert_v}^{\dim V'} \prod_{v' \mid v} \omega_{v'}$.
\end{theorem}

\begin{proof}
    Let $\Res s \in \Gamma(\Res_{F'/F} V',\omega^\vee_{\Res V'})$ be a non-vanishing section at $Q \in \Res_{F'/F} V'(F_v)$, and let $s \in \Gamma(V',\omega^\vee_{V'})$ and $(P_v') \in \prod_{v'\mid v} V'(F'_{v'})$ be the corresponding data on $\calL'$ and $V'$. We use $s$ to construct the measures $\omega_{v'}$ and the section $\Res s$ for $\omega_v$. From \ref{prop:adelic metric on Weil res}(b) it is enough to prove that the image of $\omega_v$ equals
    \[
    \sqrt{\vert \Delta_{F'/F}\vert_v}^{\dim V'} \prod_{v' \mid v} \omega^{[s]}_{v'}.
    \]
    We will do so by proving that under
    \[
    \nu \colon 
    \Res_{F'_{v'}/F_v} V_{F'_{v'}}'(F_v) \xrightarrow{\cong} V'(F'_{v'})
    \]
    we have
    \begin{equation}\label{eq:local pullback measure}
    \sqrt{\vert \Delta_{F'_{v'}/F_v}\vert_{v}}^{\dim V'} \omega_{v'}^{[s]} = \nu_*\omega_{v}^{[\Res s]}.
    \end{equation}
    We prove this statement in several steps.

    First we establish \eqref{eq:local pullback measure} for $V'=\mathbb A^1_{F'}=\Spec F'[y]$ and $s=dy$. Let us write $r=[F'_{v'} \colon F_v]$ and fix an integral basis $e_i$ for $\mathcal O_{v'}$ over $\mathcal O_v$. In this way we can identify $\Aff_{F_v}^r = \Res \Aff_{F'_{v'}}^1$, such that the composition
    \begin{equation}\label{eq:Weil restriction affine line}
    F_v^r = \Aff^r(F_v) = \Res_{F'_{v'}/F_v} \Aff^1(F_v)  \xrightarrow{\cong} \Aff^1(F'_{v'})=F'_{v'}
    \end{equation}
    is given by $(y_i) \mapsto \sum y_i e_i$. On the codomain the local Tamagawa measure equals $d\mu(y)$, as it is induced by $s= dy$.
    To determine the measure on the domain, we will compute $\Res s$ in the new coordinates on $\Res_{F'_{v'}/F_v} \Aff_{F'_{v'}}^1$. Following the construction of the restriction of scalars on schemes and sections we get $\Aff^r = \Res \Aff^1 \leftarrow \prod_{j} \Aff^{1(j)} \to \Aff^1$.
    The section $s=dy$ produces the section $\prod_j s^{(j)} = \prod_j dy^{(j)}$, which then descends to
    \[
    \bigwedge_{j} \sum_i e_i^{(j)} dx_i  = \det(e_i^{(j)})_{i,j} \bigwedge_i dx_i
    \]
    since $y^{(j)} = \sum_i x_i e_i^{(j)}$. This produces the measure $\vert \Delta_{F'_{v'}/F_v}\vert_v^\frac12 \prod d\mu(x_i)$ on $F_v^r$. To compare the measures along the identification \eqref{eq:Weil restriction affine line} we noticed that because of our choice of integral basis, we have identified $\mathcal O_v^r$ on the left with $\mathcal O_{v'}$ on the right. By our choice of Haar measures this proves the statement in this case.

    Secondly, we show that if \eqref{eq:local pullback measure} holds for two varieties over $F'_{v'}$, then it holds for their product with the expected factor. This proves the statement for $V'=\mathbb A^n_{F'}=\Spec F'[y_i]$ and $s=\bigwedge dy_i$, for any $n$.

    Lastly we establish \eqref{eq:local pullback measure} for general $V'$ and $s$. Locally on an open $U'$ around $P \in V'(F'_{v'})$ we have an \'etale chart $c \colon U' \to \Aff^{\dim V'}$, which induces an \'etale chart $\Res c \colon \Res_{F'_{v'}/F_v} V' \to \Res_{F'_{v'}/F_v} \left(\Aff^{\dim V'}\right)$ around the associated point $Q \in \Res_{F'_{v'}/F_v} V'(F_v)$. We have the commutative square
    \[
    \begin{tikzcd}
    \Res_{F'_{v'}/F_v} U'(F_v) \arrow[r,"\cong"] \arrow[d, "\Res c"] & U'(F'_{v'}) \arrow[d, "c"]\\
    \Res_{F'_{v'}/F_v}\left(\Aff^{\dim V'}\right)(F_v) \arrow[r, "\cong"] & \Aff^{\dim V'}(F'_{v'})        \end{tikzcd}
    \]
    The measures on the top row are constructed by pulling back the measures on the bottom, which differ by the required discriminantal factor by the previous step. This proves \eqref{eq:local pullback measure} in general, and completes the proof.
\end{proof}

Note that the discriminantal factor in \eqref{eq:Tamagawa measure} could be absorbed into the Tamagawa measure if one would normalise the Haar measures to be self-dual. However, for our applications it will be convenient to separate the discriminant from the measure.

\subsection{Relation between restrictions of scalars and symmetric powers}

The relevance of restriction of scalars is captured in the following proposition. We consider a variety $V/F$, and apply the results of the previous subsection to $V'=V_{F'}$.

\begin{prop}\label{prop:map from res to sym}
    Let $V$ be a quasi-projective $F$-variety and $F'/F$ be an \'etale extension of degree $r$. There is a finite morphism of degree $r!$
    \[
    \eta \colon \Res_{F'/F}V_{F'} \to \Sym^r V,
    \]
    which sends an $A$-point in $\Res_{F'/F}V_{F'}(A)=V(A\otimes_F F')$ to the tuple $V(A\otimes_F \bar F)^r$ under the $r$ embeddings $F' \to \bar F$. The morphism $\eta$ is \'etale on the complement of a closed subscheme of codimension $\dim V$. 
\end{prop}

\begin{proof}
    Let $L/F$ be a normal field extension that contains $F'$. We will construct the morphism over $L$ and descend it back to $F$.

    After base changing to $L$ we have an $L$-morphism
    \[
    \prod_j V'^{(j)} \to \prod V'_L,
    \]
    but this morphism does not respect the effective descent data; on the left, $\Gamma_F$ also permutes the factors while it does not on the right. Hence the composition
    \[
    \prod_j V'^{(j)} \to \prod V'_L \to \Sym^r V'_L
    \]
    does descend to the required morphism $\eta$.

    The morphism $\eta$ base changed to $L$ is \'etale away from the total diagonal, i.e.\ those point on $V_L^r$ for which at least $2$ coordinates coincide, which has codimension $\dim V$. Since \'etale morphism are preserved under Galois descent the result follows.
\end{proof}

The proof here also shows that $\eta$ is the twist of the map $V^r \to \Sym^r V$ with its $S_r$-action, classified by the class of $F'/F$ in $\HH^1(F,S_r)$.

\begin{cor}\label{cor:compatibility Res and Sym}
Let $X$ be a smooth, geometrically integral, quasi-projective variety of dimension $d\geq 2$ over $\Q$, and $K/\Q$ a quadratic extension. The natural morphism $\eta^* \omega_{\Sym^2 X} \to \omega_{\Res_{K/\Q} X_K}$ is an isomorphism.
This isomorphism preserves the adelic metrics on $\Sym^2 X$ and $\Res_{K/\Q} X_K$ induced from the adelic metric on $X/\Q$.
\end{cor}

\begin{proof}
On the open subscheme where $\eta$ is \'etale the natural morphism $\eta^* \omega_{\Sym^2 X} \to \omega_{\Res_{K/\Q} X_K}$ is an isomorphism. Since the complement of this locus has codimension $2$ and line bundles are determined up to such opens on normal schemes, we deduce that $\eta^*$ induces an isomorphism.

Start with a section $s \in \Gamma(X,\calL)$ and write $s_K \in \Gamma(X_K,\calL_K)$ for the corresponding section on the base change. We will first show that the natural section $s^{(2)}$ on $\Sym^2 X$ pulls back to the natural section $\Res s_K$ on $\Res_{K/\Q} X_K$ along $\eta$. Over $K$ the morphism $\eta$ is given by the composition
\[
\prod_j X^{(j)} \to \prod_j X_K \to \Sym^2 X_K,
\]
and by definition $s_K^{(2)}$ pulls back to $\boxtimes_j s_K$ on $\prod_j X_K$, which pulls back to $\boxtimes_j s^{(j)}$ by definition of the twisted product. Since $\boxtimes_j s^{(j)}$ descends to $\Res s_K$ by definition, the result follows.

Now we will compute the pullback adelic metric $\| \Res s_K(Q)\|_v$ at a section $\Res s_K$ on $\Res_{K/\Q} X_K$ at a point $Q \in \Res_{K/\Q} X_K(\C_v)$ which corresponds to a pair $(P_1, P_2) \in X(\C_v)^2$.
The pullback metric of $\Res s_K(Q) = (\eta^*s^{(2)})(Q)$ equals by definition
\begin{align*}
\|s^{(2)}(\eta (Q))\|_v & = \|s^{(2)}([P_1, P_2])\|_v = \|(s \boxtimes s)(P_1, P_2)\|_v\\
& = \|s(P_1)\|_v \|s(P_2)\|_v = \left(\prod_{\mathfrak v \mid v} \|s_K(P_1)\|_{\mathfrak v} \|s_K(P_2)\|_{\mathfrak v}\right)^{\frac{[K_{\mathfrak v} \colon \Q_v]}{[K \colon \Q]}} \!\!\!\!\!\! =: \|\Res s_K(Q_v)\|. \qedhere
\end{align*}
\end{proof}

\begin{remark}
    The existence of the morphism $\eta$ might make it logical to study the arithmetic of $\Sym^r X$ through the $\Res_{F'/F} X_{F'}(F)$ instead of $X(F')$, for all fields $F'/F$ of degree $r$. Indeed, \cite{MR3430268} showed that in reasonable situations all geometric invariants agree. Furthermore, one can show that if $\dim X \geq 2$ the induced morphism $\eta^*$ induces a homomorphism which makes the triangles
    \[
    \begin{tikzcd}
         & & \Pic \Sym^r X \arrow[dd, dashed, "r'\cdot \eta^*"] & & & \Br \Sym^r X \arrow[dd, dashed, "r'\cdot \eta^*"]\\
        \Pic X \arrow[drr, "\calL\mapsto \Res\calL_{F'}", sloped] \arrow[urr,"\calL\mapsto \calL^{(r)}", sloped] & & & \Br X \arrow[urr, "\calA\mapsto \calA^{(r)}", sloped] \arrow[drr, "\calA\mapsto \Res\calA_{F'}", sloped] & &\\
         & & \Pic \Res_{F'/F} X_{F'} & & & \Br \Res_{F'/F} X_{F'}
    \end{tikzcd}
    \]
    commute. Here $r'=r!/\tilde r$, where $\tilde r$ is the degree of the smallest normal extension over $F'$.
    
    In the case $r=2$ these vertical morphisms are simply $\eta^*$. This allows us to compute $p$-adic integrals on $\Sym^2 X$ on the $\Res_{K/F} X_F$, for all quadratic fields. For $r>2$ the situation is more complicated.

    We also remark that only in the case that $\dim X=2$ a general resolution of singularities is known, namely the Hilbert scheme of $r$ points $\Hilb^r X$. However, by the discussion above this resolution will only be crepant in the case $r=2$, and hence for the symmetric square of a surface we can interpret the point count on $\Sym^2 X$ within Manin's conjecture on $\Hilb^2 X$.
\end{remark}

\subsection{The Hilbert scheme of two points on a surface}\label{ss:hilb2}

In the case that $X$ is a surface, we can describe a crepant desingularisation of $\Sym^2 X$, namely the Hilbert scheme of two points $\Hilb^2 X$. It is well known that $\Hilb^2 X$ can be obtained from the blowup $\epsilon_X \colon \Hilb^2 X \to \Sym^2 X$ in the diagonal $\Delta^{(2)}$. Let us write $\Delta^{[2]}$ for the exceptional divisor on $\Hilb^2 X$. We could equivalently have constructed the Hilbert scheme by first blowing up and then taking the quotient. That is, the following diagram commutes
\[
\begin{tikzcd}
\Bl_\Delta X^2 \arrow[r, "\tilde \epsilon_X"] \arrow[d, "\tilde \pi"] & X^2 \arrow[d, "\pi"]\\
\Hilb^2 X \arrow[r,, "\epsilon_X"] & \Sym^2 X    
\end{tikzcd}
\]
where the vertical maps $\tilde \pi$ and $\pi$ are $S_2$-quotients, the horizontal maps $\tilde \epsilon_X$ and $\epsilon_X$ are blowups in respectively $\Delta$ and $\Delta^{(2)}$.

\begin{prop}\label{prop:hilbbasics}
Let $X$ be a smooth surface over $F$.
    \begin{enumerate}
        \item[(a)] The Hilbert--Chow morphism $\epsilon_X$ is a crepant resolution of singularities, that is $\epsilon_X^* \omega_{\Sym^2 X}$ is canonically isomorphic to $\omega_{\Hilb^2 X}$.
        \item[(b)] \label{hilbbasics2}
        The map $\Pic X \to \Pic(\Hilb^2  X)$, $\calL \mapsto \calL^{[2]}:= \epsilon_X^* \calL^{(2)}$ is a homomorphism of groups, and induces an isomorphism
        \[
        \Pic(\Hilb^2 X) \cong \Pic X \oplus \mathbb Z\Delta^{[2]}.
        \]
        \item[(c)] Assume $X$ is weak Fano and $\Hilb^2 X(F) \ne \emptyset$. The morphism $\Br X \to \Br(\Hilb^2 X)$, $\calA \mapsto \calA^{[2]}:= \epsilon_X^* \calA^{(2)}$ is an isomorphism.
    \end{enumerate}
\end{prop}

\begin{proof}
    \begin{enumerate}
        \item[(a)] See \cite[Prop.\ 3.3]{lerudulier-thesis}.
        \item[(b)] This follows from tracing through the maps given in \cite[\textsection 3.1.1]{lerudulier-thesis} with the conclusion being in Proposition 3.1 \textit{loc.\ cit.}. In particular, we note that the map
        \[
        \rho^*\colon \Pic(\Hilb^2 X) \to \Pic(\widetilde{X}^2)^{S_2},
        \]
        which is described by pulling back an $S_2$-invariant line bundle on $\widetilde{X^2}$ to $\Hilb^2 X = \widetilde{X^2/S_2}$
 coincides with taking the norm of a line bundle $\calL \in \Pic X$ and pulling back; that is, $\ve_X^*N_{\pi}(\calL) = \rho^*(\calL^{\otimes m})$.
        \item[(c)]
Defining the usual filtration for any variety over a field $F$
\[
\Br_0 V \subseteq \Br_1 V \subseteq \Br V
\]
with $\Br_0 V := \im[\Br F \to \Br V]$ and $\Br_1 V := \ker[\Br V \to \Br \Vbar]$, we begin by proving that
\[
\Br_1 X/\Br_0 X \isom \Br_1 \Hilb^2X/\Br_0 \Hilb^2X.
\]
By \cite[Rmk.\ 5.4.3(3)]{brauerbook}, the above isomorphism is equivalent to showing that
\[
\HH^1(F, \Pic \Xbar) \isom \HH^1(F, \Pic(\Hilb^2\Xbar)).
\]
This isomorphism is clear from part (b) since
\[
\HH^1(F, \Pic(\Hilb^2\Xbar)) \isom \HH^1(F, \Pic \Xbar\oplus \Z) \isom \HH^1(F, \Pic \Xbar).
\]

Next we will show that the transcendental Brauer groups of $X$ and $\Hilb^2X$ are trivial, i.e. $\Br_1 X = \Br X$ and $\Br_1(\Hilb^2X) = \Br(\Hilb^2X)$.

The first equality follows since $X$ being a weak Fano surface implies that $X$ is rational. Since $\Br$ is a birational invariant for proper, smooth varieties over a field of characteristic $0$, we have that $\Br \Xbar = \Br \PP^2_{\overline{F}} = 0$.

So it remains to show that $\Br(\Hilb^2 \Xbar) = 0$. By the exact sequence (\cite[Prop.\ 5.4.8]{brauerbook}) we have 
\[
\coker(\Br(\Hilb^2\Xbar)_{\divv}^{\Gamma} \to \Br(\Hilb^2\Xbar)^{\Gamma}) = \ker(\HH^2(F, \Pic(\Hilb^2\Xbar)) \to \HH^2(F, \NS(\Hilb^2\Xbar)/\tors)).
\]
As $\Xbar$ is rational, we have that $\Pic(\Hilb^2\Xbar) = \Pic \Xbar \oplus \Z$ is torsion-free, so $\Pic(\Hilb^2\Xbar) = \NS(\Hilb^2\Xbar)$. Hence the kernel is trivial, and we obtain the truncated exact sequence
\[
\Br(\Hilb^2\Xbar)_{\divv}^{\Gamma} \to \Br(\Hilb^2\Xbar)^{\Gamma} \to 0.
\]

Finally, we claim that $\Br(\Hilb^2\Xbar)$ is a $2$-group. Recall the diagram 
\[
    \begin{tikzcd}
        \Bl_\Delta X^2 \arrow[d] \arrow[r] & X^2 \arrow[d]\\
        \Hilb^2 X \arrow[r] & \Sym^2 X
    \end{tikzcd}
    \]
    By \cite[Prop.\ 11.2.1]{brauerbook} $\Br(\Xbar^2)$ is generated by $\Br \Xbar  = \Br(\PP^2_{\overline F}) = 0$, so $\Br(\Xbar^2) = 0$, and by the birational invariance of the Brauer groups we have $\Br(\Bl_{\Delta}\Xbar^2) = 0$. Since $\Bl_\Delta{\Xbar^2} \to \Hilb^2 \Xbar$ is a finite Galois morphism of degree $2$ of smooth varieties, we have that the map
\[
\Br(\Hilb^2 \Xbar)[n] \to \Br(\Bl_\Delta{\Xbar^2})^{S_2}[n] = 0
\]
is an isomorphism for odd $n$, see \cite[Thm.\ 3.8.5]{brauerbook}. Thus, $\Br(\Hilb^2 \Xbar)$ is a $2$-group, as claimed.

Hence $\Br(\Hilb^2\Xbar)_{\divv} = 0$ and $\Br(\Hilb^2\Xbar)^{\Gamma} = 0$ from the above exact sequence. Then since the transcendental part of the Brauer group embeds as 
\[
\Br(\Hilb^2\Xbar)/\Br_1(\Hilb^2\Xbar)\hookrightarrow \Br (\Hilb^2\Xbar)^{\Gamma} = 0,\]
we can conclude that the transcendental part is trivial, i.e.\ $\Br_1 \Hilb^2X = \Br(\Hilb^2X)$. Hence, the induced morphism $\Br X \to \Br \Hilb^2 X$ is an isomorphism on the constant subgroups, and on the quotients $\Br X/\Br_0 X$ and $\Br \Hilb^2X/\Br_0 \Hilb^2X$. By the snake lemma we conclude that $\Br X \to \Br \Hilb^2 X$ is an isomorphism, where we have used that $\Br F \to \Br(\Hilb^2 X)$ is injective due to the presence of a rational point. \qedhere
\end{enumerate}
\end{proof}

Since the elements $\mathcal A^{[2]}$ are pullbacks of the $\mathcal A^{(2)}$ on $\Sym^2 X$ we can use the expressions above to compute invariant maps at a subscheme $Z \subseteq X$ of dimension $0$ and degree $2$.
We can also pullback adelic metrics along $\epsilon_X$ to produce height functions on $\Hilb^2 X$, which will satisfy the Northcott property away from $\Delta^{[2]}$. Since $\Hilb^2 X\setminus \Delta^{[2]} \cong \Sym^2 X\setminus \Delta^{(2)}$, we will be able do many computations on the Hilbert scheme on the symmetric square instead.

\section{The contribution from a single field to quadratic Manin--Peyre}\label{S:PeyreoverK}

We are interested in establishing asymptotics for
\[
N_{\calL^{(2)}}(\calU,B) := \{Q \in \mathcal U \colon H^{(2)}(Q) \leq B\}
\]
for an appropriate cothin set $\calU \subseteq \Sym^2 X(\Q)$ for $X$ a weak Fano variety. We note that weak Fano varieties are rationally connected, see \cite{Zhang+2006+131+142}, and so their geometric Picard groups are free and finitely generated as $\Z$-modules and their Brauer groups modulo constants are finite. 

In this section we will first examine thin sets on symmetric squares and where we expect accumulating subsets with too many points. 
When $X$ is a surface, and thus $\Hilb^2 X$ is in the domain of the Manin--Peyre conjecture, one must identify and remove such sets in order to achieve the expected growth of the counting function. Even in higher dimensions it is necessary to identify fields whose contribution to the count can be of a greater order of magnitude than the typical quadratic field and which would thus dominate the count.

Once these thin sets are identified, we will count the number of points in a cothin subset $\mathcal U_K$ of $X(K)$ for (almost) every quadratic extension $K/\Q$, making completely explicit how the predicted leading term in the asymptotic formula for $N_K(\mathcal U_K, B)$ depends on $K$. This will then allow us to sum this term over all quadratic fields containing a point of bounded height.

\subsection{Accumulating subsets of $\Sym^2 X(\Q)$}\label{ss:accumulating subsets}

\begin{lemma}\label{lem:kptsthin}
The following sets are thin sets of $\Sym^2 X(\Q)$:
\begin{enumerate}
\item[(a)] $\{ [P_1,P_2] \colon P_1,P_2 \in X(\Q)\}$;
\item[(b)] $\{ [P,\bar P] \colon P \in X(K)\}$ for any fixed quadratic field $K$;\label{kpts2}
\item[(c)] $\{[P_1,P_2]  \colon P_1 \in Z(\Qbar) \text{ or } P_2 \in Z(\Qbar)\}$ 
for a proper closed subscheme $Z \subseteq V$, and
\item[(d)] $\{[P_1,P_2] \colon P_1 \in f(W(\Qbar)) \text{ or } P_2 \in f(W(\Qbar))\}$ 
for a generically finite morphism  $f \colon W \to X$ of degree larger than $1$.
\end{enumerate}
\end{lemma}

\begin{proof}
\begin{enumerate}
    \item[(a)] This follows from the existence of the finite morphism $X \times X \to \Sym^2 X$.
    \item[(b)] This follows from Proposition \ref{prop:map from res to sym}.
    \item[(c),(d)] Any morphism $X'\to X$ gives a morphism $(X' \times X) \coprod (X \times X') \to X \times X$ which descends to a morphism to $\Sym^2 X$. For $X' = Z$ this gives a proper closed immersion, and for $X'=W$ we obtain a generically finite morphism of degree at least $1$. \qedhere
\end{enumerate}
\end{proof}

Now consider a weak Fano variety $X$ with an anticanonical height function. From Manin's conjecture we see that the sets
\[
\{ [P_1,P_2] \colon P_1,P_2 \in X(\Q)\}
\]
and
\[
\{ [P,\bar P] \colon P \in X(K)\setminus X(\Q)\}
\]
for fields with $\rho(X_K) > \rho(X)+1$ are expected to be accumulating, since $\rho(\Hilb^2 X) = \rho(X)+1$. The pairs of conjugate $K$-points for the quadratic fields $K$ with $\rho(X_K) = \rho(X)+1$ contribute the expected order of growth, but we will see that we will need to remove those too. For other quadratic fields the contribution has a smaller order of magnitude individually than the expected total and we should be able to remove any finite number of them without changing the asymptotics. 

In the higher dimensional setting, any single field with for which the rank of the Picard group jumps will dominate the count completely and it seems reasonable to want to remove the contribution from such fields.
For simplicity we define a finite collection of quadratic fields to exclude at this point.

\begin{defn}
    For a weak Fano variety $X/\Q$, let $L$ be the splitting field of $\Pic \Xbar$ (we remark that $\Pic \Xbar$ is finitely generated and free as a $\Z$-module, and thus $L/\Q$ is a finite Galois extension). Let $\calK := \calK(X)$ be the collection of quadratic fields $K/\Q$ which are linearly disjoint from $L$.
\end{defn}

\begin{lemma}\label{lem:lindisjointL}
    There are finitely many quadratic fields not in $\mathcal K(X)$.
\end{lemma}

\begin{proof}
    A quadratic field is linearly disjoint from the splitting field $L$ if and only if it is not contained in $L$. Since $\Pic \Xbar$ is finitely generated, we have that $L/\Q$ is a finite extension. Since $L/\Q$ is moreover separable, it has finitely many subfields, in particular of degree $2$. 
\end{proof}

\subsection{Peyre constants under quadratic extensions}\label{sec:quadconst}

We will need to understand the expected order of growth in Manin's conjecture \ref{c:BM} for $X_K$ for all but finitely many quadratic fields $K$. Thanks to our equivalent reformulation of the Manin--Peyre invariants (c.f. Lemma \ref{lem:rankseq}), we can easily understand how $\alpha$ and $\beta$ appearing in Peyre's constant as well as $\rho$ behave under finite extensions disjoint from the splitting field of the geometric Picard group.

\begin{prop} \label{peyrequad} Let $X$ be a weak Fano variety over $\Q$. Then, for any $K \in \calK(X)$, we have    $\rho(X_K) = \rho(X)$, $\alpha(X_K) = \alpha(X)$,  and $\beta(X_K) = \beta(X)$.
\end{prop}

\begin{proof}  
 Let $K \in \calK(X)$ and let $L/\Q$ be the splitting field of $\Pic X$, then $L$ and $K$ are linearly disjoint over $\Q$. It follows that $\Gamma_\Q = \Gamma_L \Gamma_K$. Since $\Gamma_L$ acts trivially on $\Pic \Xbar$, it follows that the natural map $(\Pic \Xbar)^{\Gamma_\Q} \to (\Pic \Xbar)^{\Gamma_K}$ is an isomorphism. By Lemma \ref{lem:rankseq}, this shows that $\rho(X_K) = \rho(X)$ and $\alpha(X_K) = \alpha(X)$.
         Since we can identify $\Pic (\Xbar)$ with $\Pic X_L$ and as $\Gamma_L$ is profinite and $\Pic \Xbar$ is torsion-free we have $\HH^1(L, \Pic \Xbar)= 0$. It follows that $\HH^1(\Q, \Pic \Xbar)$ can be identified with $\HH^1(\Gal(L/\Q), \Pic X_L)$. Similarly, we can make the identification $\HH^1(K, \Pic \Xbar) = \HH^1(\Gal(LK/K), \Pic X_{LK})$. But $\Gal(LK/K) = \Gal(L/\Q)$ and $\Pic X_L = \Pic X_{LK}$, since $K \in \calK(X)$ and $L$ is the splitting field for $\Pic X$. Hence, 
    \[
    \HH^1(\Q, \Pic\Xbar) =\HH^1(\Gal(L/\Q), \Pic X_L) = \HH^1(\Gal(LK/K), \Pic X_{LK}) =  \HH^1(K, \Pic\Xbar),
    \]
    meaning that $\beta(X) = \beta(X_K)$.
    \end{proof}

The Tamagawa number of $X_K$ will have a much greater dependence on $K$. While almost all local Tamagawa numbers are easily computed by point counting over finite fields, there are some exceptional places.

\begin{defn}\label{defn:bad places for X}
    Let $X$ be an anticanonically adelically metrised weak Fano variety  over $\Q$. We let $\bbS := \bbS_{X, \Q}$ be the finite set of places of $\Q$ consisting of
    \begin{itemize}
        \item[(i)] the places $\infty$, $2$,
        \item[(ii)] all places of bad reduction of $X$,
        \item[(iii)] all places for which $\|.\|_v$ is not induced by a model, and
        \item[(iv)] all places $v$ for which  $X(\Q_v) = \emptyset$.
    \end{itemize}
    If $K/\Q$ is a finite field extension, we let $\bbS_{K}$ be the places of $K$ above the places in $\bbS$. 
\end{defn}

\begin{remark}
\begin{enumerate}
    \item By spreading out and Lang--Weil, we know that for places $v$ with $q_v \gg_X 1$ there is a model $\calX/ \calO_v$ such that $\calX_{\F_v}/\F_v$ is smooth and has an $\F_v$-point, which lifts to an $\Q_v$-point using Hensel's lemma. Hence, $\bbS$ is indeed a finite set. 
    \item In practice, for the purposes of this paper, we can and often do work with any finite set of places containing $\bbS$.
    \item Let $K/\Q$ be a quadratic extension. It is easily deduced that if $\mathfrak{v} \notin \bbS_K$, then $\mathfrak{v}$ is a finite odd place of $K$ of good reduction for $X_K$, the adelic metric at  $\mathfrak{v}$ is induced by a model, and $X_K(K_{\mathfrak{v}}) \neq \emptyset$.
\end{enumerate}
\end{remark}

To compute the Tamagawa number over a quadratic field, we begin by understanding the Brauer set of $X$ over a quadratic field.

\begin{lemma}\label{lem:agree} Let $K \in \calK(X)$. Then we have the equality
\[
X_K(\Adeles_K)^{\Br X_K} = X(\Adeles_K)^{\Br X}.
\]
\end{lemma}

When $K \in \calK(X)$, we will write $X(\Adeles_K)^{\Br}$ for this set.

\begin{proof}
    We obtain
    \[
    \Br X/\Br_0 X = \HH^1(\Q,\Pic \Xbar) = \HH^1(K,\Pic \Xbar) = \Br X_K/\Br_0 X_K,
    \]
    where the middle equality follows from $K$ being linearly disjoint to the splitting field of $\Pic \Xbar$, and the extremal identifications come from the Hochschild--Serre spectral sequence.
    Now pick elements $\calA_i \in \Br X$, which generate the quotient $\Br X/\Br_0 X$, then their restrictions $\calA_{i,K}$ to $\Br X_K$ generate $\Br X_K/\Br_0 X_K$.
    We get
    \[
    X_K(\Adeles_K)^{\Br X_K} = \left\{ (y'_{\mathfrak v}) \in X_K(\Adeles_K):  \sum_{\mathfrak v} \inv_{\mathfrak v} \calA_{i,K}(y'_{\mathfrak v}) = 0\right\}
    \]
    and
    \[
    X(\Adeles_K)^{\Br X} =  \left\{ (y_{\mathfrak v}) \in X(\Adeles_K): \sum_{\mathfrak v} \inv_{\mathfrak v} \calA_i(y_{\mathfrak v}) = 0\right\}.
    \]
    Since $y_{\mathfrak v} \colon \Spec K_{\mathfrak v} \to X$ decomposes as $\Spec K_{\mathfrak v} \xrightarrow{y'_{\mathfrak v}} X_K \to X$ we see that the sets agree.
\end{proof}

\begin{defn} For any $K \in \calK(X)$, we define 
    \[ \biggl(\prod_{\mathfrak v \in \bbS_K} X(K_{v})\biggr)^{\!\text{Br}} := \textrm{pr}_{\bbS_K} \biggl( X(\Adeles_K)^{\Br} \biggr), \]
    where $\textrm{pr}_{\bbS_K} : X(\Adeles_K) \to \prod_{\mathfrak v \in \bbS_K} X(K_{\mathfrak v})$ is the projection map, and we are viewing $X(\Adeles_K)^{\Br}$ as a subset of $X(\Adeles_K)$.
\end{defn}

Note that by the assumption that $X(\Q_v)\ne \emptyset$ for $v\notin\bbS$ we see that the projection $\textrm{pr}_{\bbS_K}$ is surjective for all $K$.

\begin{lemma}\label{lem:bmoplaces}
    For any $K \in \calK(X)$,  we have
    \[
    X(\Adeles_K)^{\Br} = \biggl(\prod_{\mathfrak v \in \bbS_K} X(K_{\mathfrak v})\biggr)^{\Br} \times \prod_{\mathfrak v \not \in \bbS_K} X(K_v).
    \]
\end{lemma}

\begin{proof}
Since $X$ is (geometrically) rationally connected (see \cite{Zhang+2006+131+142}), we have $\Br \Xbar = 0$, $\HH^1(\Xbar, \calO_{\Xbar}) = 0$, and $\Pic \Xbar = \NS \Xbar$ is a finitely generated, torsion-free abelian group. Hence, for any quadratic extension $K/\Q$, we have that $\Br_1 X_K = \Br X_K$ and that $\Br X_K/\Br_0 X_K = \HH^1(K, \Pic \Xbar)$ is finite. Since $\bbS_K$ contains all the places of bad reduction of $X_K$, we can fix an integral model $\calX_K$ proper and smooth over $\Spec(\calO_{K, \bbS_K})$. The proof of \cite[Thm.\ 13.3.15]{brauerbook} then yields the required result. 
\end{proof}

\begin{remark}
    We warn the reader that the notation $\left( \prod_{\mathfrak{v} \in \bbS_K}X(K_{\mathfrak v})\right)^{\Br}$ on the right hand side of the equation in the previous lemma is non-standard. This is merely a label which we have chosen to give to a certain subset of $\prod_{\mathfrak{v} \in \bbS_K} X(K_{\mathfrak v})$. In particular, it is not necessarily the subset of points whose invariants at $\mathfrak v \in \bbS_K$ sum to $0$. 
\end{remark}

Understanding the behaviour of $X(K_v)$ with $v \in \bbS_K$  is crucial in computing the Tamagawa number of $X$. We begin by partitioning the set of quadratic fields by their splitting behaviours at the places over $\bbS$.

\begin{defn}\label{def:Fxi}
    For a place $v \in \Omega_\Q$ we let $\Q_{v,\xi_v}$ be the quadratic \'etale algebra indexed by $\xi_v \in \Q_v^\times/\Q_v^{\times,2}$, that is
    \[
    \Q_{v,\xi_v} =
    \begin{cases}
        \Q_v \times \Q_v & \text{ if } \xi \equiv 1;\\
        \Q_v(\sqrt \xi_v) & \text{ otherwise}.
    \end{cases}
    \]
    For an element $\Xi := (\xi_v)_{v \in \bbS} \in \prod_{v \in \bbS} \Q_v^\times/\Q_v^{\times,2}$, define
    \[
    \mathcal F_\Xi := \{K/\Q \colon [K\colon \Q]=2, \forall v \in \bbS, K_v \cong \Q_{v,\xi_v}\},
    \]
    where $K_v := \prod_{{\mathfrak v} \mid v} K_{\mathfrak v}$, and we let $\mathcal F_\Xi^\circ := \mathcal F_\Xi\cap \mathcal K(X)$. We also denote by either $\mathcal{F}_\Xi$ or $\mathcal{F}_\Xi^\circ$ the associated set of fundamental discriminants, while the set of all fundamental discriminants associated to fields in $\mathcal{K}(X)$ is denoted by $\mathcal F^\circ$.
\end{defn}

Furthermore, letting $\mathfrak v \in \bbS_K$ be the unique place above $v \in \bbS$ in the local field extension $\Q_{v, \xi_v}$ of $\Q_v$, we have the natural idenification
\[
\prod_{\mathfrak v \in \bbS_K} X(K_{\mathfrak v}) = \prod_{v \in \bbS} X(\Q_{v,\xi_v}),
\]
which allows us to compare the $K_{\mathfrak v}$-points of $X$ as we range over $K \in \calF_{\Xi}^{\circ}$.

\begin{prop}\label{prop:independenceatbadprimes}
    Let $\Xi \in \prod_{v \in \bbS}\Q_v^{\times}/\Q_v^{\times 2}$ be a fixed element. The subset
    \[
    \biggl(\prod_{\mathfrak v \in \bbS_K} X(K_{\mathfrak v})\biggr)^{\!\Br} \subseteq \prod_{\mathfrak v \in \bbS_K} X(K_{\mathfrak v})  = \prod_{v \in \bbS} X(\Q_{v,\xi_v})
    \]
    is independent of the choice of $K \in \calF_{\Xi}^{\circ}$, and it is henceforth denoted $\biggl(\prod_{v \in \bbS} X(\Q_{v,\xi_v})\biggr)^{\!\text{Br}}$.
\end{prop}
\begin{proof}
By definition of $\bbS$, we have $X(\Q_v) \neq \emptyset$ for any $v \notin \bbS$, we fix some $x_v \in X(\Q_v)$ for all $v \notin \bbS$.
\newlength{\myww}%
\settowidth{\myww}{ for ${\mathfrak v} \mid v$ }%
Let us start with an adelic point $(y_{\mathfrak v})_{\mathfrak v} \in X(\Adeles_K)$, and consider the adelic point $(y'_{\mathfrak v})_{\mathfrak v}$ constructed as
\[
y'_{\mathfrak v} :=
\begin{cases}
y_{\mathfrak v} & \makebox[\myww][l]{ }\textrm{if ${\mathfrak v} \in \bbS_K$,}\\
x_v & \textrm{ for ${\mathfrak v} \mid v$ if  ${\mathfrak v} \notin \bbS_{K}$.}
\end{cases}\]
For any $\calA \in \Br X$, Lemma \ref{lem:bmoplaces} yields that $\calA$ is constant on $X(K_{\mathfrak v})$ for ${\mathfrak v} \notin \bbS_K$. Hence, $X(\Adeles_K)^{\Br}$ contains $(y_{\mathfrak v})_{\mathfrak v}$ if and only if it contains $(y'_{\mathfrak v})_{\mathfrak v}$.
We conclude that any point $(y_{\mathfrak v})_{{\mathfrak v} \in \bbS_K} \in \prod_{\mathfrak v \in \bbS_K} X(K_{\mathfrak v})$ is the projection of an adelic point, and it lies in $\biggl(\prod_{\mathfrak v \in \bbS_K} X(K_{\mathfrak v})\biggr)^{\!\Br}$ if and only if
\[
\sum_{{\mathfrak v} \in \bbS_K} \inv_{\mathfrak v} \calA(y_{\mathfrak v}) = - \sum_{{\mathfrak v} \not\in \bbS_K} \inv_{\mathfrak v} \calA(x_v) = -[K \colon \Q] \sum_{v \not\in \bbS} \inv_v\calA(x_v).
\]
Since $\prod_{\mathfrak v \in \bbS_K} X(K_{\mathfrak v})$ only depends on $\Xi$ and the latter condition is independent of $K$, the conclusion follows.
\end{proof}

\settowidth{\myww}{$\omega_v \times \omega_v$}

\begin{defn}\label{defn:wXi}
    Let $v$ be a place of $\Q$, and choose $\xi_v \in \Q_v^{\times}/\Q_v^{\times2}$. Define $\omega_v$ to be the measure on $X(\Q_{v,\xi_v})$ given by 
    \[
    \begin{cases}
        \omega_v \times \omega_v \text{ on } X(\Q_{v,\xi_v})=X(\Q_v) \times X(\Q_v) & \text{ if } \xi_v = 1,\\
        \makebox[\myww][l]{$\omega_{\mathfrak v}$} \text{ on } X(\Q_{v,\xi_v}) & \text{ otherwise}.
    \end{cases}
    \]
    where $\mathfrak v$ is the unique place of the field $\Q_{v,\xi_v}$ for $\xi_v \not \neq 1$. 
    We define 
    \[
    \omega_\Xi(X) := \left( \prod_{v \in \bbS} \omega_v \right)\biggl(\prod_{v \in \bbS} X(\Q_{v,\xi_v})\biggr)^{\!\text{Br}}.
    \]
\end{defn}
Now we study the local Tamagawa numbers at primes $p$ away from places in $\bbS_K$. These depend on the splitting of $p$ in $K$, which we encode in the quadratic character associated to $K$.

\begin{defn}\label{defn:local convergence factors and Tamagawa numbers}
    Denote by $\chi_K$ both the Galois representation associated to the quadratic field $K$ and the associated quadratic character. For $p \not \in \bbS$ and $\mathcal X/\Z_p$ a smooth model, define
    \begin{enumerate}
    \item[(a)] the local Tamagawa numbers
    \[
    w_p =  \frac{\#\calX(\F_p)}{p^{\dim X}} \quad \text{ and } \quad w_{p^2} =  \frac{\# \calX(\F_{p^2})}{p^{2\dim X}},
    \]
    \item[(b)] the modified convergence factor
    \begin{align*}
    \bar \lambda_p = & L_{\Q,p}(1,\Pic \Xbar \otimes \chi_K).
    \end{align*}
    Note that this should be compared with (and will occasionally appear alongside) the traditional convergence factors 
    \[
    \lambda_p = L_{\Q,p}(1, \Pic \overline X) \quad \text{ and } \quad \lambda_{\mathfrak p}=L_{K,\mathfrak p}(1, \Pic \overline{X}).
    \]
    \end{enumerate}
\end{defn}

Note that if we write $\lambda_p$ as a rational function in the variable $\tfrac1p$, then $\bar \lambda_p$ is the same function in $\chi_K(p)/p$. Then $\bar \lambda_p = 1$  for ramified primes, and $\bar \lambda_p = \lambda_p$ for split primes.

The following generalises the fact that $\zeta_K(s)=\zeta_\Q(s)L_{\Q}(s,\chi_K)$ for quadratic fields $K$.

\begin{prop}\label{prop:localfactors}
    Let $K/\Q$ be a quadratic extension and let $p$ be a prime number. Then
    \[
    \prod_{\mathfrak p \mid p} L_{K,\mathfrak p}(s,\Pic \Xbar) = L_{\Q,p}(s,\Pic \Xbar) L_{\Q,p}(s,\Pic \Xbar \otimes \chi_K).
    \]
\end{prop}

\begin{proof} We note that since $K/\Q$ is quadratic, if $K \subseteq L$ then $LK = L$ and so $LK/K$ is Galois since $L/\Q$ is Galois by definition. If $K$ is not contained in $L$, then $L$ and $K$ are linearly disjoint over $\Q$ and so $LK/K$ is again Galois. 
    Let $(\rho, V)$ denote the $\Gal(LK/K)$ representation associated to $\Pic X_{LK} = \Pic \overline{X}$. There is the induced $\Gal(LK/\Q)$ representation $(\text{Ind}^{\Gal(LK/K)}_{\Gal(LK/\Q)}\rho, V')$ where $V'$ can be written explicitly as $V \oplus \sigma V$  for $\sigma \in \Gal(LK/\Q)$ representing the non-trivial element of $\Gal(K/\Q)$. By the inductivity property of Artin $L$-functions
    \[
    L_K(s,\rho) = L_\Q\left(s,\text{Ind}^{\Gal(LK/K)}_{\Gal(LK/\Q)}\rho\right),
    \] and so in particular the local factor of the Euler product at any prime $p$ is the same for both $L$-functions.
    By definition, the local factor of $L_K(s,\rho)$ is $\prod_{\mathfrak{p} \mid p} L_{K, \mathfrak{p}}(s, \rho)$. Moreover, $\text{Ind}^{\Gal(LK/K)}_{\Gal(LK/\Q)}\rho$  can be written as $\rho \oplus (\rho \otimes\chi_K)$, with the associated module being $V' = V \oplus (V \otimes\chi_K)$. 
    Therefore,
    \[
    L_{\Q, p}\left(s,\text{Ind}^{\Gal(LK/K)}_{\Gal(LK/\Q)}\rho\right)
    = L_{\Q,p}(s,\rho \oplus (\rho \otimes \chi_K))
    = L_{\Q,p}(s,\rho) L_{\Q,p}(s,\rho \otimes \chi_K),
    \] which gives the desired equality.
\end{proof}

\begin{prop}\label{prop:eulerfactors}
    Let $K \in \mathcal K(X)$ be a quadratic extension and $p \not \in \bbS$.
    We have
    \[
        \prod_{\mathfrak p \mid p} \lambda_{\mathfrak p} =  \lambda_{p} \bar \lambda_p =
        \begin{cases}
        \lambda_p & \textup{ if } p \text{ ramifies in } K;\\
        \lambda^2_p & \textup{ if } p \text{ splits completely in } K;\\
        \lambda_{p} \bar \lambda_p & \textup{ if } p \text{ is inert in } K,
    \end{cases}
    \]
    and
    \[
        \prod_{\mathfrak p \mid p} \omega_{\mathfrak p}(X(K_{\mathfrak p})) =
        \begin{cases}
        w_p & \textup{ if } p \text{ ramifies in } K;\\
        w^2_p & \textup{ if } p \text{ splits completely in } K;\\
        w_{p^2} & \textup{ if } p \text{ is inert in } K.
    \end{cases}
    \]
\end{prop}

\begin{proof}
    The statement concerning the convergence factors $\lambda_{\mathfrak p}$ follows almost immediately from the previous proposition. Indeed, for inert primes there is nothing to prove. For split primes $p$, note that $\chi_K(p) = 1$ and therefore $L_{\Q, p}(s,\rho \otimes \chi_K) = L_{\Q, p}(s,\rho)$, and for ramified primes $p$, we have $\chi_K(p) = 0$ and thus $L_{\Q, p}(s,\rho \otimes \chi_K) = 1$.

    Now suppose $p$ is a rational prime which ramifies in $K$ and that $p\mathcal{O}_K = \mathfrak{p}^2$. Since in this case $K_{\mathfrak p} = \Q_p$, we have
    \[
    \omega_{\mathfrak p}(X(K_{\mathfrak p}))
    =
    \omega_p(X(\Q_p)).
    \]
    Hence by \cite[Lem.\ 2.2.1]{MR1340296}, $\prod_{\mathfrak p \mid p} \omega_{\mathfrak p}(X(K_{\mathfrak p})) = \omega_p(X(\Q_p)) = w_p$. Similarly, if $p$ 
    splits in $K$ then for every $\mathfrak p$ above $p$ 
    \[
    \omega_{\mathfrak p}(X(K_{\mathfrak{p}})) 
    =
    \omega_p(X(\Q_p)) = w_p.
    \]
    Hence $\prod_{\mathfrak{p} \mid p} \omega_{\mathfrak p}(X(K_{\mathfrak{p}})) = w_p^2$. Finally, if $p$ is inert then the residue field associated to $K_{\mathfrak{p}}$ has order $p^2$. Since $p \not \in \bbS$ and thus is of good reduction, \cite[Lem.\ 2.2.1]{MR1340296} implies that  $\omega_{\mathfrak{p}}(X(K_{\mathfrak p})) = w_{p^2}$.
\end{proof}

\begin{thm}\label{thm:Tamagawa of XK}
For $K \in \mathcal K(X)$ the Tamagawa number of $X_K$ is given by
\[
\tau(X_K) = \vert \Delta_K \vert^{-\frac{\dim X_K}{2}} \left( \lim_{s \to 1} (s-1)^{\rho(X)} L_{\Q,\bbS}(s,\Pic \Xbar)\right) L_{\Q,\bbS}(1,\Pic \Xbar \otimes \chi_K) \hspace{3cm}
\]
\[
\hspace{3cm} \times\ \omega_\Xi(X) \prodS_{p}
\begin{cases}
\lambda^{-1}_p w_p & \textup{ if } p \text{ ramifies in } K;\\
\lambda^{-2}_p w^2_p & \textup{ if } p \text{ splits completely in } K;\\
\lambda^{-1}_p \bar \lambda^{-1}_p w_{p^2} & \textup{ if } p \text{ is inert in } K.
\end{cases}
\]
\end{thm}

\begin{proof}
  By the definition of the Tamagawa number
  \begin{align*}
      \tau(X_K) &= \omega(X_K(\Aff_K))^{\Br}\\
      &= \vert \Delta \vert^{-\frac{\dim X_K}{2}} \lim_{s \rightarrow 1^+} [(s-1)^{\rho(X_K)} L_{K, \bbS}(s, \Pic \overline{X_K})] \left[\prod_{\mathfrak v} \lambda^{-1}_{\mathfrak v}\omega_{\mathfrak v} \right]\left(X(\Aff_K)^{\Br}\right).
  \end{align*}
  By Proposition \ref{prop:localfactors}, we see that
  \[ L_{K, \bbS}(s, \Pic \overline{X_K})
  =
  L_{\Q,\bbS}(s,\Pic \Xbar) L_{\Q,\bbS}(s,\Pic \Xbar \otimes \chi_K).\]
  Note that for $K \in \mathcal{K}(X)$, the $L$-function $L_{\Q,\bbS}(s,\Pic \Xbar \otimes \chi_K)$ converges absolutely at $s=1$. By Lemma \ref{lem:bmoplaces}, the influence of the Brauer--Manin obstruction factors through the places in $\bbS$ and so we have
  \[
  \left[\prod_{\mathfrak v} \lambda^{-1}_{\mathfrak v}\omega_{\mathfrak v} \right]\left(X(\Aff_K)^{\Br}\right) 
      =
      \left[\prod_{\mathfrak v \not \in \bbS_K}\lambda_{\mathfrak v}^{-1} \omega_{\mathfrak v}(X(K_{\mathfrak{v}}))\right]
     \left( \prod_{v \in \bbS} \omega_v\right)\left( \prod_{v \in \bbS} X(\Q_{v, \xi_v}) \right)^{\Br}.
  \]
  The factor on the far right is precisely $\omega_\Xi(X)$ by definition. For the remaining product, we apply Proposition \ref{prop:eulerfactors}.
\end{proof}

\section{The Peyre constant for the Hilbert scheme of $2$ points on a surface $X$}
\label{S:Tamagawakening}

If $X$ is a surface, then the crepant desingularisation $\Hilb^2 X$ of $\Sym^2 X$ is a weak Fano variety. Hence, Manin's conjecture applies to the point count on the symmetric square and the result depends on the invariants of the Hilbert scheme. In this section, we study the Peyre constant of the $\Hilb^2X$ in relation to the Peyre constant of the surface $X$.

\subsection{Cohomological invariants}

From the description of the Picard group of $\Hilb^2 X$ in Proposition \ref{prop:hilbbasics} we have the following result.

\begin{prop}\label{prop:rho and beta of Hilb2}
    We have
    \[
    \rho(\Hilb^2 X) =\rho(X) + 1 \quad \text{ and } \quad \beta(\Hilb^2 X) = \beta(X).
    \]
\end{prop}

\begin{proof}
    The isomorphism $\Pic(\Hilb^2 \Xbar) \cong \Pic \Xbar \oplus \mathbb Z\Delta^{[2]}$ of Galois modules proves the first statement. The second statement follows from $\HH^1(\Q,\Z)=0$.
\end{proof}

We note in particular that the equality $\beta(\Hilb^2 X) = \beta(X)$ is induced by the isomorphism
\[
\HH^1(\Q,\Pic \Xbar) \xrightarrow{\cong}   \HH^1(\Q, \Pic(\Hilb^2 \Xbar)) \cong \Br(\Hilb^2 X)/\Br_0(\Hilb^2 X)
\]
which is compatible with the homomorphism
\[
\Br X \to \Br(\Hilb^2 X), \quad \mathcal A \mapsto \mathcal A^{[2]},
\]
from Proposition~\ref{prop:hilbbasics}, and invariant maps of $\mathcal A^{[2]}$ can be computed as in Proposition~\ref{prop:BrSym2}, since the homomorphism factors through $\Br(\Sym^2 X)$.

The effective cone of $\Hilb^2X$, and hence $\alpha(\Hilb^2X)$, are more mysterious even for the case of $X$ being a surface, and were only successfully studied in a few cases \cite{lerudulier-thesis,Bertram2013}. For us there will be relation to a combination of arithmetical and geometrical properties, which are studied in depth in \textsection\ref{S:fringes}. 

\subsection{The Tamagawa number}

In this section, we relate the Tamagawa number of $\Hilb^2X$ to the Tamagawa numbers of the $X_K$ for all quadratic extensions $K/\Q$. For the good places $p \not \in \bbS$ the local Tamagawa number and convergence factor for $\Hilb^2 X$ and are directly described in terms of the invariants of $X$ from Definition~\ref{defn:bad places for X}, as these places play no role in the Brauer--Manin obstruction. We will show, see Proposition~\ref{prop:wptilde} and Proposition~\ref{prop:lambdaptilde}, that the following definition for $\widetilde w_p$ and $\widetilde \lambda_p$ correspond to $\omega_p(\Hilb^2X)$ and the convergence factors for $\Hilb^2X$.

\begin{defn}
For $p \not \in \bbS$ we define
\[
\widetilde{w}_p := \frac{w_p}p +\frac{w_p^2+w_{p^2}}2 \quad \text{ and } \quad \widetilde \lambda_p := \left(1-\frac1p\right)^{-1} \lambda_p 
\]
\end{defn}

When considering the effect of the Brauer--Manin obstruction at the places $v \in \bbS$ we will need to sort the quadratic extensions by splitting type using the $\Xi = (\xi_v)_v \in \prod_{v \in \bbS} \Q_v^\times/\Q_v^{\times,2}$ as in \textsection\ref{sec:quadconst}. We will see that the local factor of $\Hilb^2 X$ at $v\in \bbS$ is intimately connected to the distribution of quadratic extensions $K/\Q$ with splitting type $\Xi$. Therefore, we introduce succinct notation for a local factor involving the local part of the discriminant of such extensions.

\begin{defn}\label{def:deltaxi}
For a finite place $p \in \bbS \setminus \{\infty\}$ and $\xi_p \in \Q_p^\times/\Q_p^{\times,2}$ we define
\[
\delta_{\xi_p,p} = \frac12 |\Delta_{\Q_{p,\xi_p}/\Q_{p}}|_p. 
\]
In the case of $v=\infty$ we define the local factor as $\delta_{\xi_\infty,\infty}=\frac 12$, for both quadratic \'etale extensions of $\R$.

For $\Xi = (\xi_v)_v \in\prod_{v \in \bbS} \Q_v^\times/\Q_v^{\times,2}$ we collect these terms at all bad places in
\[
\delta_\Xi = \prod_{v \in \mathbb S} \delta_{\xi_v,v} \quad \text{ and } \quad \widetilde \delta_\Xi = \delta_\Xi\prod_{p \in \bbS\setminus\{\infty\}} \left(1-\frac1p\right).
\]
\end{defn}

Finally, we combine the above ingredients to complete the description of the Tamagawa measure on $\Hilb^2 X$, which should be compared to our expression of $\tau(X_K)$ in Theorem~\ref{thm:Tamagawa of XK}. 
In that Theorem, the Tamagawa number at bad places captured information about the specific local behaviour of the extension $K/\Q$. On the Hilbert scheme, we are interested in all quadratic points and this is reflected in the Tamagawa number taking the form of a sum over all possible local factors in $\tau(X_K)$ weighted by the local discriminant of the corresponding extension.
The existence of such a description is in no way obvious and is also completely pivotal for the strategy to counting points on $\Hilb^2 X$ outlined in \textsection\ref{sss:framework}.

\begin{thm}\label{thm:Tamagawa of Hilb2}
If $X$ is a smooth weak Fano surface over $\Q$, we have
\[
\tau(\Hilb^2 X)= \lim_{s\to 1}\left[ (s-1)^{\rho(X)}L_{\Q,\bbS}(s,\Pic \Xbar)\right]
\left[\sum_\Xi  \widetilde \delta_\Xi w_\Xi(X)\right] \left[\prodS_p \widetilde\lambda^{-1}_p \widetilde{w}_p\right].
\]
\end{thm}

This result implies the following.

\begin{cor} \label{cor:equivtamagawa}
    The following are equivalent.
    \begin{itemize}
        \item[(i)] $\tau(\Hilb^2 X) > 0$.
        \item[(ii)] $\tau(X_K) >0$ for a quadratic extension $K/\Q$.
        \item[(iii)] $\tau(X_K) >0$ for infinitely many quadratic extensions $K/\Q$.
    \end{itemize}
\end{cor}

\begin{remark}\label{rem:ctconj}
    Colliot-Th{\'e}l{\`e}ne~\cite{CT03} conjectured that, for any rationally connected variety $V$ over a number field $K$, the set of rational points $V(K)$ has dense image in $V(\Adeles_K)^{\Br}$. This conjecture applies to both the rational surface $X$ and the unirational variety $\Hilb^2 X$. So, under this conjecture, the equivalent conditions in Corollary~\ref{cor:equivtamagawa} could be replaced by ``$X(K)$ is non-empty for a quadratic field $K$'', or ``$\Hilb^2X$ has a $\Q$-point''.
\end{remark}

For the proof we first write the Brauer--Manin set of $\Hilb^2 X$ as a disjoint union over all $\Xi$, and then compute the Tamagawa measure of each of the finitely many parts. 

\subsubsection{The Brauer--Manin set}\label{s:BMsetofHilb}

To compute the Tamagawa number, we need to first understand the Brauer--Manin set $\Hilb^2X(\Aff_{\Q})^{\Br} \subseteq \prod_{p} \Hilb^2X(\Q_p)$. When $p \notin \bbS$, there is no constraint imposed on the factor $\Hilb^2X(\Q_p)$ by the Brauer group.

\begin{prop}\label{prop:constev}
\begin{itemize}
    \item[(a)] For $p \not \in \bbS$ and $\mathcal A' \in \Br(\Hilb^2 X)$, the map
    \[
    \inv_p \calA' \colon \Hilb^2 X(\Q_p) \to \Q/\Z
    \]
    is constant.
    \item[(b)] The set $\Hilb^2 X(\Adeles_\Q)^{\!\textup{Br}}$ is a product with $\prod_{p \not \in \bbS} \Hilb^2 X(\Q_p)$ as one of the factors.
\end{itemize}
\end{prop}

\begin{proof}
\begin{itemize}
    \item[(a)] Proposition~\ref{prop:hilbbasics} gives an isomorphism  $\Br X \rightarrow \Br(\Hilb^2 X)$ so let $\mathcal{A} \in \Br X$ be such that $\mathcal{A}^{[2]} =\calA'$. Recall that we have constructed $\calA^{(2)} = \cores_{\pi}(\pi_i^*\calA) \in \Br(\Sym^2X)$ in Definition~\ref{def:cores}, which pulls back to $\calA^{[2]}$ along the Hilbert--Chow morphism. Then by Proposition~\ref{prop:BrSym2}, the evaluation maps of $\mathcal{A}^{[2]}$ on $\Hilb^2 X(\Q_p)$ is expressed in terms of the evaluation maps of $\mathcal{A}$ on $X(K_v)$ for $K/\Q$ some quadratic extension and $v$ a place above $p$.

    Let $\mathcal{X}$ be a $\Z_p$-model for $X_{\Q_p}$. By \cite[Prop. 2.3(i)]{MR2995366}, the image of $\Br X \rightarrow \Br X_{\Q_p}$ is generated by $\Br \Q_p$ and $\Br \mathcal X$. Thus, the image of $\Br X_K \rightarrow \Br X_{K_v}$ is also generated by $\Br \Q_p$ and $\Br \mathcal X$. 
    The evaluation maps are constant on the constant classes in $\Br \Q_p$; The classes in $\Br X_{\Q_p}$ coming from $\Br \mathcal X$ can be pulled back there and since $\Br \mathcal O_{K_v}$ is trivial, these classes have trivial evaluation maps. Thus $\inv_v \mathcal{A}(P)$ is constant on $X(K_v)$ for all $K_v/\Q_p$ with value independent of the specific field $K$.
    \item[(b)] This follows directly from (a).\qedhere
\end{itemize}
\end{proof}

\begin{remark}
    The work of \cite[Cor.\ 3.2]{MR2995366} directly shows that $\inv_{\mathfrak p} \calA$ is constant on $X(K_{\mathfrak p})$ for every $K$ and $\mathfrak p \mid p$ for $p \notin \bbS$. Then using the assumption that $X(\Q_p) \ne \emptyset$ we can conclude that this value is independent of $K$, and hence only depends on $p$. However, the above proof does not rely on this last assumption and is true for a smaller set of exceptional primes $\bbS$.
\end{remark}


As a consequence of the above theorem, we have the following factorisation of $\Hilb^2 X(\Adeles_\Q)^{\text{Br}}$.

\begin{defn}\label{dfn:BM-set of Hilb2}
    Let $\biggl(\prod_{v \in \bbS} \Hilb^2 X(\Q_v)\biggr)^{\!\text{Br}}$ be the subset of $\prod_{v \in \bbS} \Hilb^2 X(\Q_v)$ such that
    \[
    \Hilb^2 X(\Adeles_\Q)^{\text{Br}} = \biggl(\prod_{v \in \bbS} \Hilb^2 X(\Q_v)\biggr)^{\!\text{Br}} \times \prod_{p \not \in \bbS} \Hilb^2 X(\Q_p).
    \]
\end{defn}

\subsubsection{The local Tamagawa numbers of the Hilbert scheme}\label{s:localTamagawaHilb}

To compute the Tamagawa measure of local and adelic sets on the Hilbert scheme we will use the natural morphisms $\Res_{K/\Q} X_K \to \Sym^2 X \leftarrow \Hilb^2 X$, both preserving the anticanonical line bundle and adelic metrics. For the first morphism, constructed in Proposition~\ref{prop:map from res to sym}, this follows from Corollary~\ref{cor:compatibility Res and Sym}, and for Hilbert--Chow morphism this follows from Proposition~\ref{prop:hilbbasics}(a). These morphism allow us to translate local integrals on $\Hilb^2 X$ to $\Res_{K/\Q} X_K$ and in turn back to $X_K$ using Theorem~\ref{thm:measure on restriction of scalars}.

By the factorisation above the Tamagawa number of $\Hilb^2 X$ will be a product of the local Tamagawa number at each $p \not \in S$, and a single factor which will incorporate the effect of the Brauer--Manin obstrution at the bad places $v \in S$. For both cases we will use the following partition of $\Sym^2 X(\Q_v)$ for a single place $v$.

\begin{prop}\label{prop:paritition of local points}
    Consider the morphism $\vartheta_v$ constructed as the composition
    \[
     \vartheta_v \colon \coprod_{\xi \in \Q^\times_v/\Q^{\times,2}_v} X(\Q_{v,\xi}) \xrightarrow{\nu^{-1}_{X,v}} \coprod_{\xi \in \Q^\times_v/\Q^{\times,2}_v}\left[\Res_{\Q_{v,\xi}/\Q_v}X_{\Q_{v,\xi}}\right](\Q_{v}) \xrightarrow{\eta_{X,v}} \Sym^2 X(\Q_v).
    \]
    \begin{enumerate}
        \item[(a)] The morphism $\vartheta_v$ is a surjective degree $2$ local diffeomorphism away from measure zero loci.
        \item[(b)] Recall the measure $\omega_v$ from Definition~\ref{defn:wXi} on $X(\Q_{v,\xi})$. The pushforward of the measure 
        \[
	\sum_{\xi \in \Q^\times_v/\Q^{\times,2}_v} \delta_{\xi,v} \omega_v
        \]
        along $\vartheta_v$ equals the Tamagawa measure $\omega_v^{(2)}$ on $\Sym^2 X$.
    \end{enumerate}
\end{prop}

\begin{proof}
    \begin{enumerate}
        \item[(a)] The map $\nu_{X,v}$ is a bijection and the maps $\eta_{X,v}$ are degree 2 étale maps away from the diagonal according to Proposition~\ref{prop:map from res to sym} so one gets a degree 2 map. The surjectivity does not hold for an individual map $\eta_{X,v}$ but follows from taking the disjoint union over all possible quadratic extensions of $\mathbb Q_v$. 
        \item[(b)] As $\vartheta_v$ is defined as a composition, let us compute the effect of pushing forward along each of the parts. Let us write $W_v$ for the Tamagawa measure on the middle scheme in the definition of $\vartheta_v$. As $X$ is of the dimension $2$ Theorem~\ref{thm:measure on restriction of scalars} says that $\sum_\xi \delta_{\xi,v} \omega_v$ pushes forward to $\frac 12 W_v$. For the final step, we remark that $\prod_\xi \Res_{\Q_{v,\xi}/\Q_v}X_{\Q_{v,\xi}} \to \Sym^2 X_{\Q_v}$ is crepant, see Corollary~\ref{cor:compatibility Res and Sym}, and \'etale away from the diagonal, and hence locally the Tamagawa measures on $\coprod_\xi \left[\Res_{\Q_{v,\xi}/\Q_v}X_{\Q_{v,\xi}}\right](\Q_{v})$ and $\Sym^2 X(\Q_v)$ agree. To compute the image measure, let $U$ be an open subset of $\Sym^2 X(\Q_v)$ contained in the complement of $\Delta$, then $\eta_{X,v}^{-1}(U)=U_1 \coprod U_2$ with $U$ diffeomorphic to $U_i$ for $i\in \{1,2\}$ and $$\int_{U_i} \eta_{X,v}^{\ast}\omega_v^{(2)}=\int_U \omega_v^{(2)}.
        $$
        Therefore, by definition of the image measure
        \[
        \eta_{X,v*} W_v(U) =\sum_i W_v(U_i) = 2 \omega_v(U).
        \]
        Hence the image measure of $\frac 12 W_v$ equals $\omega_v$. \qedhere
    \end{enumerate}
\end{proof}

This yields the following result at a good place.

\begin{prop}\label{prop:wptilde}
    For $p \not \in \bbS$ we have
        \[
        \omega_p(\Hilb^2 X(\mathbb Q_p)) = \widetilde{w}_p.
        \]
\end{prop}

\begin{proof}
	In this case, by Proposition~\ref{prop:paritition of local points} we see that
	\[
	\omega_p(\Hilb^2 X(\mathbb Q_p)) = \sum_\xi \delta_{\xi,p} \omega_p(X(\Q_{p,\xi})) = 2 \frac{w_p}{2p} + \frac{w_{p^2}}{2} + \frac{w^2_p}2,
	\]
	as $\Q_p$ has two ramified quadratic extensions, an inert one and one split extension $\Q_p^2$.
\end{proof}

\begin{rmk}
Equivalently, one could prove that the adelic metric on $\Hilb^2 X$ at $p \not\in \bbS$ is defined by a model, and hence we could count $\F_p$-points on $\Hilb^2 \calX$, which correpond to zero-dimensional degree $2$ subschemes on a $\mathbb Z_p$-model $\calX$ of $X$. These are either pairs of disjoint points in $\calX(\F_p)$, conjugate pairs in $\calX(\F_{p^2})\setminus \calX(\F_p)$ or a point in $\calX(\F_p)$ with one of the $p+1$ tangent directions. This yields the same result.
\end{rmk}

We now describe and compute the measure of the remaining factor
\[
\biggl(\prod_{v \in \bbS} \Hilb^2 X(\Q_v)\biggr)^{\!\text{Br}},
\]
which is more complicated, since there might be non-trivial reciprocity conditions imposed by the Brauer group. We generalise Proposition~\ref{prop:paritition of local points} to include all primes appearing in the Brauer--Manin obstruction.

\begin{prop}\label{prop:tamagawa measure at bad places}
    Consider the morphism
    \[
    \vartheta^{{\normalfont{Br}}} \colon
    \coprod_{\Xi} \biggl(\prod_{v \in \bbS} X(\Q_{v,\xi_v}) \biggr)^{\!\text{Br}} \subseteq \coprod_{\Xi} \biggl(\prod_{v \in \bbS} X(\Q_{v,\xi_v}) \biggr) = \prod_{v \in \bbS} \biggl(\coprod_{\xi_v}  X(\Q_{v,\xi_v})\biggr) \xrightarrow{\prod_v \vartheta_v} \prod_{v \in \bbS} \Sym^2 X(\Q_v),
    \]
    where $\Xi$ runs over all tuples $(\xi_v)_v \in \prod_{v \in \bbS} \Q_v^\times/\Q^{\times,2}_v$.
    \begin{enumerate}
        \item[(a)] The image of $\vartheta^{\text{Br}}$ coincides up to a measure 0 locus with the image of
        \[
        \biggl(\prod_{v \in \bbS} \Hilb^2 X(\Q_v)\biggr)^{\!\text{Br}} \to \prod_{v \in \bbS} \Sym^2 X(\Q_v)
        \]
        induced by the Hilbert--Chow morphism $\epsilon_X \colon \Hilb^2 X \to \Sym^2 X$.
        \item[(b)] We have the following equality
        \[
        \left( \prod_{v \in \bbS} \omega_v^{[2]}\right) \left[\biggl(\prod_{v \in \bbS} \Hilb^2 X(\Q_v)\biggr)^{\!\text{Br}} \right] = \sum_{\Xi} \delta_\Xi w_\Xi(X),
        \]
        where $w_\Xi(X) = \left(\prod_{v \in \bbS} \omega_v\right)\left[\left(\prod_{v \in \bbS} X(\Q_{v,\xi_v})\right)^{\Br}\right]$ as in Definition~\ref{defn:wXi}.
    \end{enumerate}
\end{prop}

\begin{proof}
   \begin{enumerate}
       \item[(a)] We remove here and in what follows the measure zero loci corresponding to the diagonal $\Delta^{(2)} \subseteq \Sym^2 X$ and its preimages.
       
       It also suffices to compute the Brauer--Manin set with respect to the elements $\calA \in \Br X$ for which $\sum_{p \notin \bbS} \inv_p$ is identically zero on $X(\Adeles_\Q)$, since every element of the Brauer group is of this form modulo a constant class. For these classes we see that $\sum_{p \notin \bbS} \inv_p \calA^{[2]}$ is also identically zero on $\Hilb^2 X(\Adeles_\Q)$ by Proposition~\ref{prop:BrSym2}. It is also immediate that $\sum_{\mathfrak p \notin \bbS_K} \inv_{\mathfrak p} \calA$ is identically zero on $X(\Adeles_K)$ for every quadratic field $K/\Q$.

       Let us start from a point $(Q_v)_v \in \prod_{v \in \bbS} \Hilb^2 X(\Q_v)$. Each $Q_v$ corresponds to a $\Q_{v,\xi_v}$-point on $X$ for some $\xi_v$. Let $K$ be a quadratic field such that for all $v \in \bbS$ we have $K_v \cong \Q_{v,\xi}$. Then $Q_v$ corresponds to $(P_{\mathfrak v})_{\mathfrak v \mid v} \in \prod_{\mathfrak v \mid v} X(K_{\mathfrak v})$ and by Proposition~\ref{prop:BrSym2} we have $\inv_v \calA^{[2]}(Q_v) =  \sum_{\mathfrak v \mid v} \inv_{\mathfrak v} \calA(P_{\mathfrak v})$.
       We conclude that for $(P_{\mathfrak v}) \in \prod_{\mathfrak v \in \bbS_K} X(K_{\mathfrak v})$ we have
       \[
       \sum_{v \in \bbS} \inv_v \calA^{[2]}(Q_v) = \sum_{v \in \bbS} \sum_{\mathfrak v \mid v} \inv_{\mathfrak v} \calA(P_{\mathfrak v}).
       \]

       Since $\sum_{\mathfrak p \notin \bbS_K} \inv_{\mathfrak p}$ and $\sum_{p \notin \bbS} \inv_p$ are identically zero on $X(\Adeles_K)$ and $\Hilb^2 X(\Adeles_\Q)$ respectively, we see that
       \[
       (Q_v) \in \biggl(\prod_{v \in \bbS} \Hilb^2 X(\Q_v)\biggr)^{\!\text{Br}}
       \quad \text{ if and only if } \quad 
       (P_{\mathfrak v}) \in \biggl(\prod_{\mathfrak v \in \bbS_K} X(K_{\mathfrak v})\biggr)^{\!\Br}.
       \]
       \item[(b)] Since the Hilbert--Chow morphism is a crepant resolution of singularities with exceptional locus the preimage of the diagonal, 
       $$
       \left( \prod_{v \in \bbS} \omega_v^{[2]}\right) \left[\biggl(\prod_{v \in \bbS} \Hilb^2 X(\Q_v)\biggr)^{\!\text{Br}} \right]
       $$
       is the measure of the image inside $\prod_{v \in \bbS} \Sym^2 X(\Q_v)$. We then use Proposition~\ref{prop:paritition of local points} to conclude. \qedhere
   \end{enumerate} 
\end{proof}

\subsubsection{The proof of Theorem~\ref{thm:Tamagawa of Hilb2}}\label{ss:tamagawahilb2}

There is just one things left, before you can prove the results stated at the beginning of the subsection.

\begin{prop}\label{prop:lambdaptilde}
The quantity $\widetilde \lambda_p$ is the convergence factor for $\Hilb^2X$.
\end{prop}

\begin{proof}
This follows directly from the isomorphism of Galois modules $\Pic( \Hilb^2 X) \cong \Pic X \oplus \Z$ from Proposition~\ref{prop:hilbbasics}(b).
\end{proof}

We can now put everything together and compute the Tamagawa number of $\Hilb^2 X$.

\begin{proof}[Proof of Theorem~\ref{thm:Tamagawa of Hilb2}]
Using the definition of the Tamagawa number and $L_{\Q, \bbS}(s, \Hilb^2 \Xbar)=\zeta_{\bbS}(s) L_{\Q, \bbS}(s, \Pic \Xbar)$ we get
\[
    \tau(\Hilb^2 X) = \lim_{s \rightarrow 1^+} [(s-1)^{\rho(X)+1} L_{\Q, \bbS}(s, \Pic \Xbar)\zeta(s)] 
    \prod_{p \in \bbS} \left( 1 - \frac{1}{p} \right)
    \left( \prod_v \widetilde{\lambda}_v^{-1} \omega^{[2]}_v \right)\biggl[\Hilb^2 X(\Aff_\Q)^{\Br}\biggr].
\]
Since $\lim_{s\rightarrow 1^+}(s-1) \zeta(s) =1$, we may drop this term producing an expression more easily comparable to that in Theorem \ref{thm:Tamagawa of XK}.

By virtue of Proposition \ref{prop:constev} and Definition \ref{dfn:BM-set of Hilb2}, we know that we can write
\[
\left( \prod_v \widetilde{\lambda}_v^{-1} \omega^{[2]}_v \right)\biggl[\Hilb^2 X(\Aff_\Q)^{\Br}\biggr]
=
\left( \prod_{v \in \bbS} \omega_v^{[2]}\right) \left[
\biggl(\prod_{v \in \bbS} \Hilb^2 X(\Q_v)\biggr)^{\!\text{Br}}\right]
\prodS_p \widetilde{\lambda}_p^{-1}\omega_p(\Hilb^2 X(\Q_p)).
\]
Moreover, by Corollary~\ref{prop:wptilde}, 
$
 \omega_p(\Hilb^2 X(\Q_p))
=
\widetilde{w}_p,
$
and by the previous proposition we have
\[
\left( \prod_{v \in \bbS} \omega_v^{[2]}\right)\left[
\biggl(\prod_{v \in \bbS} \Hilb^2 X(\Q_v)\biggr)^{\!\text{Br}}\right]
=
\sum_{\Xi} \delta_\Xi w_\Xi(X).
\]
This last sum becomes $\sum_{\Xi} \widetilde{\delta}_\Xi w_\Xi(X)$ upon incorporating the term $\prod_{p \in \bbS} \left( 1 - \frac{1}{p} \right)$.
\end{proof}

\begin{proof}[Proof of Corollary~\ref{cor:equivtamagawa}]
    By the inverse function theorem, a smooth variety over a number field has a positive Tamagawa measure if and only if its Brauer--Manin set is non-empty.
    Theorem~\ref{thm:Tamagawa of Hilb2} shows that (i) is equivalent to $w_\Xi > 0$ for some $\Xi$. By Theorem~\ref{thm:Tamagawa of XK} we see that $\tau(X_K)>0$ for $K \in \calF_\Xi$ holds precisely if $w_\Xi > 0$.
\end{proof}

\section{A formal sum of quadratic Tamagawa numbers}
\label{S:summingtheconstants}
The aim of the present section is to establish Theorem \ref{thm:introsumconsts} on the sum of the Tamagawa numbers $\tau(X_K)$ over quadratic fields $K$ of bounded discriminant, with $X$ a weak Fano variety of dimension at least $2$. That is, if $\dim X \geq 3$ then we show that the sum 
\[
 \sum  \limits_{K \in \calF^\circ(Y)}\tau(X_K)
\]
converges absolutely. If $X$ is a surface then we obtain
\[
\sum_{K \in \mathcal F^\circ(Y)} \tau(X_K) = \tau(\Hilb^2 X) \log Y + O(1),
\]
with the leading constant matching the expression developed in the previous section.

By Theorem~\ref{thm:Tamagawa of XK} the Tamagawa number $\tau(X_K)$ over a quadratic extension $K/\Q$ is the product of the following three components:
\begin{enumerate}
    \item the discriminant factor $|\Delta_{K/\Q}|^{-\dim X/2}$;
    \item the Haar volume of $X$ over the good primes along with their convergence factors coming from the twisted Artin $L$-function, given by
    \[
\tau_{\bbS}(X_K) := L_{\Q,\bbS}(1,\Pic \Xbar \otimes \chi_K) \prodS_{p}
\begin{cases}
\lambda^{-1}_p w_p & \textup{ if } \chi_\Delta(p)=0;\\
\lambda^{-2}_p w^2_p & \textup{ if } \chi_\Delta(p)=1;\\
\lambda^{-1}_p \bar \lambda^{-1}_p w_{p^2} & \textup{ if } \chi_\Delta(p)=-1;
\end{cases}
\]
    \item the weighted local volume over the bad primes, given by
    \[
    \kappa_\Xi = \lim_{s\to 1}\left[ (s-1)^{\rho(X)}L_{\Q,\bbS}(s,\Pic \Xbar)\right] w_\Xi(X).
    \]
    We note that by Proposition~\ref{prop:independenceatbadprimes}, this factor depends only on the splitting types of $K$ over the primes lying above $\bbS$.
\end{enumerate}

The main challenge of summing the Tamagawa number is summing the factor $\tau_{\bbS}(X_K)$ over quadratic fields with a fixed splitting type $\Xi$. The factor $\kappa_\Xi$ is constant among those fields, and we can work in the discriminant factor using partial summations. In order to accomplish the summing of $\tau_\bbS(X_K)$, we will re-arrange the order of summation to bring the $K$ sum inside. The result will be a sum of the shape
\[
\sum_K \chi_{K}(m)g(\Delta_{K/\Q}),
\]
where $g$ is a multiplicative function. 
Sums of this form are treated in \S\ref{ss:mult}. In order to apply these results to the sum of $\tau_\bbS(X_K)$, we need to express the $L_{\Q,\bbS}(1, \Pic \Xbar \otimes \chi_K)$ factor as a function of $\chi_K$ and $\Delta_{K/\Q}$.
To accomplish this we prove very general on-average truncation results for such $L$-values  in \S\ref{ss:Lfunctions}. The remaining sections are devoted to proving Theorem~\ref{thm:mainthm1} and deducing some relevant and interesting corollaries to the developed theory in this section.

\subsection{Sums of multiplicative functions over fundamental discriminants}\label{ss:mult}

The purpose of this section is to establish Proposition \ref{prop:discsum} which allows us to sum multiplicative functions of the discriminant of quadratic fields. The key ingredient is the following result on mean values of multiplicative functions on squarefree integers.

\begin{lemma}\label{lem:SD}
Let $m \in \Z_{\geq 1}$ and let $g$ be a multiplicative function such that $g(p) = 1 + O(p^{-1})$. Then, for $Y \geq 2$, we have
\[
\sum_{\substack{0< d \leq Y\\ \gcd(d,m)=1}} \mu^2(d) g(d) = Y\prod_{p\mid m} \left( 1 - \frac{1}{p} \right) \prod_{p \nmid m} \left(1-\frac1p\right)\left(1+\frac{g(p)}p\right) + O_\epsilon\left(m^{\epsilon}Y^{\frac{1}{2}+\epsilon}\right).
\]
\end{lemma}

With the error term $O_\epsilon\left( m^\epsilon Y^{\frac{5}{6} + \epsilon}\right)$ this would be a standard consequence of the contour integration method in analytic number theory. We record the proof here since a judicious application of fractional moments bounds allows us to produce an error term that is best possible, which will be important for a later application.

\begin{proof}
We proceed by the Selberg--Delange method. The Dirichlet series associated to the sum is
\[
F(s) := \sum_{\substack{d=1\\ \gcd(d,m)=1}}^{\infty} \frac{\mu^2(d) g(d)}{d^s} = \prod_{p \nmid m} \left( 1 + \frac{g(p)}{p^s} \right),
\] for $\Rea(s) > 1$. The assumption on our function $g$ implies that $F(s) = \zeta(s) G(s)$ for $G$ some function which is holomorphic and absolutely convergent for $\Rea(s) > \frac{1}{2}$. The truncated Perron formula (e.g.\ \cite[Lem.\ 3.19]{MR0882550}) gives
\[
\sum_{\substack{d \leq Y \\ \gcd(d,m)=1}} \mu^2(d) g(d) = \frac{1}{2\pi i} \int_{1+\frac{1}{\log Y} - iT}^{1+\frac{1}{\log Y} +iT} F(s) \frac{Y^s}{s} \mathrm{d}s + O\left(\tau(m)\frac{Y \log Y}{T}\right)
\]
for any $T \geq 1$, which will be chosen at the end of the proof. We now move the line of integration back to the line $\Rea(s) = \frac{1}{2} + \epsilon$, using a rectangular contour with a keyhole around $s=1$. The contribution from the circle around $s=1$ will be the residue of the integrand at the pole, namely $Y G(1)$. It remains to bound the contribution from the rest of the contour.
Along the horizontal strips, applying Weyl's subconvexity bound for $\zeta(s)$, we have the contribution
\[
\frac{1}{2 \pi i} \int_{\frac{1}{2} + \epsilon}^{1+\frac{1}{\log Y}} F(\sigma \pm iT) \frac{Y^{\sigma \pm iT}}{\sigma \pm iT} \mathrm{d}\sigma
\ll
m^{\epsilon}
\frac{1}{T} \int_{\frac{1}{2} + \epsilon}^{1+\frac{1}{\log Y}} T^{\frac{1-\sigma}{3}} Y^{\sigma} \mathrm{d} \sigma
\ll
m^\epsilon
\frac{Y}{T\log(Y/T^{1/3})}.
\]
For the contribution from the vertical strip, we break into dyadic intervals
\begin{align*}
\frac{1}{2\pi i} &\int_{-T}^T F\left(\frac{1}{2} + \epsilon + it\right) \frac{Y^{\frac{1}{2} + \epsilon + it}}{\frac{1}{2}+\epsilon + it} \mathrm{d}t\\
&\leq
\frac{Y^{\frac{1}{2} + \epsilon}}{2 \pi} 
\left(
2\sum_{\ell = 0}^{\log_2 T} 
\int_{T/2^{\ell+1}}^{T/2^{\ell}}
\left \vert F\left( \frac{1}{2} + \epsilon + it\right) \right\vert\frac{\mathrm{d}t}{1 + \vert t \vert} + \int_0^{\frac{1}{2}} \left \vert F\left( \frac{1}{2} + \epsilon + it\right) \right\vert \frac{\mathrm{d}t}{1 + \vert t \vert}\right)\\
&\ll
m^{\epsilon}
Y^{\frac{1}{2} + \epsilon} 
\sum_{\ell = 0}^{\log_2 T} \frac{1}{T/2^{\ell}}
\int_{T/2^{\ell+1}}^{T/2^{\ell}}
\left \vert \zeta\left( \frac{1}{2} + \epsilon + it\right) \right\vert \mathrm{d}t.
\end{align*}
The innermost integral is a small fractional moment of zeta on the critical line which is known to be $\ll \frac{T}{2^{\ell}} \log (T/2^{\ell})^{1/4}$, by work of Heath-Brown~\cite{MR0623671}. Computing the summation over $\ell$ gives the bound $\ll Y^{\frac{1}{2} + \epsilon} m^{\epsilon} T^{\epsilon}$.
The combined error is minimised by taking $T = Y^{\frac{1}{2}}$, from which the lemma follows.
\end{proof}

Recall from Definition \ref{def:Fxi}, that $\mathcal{F}_\Xi$ is the family of quadratic extensions of $\Q$ with splitting type over the primes in $\bbS$ specified by $\Xi=(\xi_v)_v \in\prod_{v \in \bbS} \Q_v^\times/\Q_v^{\times,2}$. For any $K \in \mathcal{F}_\Xi$, we denote by $\Delta_K$ the associated fundamental discriminant.

\begin{prop}\label{prop:discsum} Let $\bbS$ be a finite set of places of $\Q$ and set $S := \prod_{\substack{p \in \bbS \\ p < \infty}}p$.
Let $g$ be a multiplicative function which is $1$ on primes in $\bbS$ and otherwise $g(p) = 1 + O(p^{-1})$. Let $m$ be an integer coprime to $S$. Then we have
\settowidth{\myww}{$\quad \quad +$}
\[
\sum_{K \in \mathcal F_\Xi(Y)} \!\! \chi_{\Delta_K}(m) g(\Delta_K) = \begin{cases}
\makebox[\myww][l]{ } O\left(m^{1/4}\sqrt{\log(4mS)} \sqrt{Y}\right)& \hspace{-1.2cm}  \text{ if $m \ne \square$};\\
\delta_\Xi Y \prod_{p \in \bbS} \left(1-\frac1p\right) \prod^S_{p \mid m} \left(1-\frac1p\right) \prod^S_{p \nmid m} \left(1-\frac1p\right)\left(1+\frac{g(p)}p\right)  & \\
\quad \quad + O_\epsilon\left((mS)^{\epsilon} Y^{\frac{1}{2} + \epsilon}  \right) & \hspace{-1.2cm} \text{ if $m = \square$}.
\end{cases}
\]
Here $\delta_\Xi$ is the constant defined in Definition \ref{def:deltaxi}.
\end{prop}

The case $g\equiv 1$ and $\bbS=\{\infty\}$ was handled previously by Schmidt \cite[Lem.\ 15]{MR1330740}.

\begin{proof}
Denote $S' :=  \prod_{\substack{ p \neq 2 \\ p \in \bbS}}p$.
We will break the proof into cases depending on the splitting types specified by $\Xi$. We have two subcases depending on whether we choose $+1$ or $-1$ at the place $\infty$, let us suppose for now that we have chosen $+1$; the other case has a similar proof replacing $d$ by $-d$. 

We first deal with the case when each quadratic extension $\Q_{v}$ is unramified, so we start here.  We will represent fundamental discriminants by positive integers $d$ with $d\equiv 1 \mod 4$ or $4d$ where $d\equiv 2$ or $3 \mod 4$. The other splitting conditions for each $p \in \bbS$ can be expressed as a congruence condition 
\[
\text{i.e. } \exists H \subseteq \Z/8S'\Z \text{ such that for } K \in \mathcal{F}(Y), \Delta_K \in \mathcal{F}_{\Xi}(Y) \iff \Delta_K \in H \bmod {8S'}.
\]
If $\xi_2$ corresponds to an unramified extension then $\Delta_K$ is equal to a power of 4 times a number which is congruent to 1 mod 8. The only way this can happen is if $\Delta_K \equiv 1 \mod 8$.
At odd primes, the assumption that each extension associated to $\Xi$ is unramified means that each residue class in $H$ is primitive. 

Hence, the sum we are left to evaluate is
\[
\sum_{a \in H} 
\sum_{\substack{0 < d \leq Y\\ d \equiv a \bmod {8S'}}} \mu^2(d) \chi_d(m) g(d).
\]
We will remove the congruence condition with a sum of Dirichlet characters (here it is important that our congruences are primitive), giving
\[
\frac{1}{4\phi(S')} \sum_{a \in H} \sum_{\psi \bmod {8S'}} \overline{\psi(a)} \sum_{0 < d \leq Y} \mu^2(d) \chi_d(m) \psi(d) g(d).
\]
When the product $\chi_d(m) \psi(d)$ is a non-principal character there will be oscillation and the contribution will be negligible. The main contribution comes when the product corresponds to the principal character. On such an occasion, the inner sum above becomes after using Lemma \ref{lem:SD}
\[
\sum_{\substack{0 < d \leq Y\\ \gcd(d, mS)=1}} \mu^2(d) g(d)= Y\prod_{p \mid mS} \left(1 - \frac{1}{p} \right) \prod_{p \nmid mS} \left(1-\frac1p\right)\left(1+\frac{g(p)}p\right) + O_\epsilon\left((mS)^{\epsilon} Y^{\frac{1}{2} + \epsilon}\right).\]
If $m \neq 1$ then, since $m$ is coprime to $S$, the product of characters is only principal when both characters are principal. Otherwise, the product is principal precisely when $\psi$ is the unique character mod $S$ which is induced by the non-principal character mod 4, and by the principal character modulo all other primes in $\bbS$. Either way, there is a single unique $\psi$ mod $S$ which contributes to the main term. Therefore, the total sum in question is equal to
\[
\frac{\#H}{4\phi(S')}Y\prod_{p \mid mS} \left(1 - \frac{1}{p} \right)\prod_{p \nmid mS} \left(1-\frac1p\right)\left(1+\frac{g(p)}p\right) + O_\epsilon\left((mS)^{\epsilon} Y^{\frac{1}{2} + \epsilon}\right).
\]All that remains is to observe that this leading term is equal to $\frac{1}{4}\prod_{p \mid S'} \frac{\#H_p}{p-1}$, where $H_p$ is the image of $H$ under the map $\Z/S\Z \rightarrow \Z/p\Z$. For odd primes $p$, $H_p$ consists either of all the quadratic residues or of the quadratic non-residues modulo $p$, hence we get a factor of $\frac{1}{2}$ at each place in $S$.
In the case, that $\chi_d(m)\psi(d)$ is a non-principal character, say $\sigma$, it has modulus at most $4mS$. Applying the identity $\mu^2(n) = \sum_{d^2 \mid n} \mu(d)$, we have
\[
\sum_{0 < d \leq Y} \mu^2(d) \sigma(d) g(d)
=
\sum_{\ell \leq \sqrt{Y}} \mu(\ell)\sigma(\ell)^2 g(\ell^2) \sum_{d \leq Y/\ell^2} \sigma(d) g(d).
\]
We will write $g= \mathbf{1}\ast h$ for $h(n)$ a multiplicative function which, by virtue of our assumption on $g$, satisfies $h(p) = O\left(\frac{1}{p}\right).$ Then
\begin{align*}
\sum_{0 < d \leq Y} \mu^2(d) \sigma(d) g(d)
&=
\sum_{\ell \leq \sqrt{Y}} \mu(\ell)\sigma(\ell)^2 g(\ell^2) \sum_{ab \leq Y/\ell^2} \sigma(ab) h(a)\\
&=
\sum_{a \leq Y} \sigma(a) h(a) \sum_{\ell \leq \sqrt{Y/a}} \mu(\ell)\sigma(\ell)^2 g(\ell^2) \sum_{b \leq Y/\ell^2a} \sigma(b).
\end{align*}
This inner sum is at most $ \max\{\sqrt{4mS} \log 4mS, Y/\ell^2a\}$ by the P{\'o}lya--Vinogradov inequality. Thus
\[
\sum_{0 < d \leq Y} \mu^2(d) \sigma(d) g(d)
\ll
\sqrt{mS} \log(4mS) \sum_{a \leq \sqrt{mS}\log(4mS)} \vert h(a) \vert \sum_{\ell \leq \sqrt{Y}(mS)^{-1/4}(\log(4mS)^{-1}a^{-1/2}} \mu^2(\ell) \vert g(\ell) \vert\]\[
+ 
Y
\sum_{a <Y} \frac{\vert h(a) \vert}{a} \sum_{\sqrt{\sqrt{mS}\log(4mS)}/a< \ell < \sqrt{Y/a}} \mu^2(\ell) \frac{\vert g(\ell) \vert}{\ell^2} 
\] By Lemma \ref{lem:SD}, this is
\[
\ll_\epsilon
(mS)^{\frac{1}{4} + \epsilon}\sqrt{Y} \sum_{a \leq Y} \frac{\vert h(a) \vert }{\sqrt{a}}.
\]
The inner sum converges due to the conditions on $h$.

Now, we must allow for wild ramification, and thus non-primitive congruences. As before, the splitting conditions can be expressed as congruences, but at odd primes $p$ where the extension corresponding to $\Xi$ is ramified we use a non-primitive congruence modulo $p^2$. Let $H$ be the set of permissible congruences to the modulus $16S^2$. Let $a \in H$ and denote $k = \gcd(a, S')$. Note that this definition is independent of the choice of $a$. We begin at the same starting point as before then remove a factor of $k$ to make all the congruences modulo odd primes primitive. Moreover, if the extension at 2 is ramified then we must remove a factor of 4. This gives
\[
\sum_{a \in H} 
\sum_{\substack{0 < d \leq Y/4^{\delta}\\ d \equiv a \bmod S^2}} \mu^2(d) \chi_d(m) g(d)
=
\chi_{4k}(m) g(4k)
\sum_{b \in \widetilde{H}} \sum_{\substack{0 < d \leq Y/4^{\delta}k \\ d \equiv b \bmod 16S \\ \gcd(d,k) =1}} \chi_d(m) g(d),
\] where $\widetilde{H}$ is a set of primitive congruence classes modulo $S$ capturing the splitting information of $H$ and $\delta$ is either 1 or 0, depending on the splitting condition at 2. For instance, at an odd prime $p$, the extension specified by $\xi_p$ is generated either by $\sqrt{p}$ or by $\sqrt{ap}$ for $a$ a quadratic non-residue mod $p$. The two ramified extensions of $\Q_p$ correspond to congruence classes $ap$ mod $p^2$ where $a$ is either a quadratic residue or quadratic non-residue depending on the desired extension. $\Q_2$ has 7 unique quadratic extensions: 3 generated by pairs of primitive residue classes mod 16 (although it is impossible in our situation for $\Delta_K$ to be congruent to 5 mod 16) and 4 generated by a unique even residue mod 16. 
We will break our sum over $b \in \widetilde{H}$ into two cases corresponding to whether the residue classes specified mod 16 are even or odd, denoted $b \in \widetilde{H}^e$ or $d \in \widetilde{H}^{o}$. If $\xi$ corresponds to the trivial extension at 2 then all the residue classes are odd (namely 1 and 9 mod 16).
If $b \in \widetilde{H}^o$, then the analysis may now proceed as before. Again, when $m$ is non-square we get the same error term. When $m$ is a square, we have $\chi_k(m) = 1$ and since $g$ is defined to be 1 on all primes in $S$ the main term becomes
\[
\frac{1}{8\phi(S)} \sum_{b \in \widetilde{H}^e} \sum_{0 < d \leq Y/4^{\delta}k} \mu^2(d)  g(d).
\]
If only odd primes ramify then we see that the only difference between this and the previous case is the presence of $k$ on the denominator, as explained in the factor $\delta_{\Xi}$. If the splitting condition specified by $\xi_2$ corresponds to the extensions $\Q_2(\sqrt{3})$ or $\Q_2(\sqrt{7})$ then 2 ramifies and so $\delta =1$ and then there are two possible residue classes mod 16 which $b$ may occupy (and $\widetilde{H}^e$ is empty). This corresponds to dividing by precisely the factor expected in the definition of $\delta_{\Xi}$.

If $\xi_2$ corresponds to any of the extensions $\Q_2(\sqrt{2}), \Q_2(\sqrt{6}), \Q_2(\sqrt{10})$ or $\Q_2(\sqrt{14})$, then $\widetilde{H}^o$ is empty. To handle the sum over $b \in \widetilde{H}^e$, we must divide through by 2. Now the splitting condition at 2 corresponds to a choice of primitive residue class mod 8. Hence the main term contribution is
\[
\frac{1}{4\phi(S)} \sum_{b \in \widetilde{H}^e} \sum_{0 < d \leq Y/(2\times4k)} \mu^2(d)  g(d).
\]
Again we see that the factor out front after applying Lemma \ref{lem:SD} will correspond precisely to the definition of $\delta_{\Xi}$, completing the proof.
\end{proof}

\subsection{Sums of twists of $L$-functions}\label{ss:Lfunctions}

We will require a truncation estimate for special values of quadratic twist $L$-functions. Unfortunately, the pointwise bounds that one is able to provide (without assuming GRH) are insufficient. However, the following on-average result will suffice for our application.

\begin{prop}\label{prop:trunc}
Let $\rho$ be a Galois representation which factors through a finite extension $L/\Q$. For a real parameter $Y \geq 1$, recall the notation $\mathcal{F}^\circ(Y)$ for the set of fundamental discriminants $\Delta$ absolutely bounded by $Y$ whose associated quadratic field $K$ is linearly disjoint from $L$.
Let $Z \geq Y^{\delta}$ be a real parameter, for some $\delta>0$, and let $c_\Delta$ be an absolutely bounded sequence of complex numbers, supported on fundamental discriminants $\Delta$. Then, for any $\epsilon>0$, we have
\[
\sum_{\Delta  \in \mathcal F^\circ(Y)} c_\Delta L_{\Q,\bbS}(1, \rho \otimes \chi_\Delta)
=
\sum_{\Delta  \in \mathcal F^\circ(Y)}  c_\Delta  \sumS_{n \leq Z} \frac{\psi(n)\chi_\Delta(n)}{n} + O_{\bbS, \rho, \epsilon}\left(Y^{\frac{1}{2} + \epsilon}Z^{\epsilon}+Y^{1 + \epsilon}Z^{-\frac{1}{2}+\epsilon}\right),
\]
where $\psi(n)$ is the multiplicative function defined by
\[
L_{\Q,p}(s,\rho) = \sum_{j\geq 0} \frac{\psi(p^j)}{p^{js}}.
\]
\end{prop}

In the above statement, $\bbS$ can be taken to be any finite set of places and we remind the reader of our conventions that $S= \prod_{p \in \bbS,p<\infty} p$, $\sum^S_n$ denotes a sum over integers coprime to $S$, and $\prod^S_p$ denotes a product indexed by primes not in $\bbS$, respectively.
We will ultimately be applying this result in the case that $\rho$ is the Galois representation associated to $\Pic \Xbar$. In this case, $L$ is the splitting field of $\Pic \Xbar$, as in \textsection\ref{S:PeyreoverK}.

\begin{remark}
    In the work of Schmidt and Le Rudulier on $\Hilb^2 \PP_\Q^2$ and $\Hilb^2(\PP_\Q^1 \times \PP_\Q^1)$, it was also necessary to truncate $L$-series. In Schmidt, the $L$-series is simply $L_{\Q, \emptyset}(1,\chi_\Delta)$ and in Le Rudulier $L_{\Q,\emptyset}(1,\chi_\Delta)^2$, \cite[Lem.\ 16]{MR1330740} and \cite[Lem.\ 5.3]{lerudulier-thesis}, respectively. In these cases, it is possible to truncate each individual $L$-series with an acceptable error term. However, for more general $L$-functions (even for moments of $L_{\Q, \emptyset}(1,\chi_\Delta)$ beyond the third) this simple approach fails unconditionally hence the need for an on-average truncation.
\end{remark}

The key tool that will allow us to truncate on average is a variation on the celebrated quadratic large sieve of Heath-Brown.
 
\begin{lemma}\label{lem:quadsieve}
Let $\alpha_m$ and $\beta_n$ be complex sequence satisfying $\vert \alpha_m \vert \ll 1$ for all $m$ and $\vert \beta_n \vert \ll \frac{\tau_r(n)^{C}}{n}$ for all $n$, for absolute constants $C$ and $r$. Then, for all $M,N \geq 1$ and $\epsilon>0$, we have
\begin{itemize}
\item[(a)] if $N \leq M^2$ then
\[
\sum_{m \leq M} \sum_{N \leq n < 2N} \alpha_m \beta_n \left( \frac{m}{n} \right)
\ll_{\epsilon,C}   M^{1+\epsilon}N^{-\frac{1}{2}} + M^{\frac{1}{2}+\epsilon}, \quad \text{ and if }
\]
\item[(b)] if $N > M^2$ then
\[
\sum_{m \leq M} \sum_{N \leq n < 2N} \alpha_m \beta_n \left( \frac{m}{n} \right)
\ll_{r, C} M^{\frac{1}{2}} (\log N)^{\frac{r^{2C}-1}{2}}.
\]
\end{itemize}
\end{lemma}

\begin{proof}
\begin{enumerate}
    \item[(a)] This is a variation on \cite[Cor.\ 4]{MR1347489}. Let $\gamma_n = \beta_n n^{1-\epsilon}$, which by assumption is bounded. Then the sum in question is 
    \[
    \sum_{m \leq M} \sum_{N \leq n < 2N} \alpha_m \frac{\gamma_n}{n^{1-\epsilon}} \left(\frac{m}{n} \right).
    \]
    We write this as 
    \[
    \sum_{N \leq n < 2N} \frac{\delta_n}{n^{1-\epsilon}},
    \] where $\delta_n= \gamma_n \sum_{m \leq M} \alpha_m \left( \frac{m}{n} \right).$
    The bound now follows from partial summation using \cite[Cor.\ 4]{MR1347489} and the fact that $N^\epsilon \ll M^{2\epsilon}$.
    \item[(b)] Now suppose that $N > M^2$. By the Cauchy--Schwarz inequality, the left hand side can be bounded by 
    \[ \ll
    M^{\frac{1}{2}} \left( \sum_{m \leq M} \left \vert \sum_{N \leq n < 2N} \beta_n \left( \frac{m}{n} \right) \right \vert^2  \right)^{\frac{1}{2}}.
    \] We now apply the standard large sieve for primitive characters \cite[Eqn.\ (3)]{MR1347489} producing the bound
    \[\ll
    M^{\frac{1}{2}} \left(  (M^2+N)  \sum_{N \leq n < 2N}\vert  \beta_n \vert^2 \right)^{\frac{1}{2}}.
    \]
    This becomes 
    \[ \ll
    M^{\frac{1}{2}} \left(  \left(\frac{M^2}{N}+1\right) (\log N)^{r^{2C}-1}  \right)^{\frac{1}{2}},
    \] using the standard estimate $\sum_{N \leq n < 2N} \tau_r(n)^A \ll N (\log N)^{r^A-1}$. \qedhere
    \end{enumerate}
\end{proof}

We will also require a bound on the growth of $\psi$.

\begin{lemma}\label{lem:we need to talk about psi}
    For the multiplicative function $\psi$ associated to a Galois representation $\rho$ there are constants $C>0$ and $r>0$ only depending on $\dim \rho$ such that
    \[
    \psi(n) \ll \tau_r(n)^C.
    \]
\end{lemma}

We need the bound in terms of divisor functions to be able to apply Lemma~\ref{lem:quadsieve}. For all other applications the bound $\psi(n) \ll n^\epsilon$ will suffice.

\begin{proof}
Consider all cyclotomic polynomials $\Phi_n$ of degree at most $\dim \rho$, of which there are finitely many. Since $(X^n-1)/\Phi_n(X)$ is a polynomial we see that the expansion of $\Phi_n(X)^{-1}$ has only finitely many different coefficients. So all their coefficients are bounded by an absolute constant $C'$ depending only on $\dim \rho$.

We will first bound $\psi(p^k)$ for a prime number $p$. Since $\Frob_p$ has finite order the characteristic polynomial is the product of at most $\dim \rho$ cyclotomic polynomials of degree at most $\dim \rho$. We conclude that $\psi(p^k) \leq C'^{\dim \rho} \tau_{\dim \rho}(p^k)$. Now we find $\psi(n) =\prod_{p^k \mid\!\mid n} \psi(p^k) \leq C'^{\omega(n)\dim \rho} \tau_{\dim \rho}(n)$. Using $2^{\omega(n)} \leq \tau_2(n)$ and taking $r=\max\{2,\dim \rho\}$ we find $\psi(n) \ll \tau_r(n)^{C}$ for $C >0$ only depending on $\dim \rho$.
\end{proof}

We now have all the necessary prerequisites to prove our on-average truncation result.

\begin{proof}[Proof of Proposition~\ref{prop:trunc}]
We start by observing that $L_{\Q,\bbS}(s,\rho \otimes \chi_\Delta)$ is the $L$-function attached to the Galois representation associated to $\rho$ twisted by the quadratic character $\chi_\Delta$. Writing the local factors $L_p(s, \rho) = \sum_{k\geq 0} \psi(p^k)p^{-ks}$, we can express the twisted local factor as $\sum_{k\geq 0} \psi(p^k)\chi_\Delta(p^k)p^{-ks}.$ Expanding the resulting Euler product gives the series appearing in the main term of the statement.

We break the sum in the main term into three parts
\begin{align*}
 \sum_{n \leq Z} \frac{\psi(n)\chi_\Delta(n)}{n} + \sum_{Z< n \leq e^{Z^\epsilon}} \frac{\psi(n)\chi_\Delta(n)}{n} + \sum_{n> e^{Z^\epsilon}} \frac{\psi(n)\chi_\Delta(n)}{n} &=: \Sigma_1 + \Sigma_2 + \Sigma_3.
\end{align*}
Substituting this into the original sum will yield 3 corresponding sums, $S_1, S_2$ and $S_3$ with 
\[
S_i = \sum_{\Delta \in \mathcal F^\circ(Y)}  c_\Delta \Sigma_i.
\]
The sum $S_1$ is the main term and hence it suffices to bound $S_2$ and $S_3$. We start with $S_3$ which can be bounded using partial summation. 
We write $$\sum_{n > e^{Z^\epsilon}} \frac{\psi(n) \chi_\Delta(n)}{n} = \sum_{n > e^{Z^\epsilon}} \frac{1}{n} (B_n - B_{n-1}),$$ where $B_n = \sum_{b \leq n} \psi(b) \chi_\Delta(b)$. 
We claim that
\[
B_n \ll_{\rho,A} \frac{n}{\log^A n} \vert \Delta \vert.
\] 
This leads to a bound of 
\[\sum_{n > e^{Z^{\epsilon}}} \frac{\psi(n) \chi_\Delta(n)}{n} \ll_{\rho,A} Z^{-A\epsilon } \vert \Delta \vert.\] Summing this over $\Delta$ gives an error of size $Z^{-A\epsilon}Y^{2}$ which is clearly sufficient if we choose $A$ large enough.

We first suppose that $\rho$ is irreducible.
The twisted representation $\rho \otimes \chi_\Delta$ factors through $\Gal(L/\Q) \times \Z/2\Z$. The conductor of this twisted representation is of size $O_\rho(\vert \Delta \vert)$. By the Brauer induction theorem (see \cite[Eqn.\ 5.108]{IK}), the Artin $L$-function attached to the twisted representation can be written as a product of integer powers of Hecke $L$-functions. Since we cannot rule out the possibility that the powers are negative, we bound using the approach of the prime number theorem for Hecke $L$-functions (e.g.\ \cite{Goldstein}). Since $\Delta \in \mathcal F^\circ(Y)$, the product $\psi \chi_\Delta$ is a non-principal Hecke character and thus we have the bound
\begin{equation}\label{eq:HeckePNT}
\sum_{b \leq n} \psi(n) \chi_\Delta(n) \ll_A \frac{n}{\log^A n} q(\rho \otimes \chi_\Delta) \ll_{\rho,A} \frac{n}{\log^A n} \vert \Delta \vert,
\end{equation}
as claimed.
Now suppose that $\rho$ is reducible and has $t$ irreducible subrepresentations. Then the function $\psi(n)$ is the convolution of $t$ multiplicative functions $\psi_1, \ldots, \psi_t$.
Moreover, since $K \in \mathcal F^\circ(Y)$, we have that $\psi_i \chi_\Delta$ is associated to a non-trivial irreducible representation for each $i =1, \ldots, t$.
Hence,
\begin{align*}
B_n 
&= \sum_{b \leq n} (\psi_1 \ast \ldots \ast \psi_t)(b)\chi_\Delta(b)
=
\sum_{b_1 \cdots b_t \leq n } \psi_1(b_1)\chi_\Delta(b_1) \cdots \psi_t(b_t)\chi_\Delta(b_t)\\
&=
\sum_{b_1 \cdots b_{t-1} \leq n } \psi_1(b_1)\chi_\Delta(b_1) \cdots \psi_{t-1}(b_{t-1})\chi_\Delta(b_{t-1})
\sum_{b_t \leq n/b_1\cdots b_{t-1}} \psi_t(b_t)\chi_\Delta(b_t).
\end{align*}
Applying the bound \eqref{eq:HeckePNT} to the innermost sum, we have the bound
\begin{align*}
B_n &\ll_{\rho,A} \vert \Delta \vert^{\deg \rho}\sum_{b_1 \cdots b_{t-1} \leq n } \frac{n}{b_1\cdots b_{t-1}\log(n/b_1\cdots b_{t-1})^{A+t}}\\
&\ll_{\rho,A}
\vert \Delta \vert^{\deg \rho} \frac{n}{\log^{A+t} n} 
\sum_{b_1 \cdots b_{t-1} \leq n } \frac{1}{b_1\cdots b_{t-1}} \left( 1 + O \left( \frac{\log(b_1\cdots b_{t-1})}{\log n} \right) \right),
\end{align*} from which the claim is immediate.

Finally, we estimate $S_2$, which is given by
\begin{align*}
&\sum_{\Delta \in \mathcal F^\circ(Y)}  c_\Delta   \sum_{Z< n \leq e^{Z^\epsilon}} \frac{\psi(n)\chi_\Delta(n)}{n}
\end{align*}
Since $\Delta$ is a fundamental discriminant, we may write it as either a squarefree integer $d\equiv 1 \mod 4$ or $4d$ for squarefree $d \equiv 2, 3 \mod 4$. We partition $\mathcal F^\circ(Y)$ into 4 disjoint sets
\[
\mathcal F^\circ(Y)
=
\mathcal{D}_{-} \cup \mathcal{D}_{+} \cup \mathcal{E}_{-} \cup \mathcal{E}_{+}.
\] Here $\mathcal{D}$ and $ \mathcal{E}$ correspond to a fundamental discriminant being either 1 or 0 mod 4, and the subscript indicates the sign of the fundamental discriminant.
We will bound the contribution from $\mathcal{D}_{+}$, the other bounds being mutatis mutandis.
We break the $n$ sum into dyadic intervals
\begin{align*}
\sum_{\substack{Z \leq N \leq e^{Z^\epsilon}\\ \text{dyadic}}} \sum_{\Delta \in \mathcal{D}_+}  c_\Delta \sum_{N \leq n < 2N} \frac{\psi(n)\chi_\Delta(n)}{n}
=
\sum_{\substack{Z \leq N \leq e^{Z^\epsilon}\\ \text{dyadic}}} \sum_{\substack{0<d \leq Y \\ d \equiv 1 \text{mod} 4\\ d \in \mathcal Q(Y)}} \mu^2(d)  c_d \sum_{N \leq n < 2N} \frac{\psi(n)\chi_d(n)}{n}.
\end{align*}
Set $\beta_n =\frac{\psi(n)}{n}$ and $\alpha_d$ to be $\mu^2(d) c_d$ multiplied by the indicator function of being congruent to 1 mod 4 and the indicator function of belonging to the set $\mathcal F^\circ(Y)$. Now we can apply Lemma~\ref{lem:quadsieve}. Observe that by definition $\alpha_d$ is absolutely bounded and $\vert \beta_n \vert \ll \frac{\tau_{r}(n)^C}{n}$ with $r$ as in Lemma~\ref{lem:we need to talk about psi}.
Hence, our sum is bounded by\[
\sum_{\substack{Z < N \leq Y^2\\ \text{dyadic} }} \left( Y^{1 + \epsilon} N^{-\frac{1}{2}} + Y^{\frac{1}{2}+\epsilon} \right)
+
\sum_{\substack{Y^2 < N \leq e^{Z^\epsilon}\\ \text{dyadic} }} Y^\frac{1}{2} \left(\log N\right)^{\frac{r^{2C}-1}{2}}.
\]
The first of these terms is at most $Y^{1+\epsilon}Z^{- \frac{1}{2}} + Y^{\frac{1}{2}+\epsilon}$,
and the latter is 
\[
Y^\frac{1}{2}\sum_{\substack{Z < N \leq e^{Z^\epsilon}\\ \text{dyadic} }} (\log N)^\frac{r^{2C}-1}{2} \ll_\epsilon Y^\frac{1}{2} Z^{\frac{r^{2C}-1}{2}\epsilon}.
\]
This suffices since $r$ depends only on $\rho$. 
\end{proof}

\subsection{Functions arising from Euler factors}\label{ss:Eulerfactors}
Our ultimate aim in the sequel is to sum an Euler product defined in quadratic fields over all such quadratic fields.
We define here some terms  which will appear in these Euler products.

\begin{defn}
For $p \notin \bbS$ and $w_p$, $w_{p^2}$, $\lambda_p$ and $\bar{\lambda}_p$ as defined in Definition~\ref{defn:local convergence factors and Tamagawa numbers} we define
\[
f_p = \lambda^{-1}_p w_p \quad \text{ and } \quad f_{p^2} = \bar{\lambda}^{-1}_p \lambda^{-1}_p w_{p^2}.
\]
\end{defn}

Recall that $X$ is a weak Fano variety over $\Q$ of dimension at least 2. The functions $f_p$ and $f_{p^2}$ arise naturally within $\tau_\bbS(X_K)$.  
Indeed, if $p$ is a prime which ramifies in $K$, then the Euler factor in $\tau_\bbS(X_K)$ is simply $f_p$. If the prime is inert, then it is $f_{p^2}$ and if it is split, the factor equals $f_p^2$. Hence, for unramified primes, the Euler factor in $\tau_\bbS(X_K)$ at $p$ can be expressed as
\begin{equation}\label{eq:localfactor}
\frac{f_p^2+f_{p^2}}2 + \chi_K(p) \frac{f_p^2-f_{p^2}}2.
\end{equation}
In manipulating $\tau_\bbS(X_K)$, we will encounter several multiplicative functions created from these $f_{p}$ factors.

\begin{defn}\label{def:eulerfunctions}
We define the following functions
    \begin{itemize}
\item[(i)] $f(d) = \prod^S_{p \mid d} f_p$;
\item[(ii)] $F(a) = \prod^S_{p \mid a} \frac{f_p^2+f_{p^2}}{2}$;
\item[(iii)] $G(b) = \prod^S_{p \mid b} \frac{f_p^2-f_{p^2}}{2}$;
\item[(iv)] $R(b)= \frac{G(b)}{F(b)}$, and
\item[(v)] $g(d) =\prod^S_{p \mid d} \frac{2f_p}{f_p^2+f_{p^2}}$.
\end{itemize}
\end{defn}

\begin{lemma}\label{L:props}
The following properties all hold.
\begin{enumerate}
    \item[(a)] All the functions in Definition~\ref{def:eulerfunctions} are multiplicative.
    \item[(b)] We have $f_{p} = 1 + O\left(\frac{1}{p^{3/2}} \right)$ and $f_{p^2} = 1 + O\left(\frac{1}{p^{3}} \right)$.
    \item[(c)] The product $\prod_p f_p$ converges.
    \item[(d)] The function $g$ is 1 on primes in $\bbS$ and otherwise $g(p) = 1 + O \left( \frac{1}{p^{3/2}}\right)$.
    \item[(e)] There exists an absolute constant $A$ such that $\vert g(n) \vert \leq A$ for all $n$.
    \item[(f)] For squarefree $n$, we have $R(n) \ll \frac{1}{n^{3/2}}.$
\end{enumerate}
\end{lemma}

\begin{proof}
    \begin{enumerate}
        \item[(a)] This is immediate from the definition as a product.
        \item[(b)] Since $p \notin \bbS$, we have that $w_p = \frac{\#\calX(\F_p)}{p^{\dim X}}$ and $w_{p^2} = \frac{\#\calX(\F_{p^2})}{p^{2\dim X}}$. Peyre \cite[p.\ 117]{MR1340296}) then provides the expression
        \[
        \lambda_p^{-1} \frac{\#\calX(\F_p)}{p^{\dim X}} = 1 + O\left( \frac{1}{p^{3/2}} \right) \quad \text{ and } \quad \bar\lambda_p^{-1}\lambda_p^{-1} \frac{\#\calX(\F_{p^2})}{p^{2\dim X}} = 1 + O\left( \frac{1}{p^3}\right).
        \]
        \item[(c)] This follows immediately from the previous part. Indeed, the correction factors $\lambda^{-1}_p$ are introduced precisely to ensure this convergence.
        \item[(d)] By the definition of $\prodS$, the function $g$ takes the value 1 on primes in $\bbS$. For $p \not \in \bbS$, we have 
        \[
        g(p) = \frac{2f_p}{f_p^2 + f_{p^2}}
        = \frac{2\left(1 + O\left( \frac{1}{p^{3/2}} \right)\right)}{\left(1 + O\left( \frac{1}{p^{3/2}} \right)\right)^2 + 1 + O\left( \frac{1}{p^{3}} \right)}
        =
        \frac{2 + O\left( \frac{1}{p^{3/2}} \right)}{2 + O\left( \frac{1}{p^{3/2}} \right)}.
        \]
        \item[(e)] This follows immediately from the previous part.
        \item[(f)] Similarly, we have
        \[
        R(p) = \frac{f_p^2 - f_{p^2}}{f_p^2 + f_{p^2}}
        =
        \frac{\left(1 + O\left( \frac{1}{p^{3/2}} \right)\right)^2 - 1 + O\left( \frac{1}{p^{3}} \right)}{\left(1 + O\left( \frac{1}{p^{3/2}} \right)\right)^2 + 1 + O\left( \frac{1}{p^{3}} \right)}
        = \frac{O\left( \frac{1}{p^{3/2}}\right)}{2 + O\left( \frac{1}{p^{3/2}}\right)}.
        \] We conclude that $R(p) \ll p^{-3/2}$ and the claim follows by multiplicativity. \qedhere
    \end{enumerate}
\end{proof}

\subsection{Summing Euler products over quadratic fields}
We are now ready to prove our main result on summing Tamagawa numbers.

\begin{theorem}\label{thm:summingc}
Let $X/\Q$ be a weak Fano variety of dimension at least $2$.
\begin{enumerate}
    \item[(a)] If $\dim X \geq 3$ then the sum $\sum  \limits_{K \in \calF^\circ(Y)}\tau(X_K)$ converges absolutely.
    \item[(b)] If $X$ is a surface, then as $Y \to \infty$, we have
\[
\sum_{K \in \mathcal F^\circ(Y)} \tau(X_K) = \tau(\Hilb^2 X) \log Y + O(1).
\]
\end{enumerate}
\end{theorem}

We remark that the sum can be extended to all quadratic fields of bounded discriminant with the same result, as there are only finitely many quadratic fields not linearly disjoint to $L$. We have opted to exclude them as it is possible that $N_K(X,B)$ possible has a larger order of magnitude for $K \notin \calK(X)$ by the discussion in \textsection \ref{ss:accumulating subsets}.

We first start by specifying the local splitting behaviour of the extension at the places in $\bbS$. Theorem~\ref{thm:summingc} follows by summing over all possible such choices. Recall that in the case $X$ is a surface we have
\[
\tau(\Hilb^2 X)= \lim_{s\to 1}\left[ (s-1)^{\rho(X)}L_{\Q,\bbS}(s,\Pic \Xbar)\right]
\left[\sum_\Xi  \widetilde \delta_\Xi w_\Xi(X)\right] \left[\prodS_p \widetilde\lambda^{-1}_p \widetilde{w}_p\right],
\]
so the theorem follows from the following proposition.

\begin{prop}
Let $\bbS$ denote the set of places of $\Q$ defined in Definition~\ref{defn:bad places for X}, and fix $\Xi \in \prod_{v \in \bbS} \Q_v^\times/\Q_v^{\times, 2}$.
\begin{itemize}
    \item[(a)] If $\dim X > 2$ then
    \[
    \sum_{K \in \mathcal F^\circ_\Xi(Y)} \tau(X_K)
    \]
    converges absolutely.
    \item[(b)] If $\dim X = 2$ then
    \[
\sum_{K \in \mathcal F^\circ_\Xi(Y)} \tau(X_K) \sim \lim_{s\to 1}\left[ (s-1)^{\rho(X)}L_{\Q,\bbS}(s,\Pic \Xbar)\right] \left[\widetilde \delta_\Xi w_\Xi(X) \prodS_p \widetilde\lambda^{-1}_p\widetilde w_p\right] \log Y
\]
as $Y \to \infty$.
    
\end{itemize}
\end{prop}

\begin{proof}
As discussed in the beginning of this section, the Tamagawa number  $\tau(X_K)$ may be factored as the product $\kappa_\Xi \vert \Delta_{K/\Q}\vert^{-\frac{\dim X}{2}}\tau_\bbS(X_K)$.
We will first determine $\sum_{K \in \mathcal F^\circ_\Xi(Y)} \tau_\bbS(X_K)$, then trivially incorporate $\kappa_\Xi$ and then use partial summation to work in the remaining factor $\vert \Delta_K \vert^{-1}$.

In \eqref{eq:localfactor}, we deduced an expression for the Euler factor at an unramified prime $p$. Plugging this expression, and the functions defined in Definition \ref{def:eulerfunctions}, into the expression for $\tau_\bbS(X_K)$ we find 
\begin{align*}
\sum_{K \in \mathcal F^\circ_\Xi(Y)} \tau_\bbS(X_K) & = \sum_{K \in \mathcal F^\circ_\Xi(Y)} L_{\Q,\bbS}(1,\Pic \Xbar \otimes \chi_K) \prodS_{p \mid \Delta_K}f_p \prodS_{p \nmid \Delta_K} \left[F(p) +\chi_K(p)G(p)\right] \\
& = \sum_{K \in \mathcal F^\circ_\Xi(Y)} L_{\Q,\bbS}(1,\Pic \Xbar \otimes \chi_K)  \prodS_p F(p) \prodS_{p \mid \Delta_K}\frac{f_p}{F(p)} \prodS_{p \nmid \Delta_K} \left[1 +\chi_K(p)R(p)\right]\\
& = \left(\prodS_p F(p)\right) \sum_{K \in \mathcal F^\circ_\Xi(Y)} \left(L_{\Q,\bbS}(1,\Pic \Xbar \otimes \chi_K) g(\Delta_K) \sumS_{b \geq 1} \mu^2(b) \chi_K(b) R(b)\right), 
\end{align*}
where we have used that $\prod^S F(p)$ converges absolutely since $\prod^S f_p$ does and $f_{p^2} = 1 + O(p^{-3})$.
The sum over $b$ is absolutely convergent by Lemma~\ref{L:props}(e), and is bounded uniformly in $K$.
Re-arranging, we have
\[
\left(\prodS_p F(p)\right)\sumS_{b \geq 1}  \mu^2(b) R(b) \sum_{K \in \mathcal F^\circ_\Xi(Y)} L_{\Q,\bbS}(1,\Pic \Xbar \otimes \chi_K) g(\Delta_K) \chi_K(b).
\]
As the sum defining $L_{\Q,\bbS}(1,\Pic \Xbar \otimes \chi_K)$ is only conditionally convergent and so to aid further manipulation of the sum, we will truncate it. 

We apply Proposition \ref{prop:trunc} with $c_\Delta$ equal to $\chi_K(b)g(\Delta)$. For $Z$ a power of $Y$ to be determined later,
we conclude that $\sum_{K \in \mathcal F^\circ_\Xi(Y)}\tau_\bbS(X_K)$ is given by
\begin{equation}\label{eq:newform}
\left(\prodS_p F(p)\right)\sumS_{b \geq 1}\mu^2(b) R(b) \sumS_{n \leq Z}  \left( \frac{\psi(n)}{n} \sum_{K \in \mathcal F^\circ_\Xi(Y)} \chi_K(nb)  g(\Delta_K) \right),
\end{equation} plus an error of size $O_{\bbS, \epsilon}(Z^{-\frac{1}{2} + \epsilon}Y^{1 + \epsilon} + Z^\epsilon  Y^{\frac{1}{2} + \epsilon})$.

We will attack the innermost sum in \eqref{eq:newform} with Proposition \ref{prop:discsum}, which shows that the main term will come from terms $n$ and $b$ such that $nb$ is a square. If $nb$ is not a square then applying Proposition \ref{prop:discsum} produces an error term of size
\[ S^\epsilon
\sum_{n \leq Z} \sum_{b \geq 1} \mu^2(b) \vert R(b) \vert \frac{\psi(n)}{n} Y^{\frac{1}{2} + \epsilon} (nb)^{\frac{1}{4}+\epsilon}.
\] By Lemma \ref{L:props}(e), we have that $\vert R(b) \vert \ll b^{-3/2}$. Hence the $b$ sum converges and the total error is $O_{\bbS, \epsilon}(S^\epsilon Z^{\frac{1}{4} + \epsilon} Y^{\frac{1}{2} + \epsilon})$.

We turn our attention to the contribution of $nb=u^2$ for an integer $u \geq 1$. For the inner sum, we get
\[
\sum_{\stackrel{K \in \mathcal F^\circ_\Xi(Y)}{(\Delta_K,u) =1}}  g(\Delta_K) = \delta_\Xi Y \prod_{p\in \bbS} \left(1-\frac1p\right) \prodS_{p \mid u} \left(1-\frac1p\right) \prodS_{p \nmid u} \left(1-\frac1p\right)\left(1+\frac{g(p)}p\right) + O_{\bbS, \epsilon}\left(u^{\epsilon} Y^{\frac{1}{2} + \epsilon}\right)
\]
This error term is of size $O_{\bbS, \epsilon}(Z^{\epsilon}Y^{1/2+\epsilon})$ when summed over $u$.
The main term becomes
\[
\delta_\Xi Y \left[\prod_{p \in \bbS} \left(1-\frac1p\right)\prodS_p F(p)\right]=
\widetilde \delta_\Xi Y\prodS_p F(p)
\]
multiplied by the sum over $u$, which is
\[
\sumS_{u \geq 1} \left[ \prodS_{p \mid u} \left(1-\frac1p\right) \prod_{p \nmid u} \left(1-\frac1p\right)\left(1+\frac{g(p)}p\right)  \sum_{\stackrel{nb = u^2}{n \leq Z}} \left(\mu^2(b) R(b) \frac{\psi(n)}{n}  \right)\right].
\]
 We extend the summation over $n$ to all integers so that
\[
\sumS_{u \geq 1} \left[ \prodS_{p \mid u} \left(1-\frac1p\right) \prodS_{p \nmid u} \left(1-\frac1p\right)\left(1+\frac{g(p)}p\right)  \sum_{\stackrel{n \geq 1,b \geq 1}{nb = u^2}} \left(\mu^2(b) R(b) \frac{\psi(n)}{n}  \right)\right]
\]
at the expense of an error term of size
\[Y
\sum_{u \geq 1} f(u) \sum_{\substack{nb = u^2 \\ b \geq 1, n > Z}} \mu^2(b) R(b) \frac{\psi(n)}{n} \ll Y
\sum_{u \geq Z} \frac{f(u)}{u^2} \sum_{\substack{nb = u^2 \\ b \geq 1, n > Z}} \psi(n)
\ll YZ^{-1 + \epsilon},
\]
where we have used $\vert R(b) \vert \ll 1/b$, $f(u) \ll 1$ and $\psi(n)\ll n^\epsilon$. Note that
\[
P := \prodS_p \left(1-\frac1p\right)\left(1+\frac{g(p)}p\right)
\]
converges, so we can write
\[
\prodS_{p \mid u} \left(1-\frac1p\right) \prodS_{p \nmid u} \left(1-\frac1p\right)\left(1+\frac{g(p)}p\right) = P \prodS_{p \mid u} \left(1+\frac{g(p)}p\right)^{-1}
\]
and get the main term
\[
P \sumS_{u \geq 1} \left[ \prodS_{p \mid u} \left(1+\frac{g(p)}p\right)^{-1}  \sum_{\stackrel{n \geq 1,b \geq 1}{nb = u^2}} \left(\mu^2(b) R(b) \frac{\psi(n)}{n}  \right)\right].
\]
Since the $u$ sum converges, we may write it instead as an Euler product, where the local factor at $p \not \in \bbS$ is given by
\[
1 + \sum_{r\geq 1} \left(\prod_{p \mid p^r} \left(1+\frac{g(p)}p\right)^{-1}  \sum_{\stackrel{n \geq 1,b \geq 1}{nb = p^{2r}}} \left(\mu^2(b) R(b) \frac{\psi(n)}{n}\right)  \right).
\]
On expanding the inner sum in the product, only the terms with $b=1$ and $b=p$ are non-zero,
\[
\sum_{\stackrel{b \geq 1, n\geq 1}{nb=p^{2r}}} \mu^2(b) R(b)\frac{\psi(n)}{n} = \frac{\psi(p^{2r})}{p^{2r}}+\frac{\psi(p^{2r-1})}{p^{2r-1}}R(p)
\]
and we multiply the infinite products $P$ and $\prod^S_p F(p)$ back in
\[
\prodS_p \left(F(p) \left(1-\frac1p\right)\left(1+\frac{g(p)}p\right) + \left(1-\frac1p\right)\sum_{r\geq 1} \left( \frac{\psi(p^{2r})}{p^{2r}}F(p)+\frac{\psi(p^{2r-1})}{p^{2r-1}}G(p)\right)\right).
\]
Using the power series 
\[
\sum_{r\geq 1} \frac{\psi(p^{2r})}{p^{2r}} = \frac{\lambda_p+\bar \lambda_p}2-1
\quad \text{ and } \quad 
\sum_{r\geq 1} \frac{\psi(p^{2r-1})}{p^{2r-1}} = \frac{\lambda_p-\bar \lambda_p}2
\]
we find
\[
\prodS_p \left(1-\frac1p\right)\left[F(p)\left(1+\frac{g(p)}p\right) +\left(\frac{\lambda_p+\bar \lambda_p}2-1 \right)F(p) + \left(\frac{\lambda_p-\bar \lambda_p}2 \right)G(p) \right].
\]

Expanding and using $F(p) = \frac{f_p^2+f_{p^2}}{2}$, $G(p) = \frac{f_p^2-f_{p^2}}{2}$ and $g(p)F(p) = f_p$ gives
\begin{align*}
\prodS \limits_p \left(1-\frac1p\right)\left[F(p)+\frac{g(p)F(p)}p +\frac{\lambda_p f_p^2+\bar \lambda_p f_{p^2}}2 - F(p) \right] & =\\
\prodS \limits_p \left(1-\frac1p\right)\left( \frac{f_p}p +\frac{\lambda_p f_p^2+\bar \lambda_p f_{p^2}}2\right) & =\\
\prodS \limits_p \left[ \left(1-\frac1p\right)\lambda_p^{-1} \right]\left(\frac{w_p}p +\frac{w_p^2+w_{p^2}}2 \right) &= \prodS \limits_p\widetilde \lambda^{-1}_p \widetilde w_p.
\end{align*}

Hence, we arrive at 
\[
\sum_{K \in \mathcal F^\circ_\Xi(Y)} \tau_\bbS(X_K)=\left[\widetilde \delta_\Xi \prodS_p \widetilde\lambda^{-1}_p\widetilde w_p\right]  Y + O_{\bbS, \epsilon}\left(Z^{\frac{1}{4} + \epsilon}Y^{\frac{1}{2} + \epsilon}+Z^{-\frac{1}{2} + \epsilon}Y^{1 + \epsilon}\right).
\]
This error term is minimal when $Z=Y^{\frac{2}{3}}$ which results in a total error term of size $Y^{\frac{2}{3} + \epsilon}$.

We can now work in the factor $\kappa_\Xi = \lim_{s\to 1}\left[ (s-1)^{\rho(X)}L_{\Q,\bbS}(s,\Pic \Xbar)\right] w_\Xi(X)$ which is constant on $K \in \calF^\circ_\Xi$. Finally, we will reintroduce the factor $\vert \Delta_K \vert^{-\frac{\dim X}{2}}$ which was left by the roadside at the beginning of this journey. aA application of partial summation yields
\begin{align*}
& \sum_{K \in \mathcal F^\circ_\Xi(Y)} \tau(X_K)
=
\kappa_\Xi \sum_{K \in \mathcal{F}_{\Xi}(Y)} \frac{\tau_\bbS(X_K)}{\vert \Delta_K \vert^{\frac{\dim X}{2}}}\\
=& 
\kappa_\Xi
\left( \left[\left(\widetilde \delta_\Xi \prodS \limits_p \widetilde\lambda^{-1}_p\widetilde w_p \right) Y +O_{\bbS, \epsilon}(Y^{\frac{2}{3} + \epsilon})\right] \times \frac{1}{Y^{\frac{\dim X}{2}}} + \int_1^Y \frac{\left(\widetilde \delta_\Xi\prodS\limits_p \widetilde\lambda^{-1}_p\widetilde w_p\right) t + O_{\bbS, \epsilon}(t^{\frac{2}{3} + \epsilon})}{t^{\frac{\dim X}{2}+1}} \mathrm{d}t\right).
\end{align*}
The result depends on whether $\dim X =2$ or whether it is greater.

\begin{enumerate}
\item[(a)]  In the case that $\dim X > 2$ each of the terms is $O(1)$, and the sum converges.
\item[(b)] If $X$ is a surface, the first term in the integral dominates and we find
\[
\left(\lim_{s\to 1}\left[ (s-1)^{\rho(X)}L_{\Q,\bbS}(s,\Pic \Xbar)\right] \widetilde \delta_\Xi w_\Xi(X) \prodS_p \widetilde\lambda^{-1}_p\widetilde w_p\right) \log Y  + O_{\bbS, \epsilon}(1). \qedhere
\]
\end{enumerate}
\end{proof}

\subsection{Moments of twisted $L$-functions}\label{ss:moments}
This section is dedicated to establishing a corollary to our on-average $L$-function truncation result, Proposition ~\ref{prop:trunc}.
Our main result is to give an asymptotic formula for all the moments of the value at $s=1$ of the quadratic twists (with fixed splitting data) of an Artin $L$-function. 
One way to view the result is that the values at $s=1$ of Artin $L$-functions are constant on average in quadratic twist families, an idea which we will use to estimate certain error terms in $\S$\ref{finalsection}.

\begin{cor}\label{cor:galtwists}
    Let $\rho$ be a Galois representation which factors through a finite extension $L/\Q$, let $\bbS$ a finite set of places and let $\Xi=(\xi_v)_v \in\prod_{v \in \bbS} \Q_v^\times/\Q_v^{\times,2}$. Define $\mathcal{F}_\Xi^\circ (Y)$ to be the discriminants of fields in $\mathcal{F}_\Xi (Y)$ which are linearly disjoint with $L$. Then for any $Y \geq 1$ and $t \in \Z_{\geq 1}$, we have
    \[
    \sum_{\Delta \in \mathcal F^\circ_\Xi (Y)}
    L_{\Q, \bbS}(1, \rho \otimes \chi_\Delta)^t
    = \delta_\Xi
    C_{\rho, t, \bbS} Y + O_{\bbS,\epsilon, t, \rho} \left( Y^{\frac{2}{3} + \epsilon}\right),
    \]
    for $C_{\rho,t,S}  = \frac 6{\pi^2}
    \prod_{p \in \bbS} \left(1+\frac 1p\right)^{-1}
    \prod\nolimits^S_p\left(1+\left(1+\frac 1p\right)^{-1} \sum_{k\geq 1} \frac{\psi^{\ast t}(p^{2k})}{p^{2k}} \right)$ where $\psi^{\ast t}$ is the $t$-fold convolution product of $\psi$ with itself.
\end{cor}

As far as the authors are aware the only existing results concerning twisted moments at $s=1$ are for the trivial representation. We hope that Corollary~\ref{cor:galtwists} will find some new application, for instance perhaps for Malle's conjecture in towers or in the study of the distribution of relative class numbers.

\begin{proof}
    We will first prove it for $t=1$. Throughout we will allow all implicit constants in error terms to depend on $\rho$, $t$ and $\bbS$. By Proposition~\ref{prop:trunc} with $c_\Delta$ the characteristic function for $\mathcal F^\circ_\Xi$ we obtain
    \begin{equation}\label{eq:truncation only L}
    \sum_{\Delta \in \mathcal F_\Xi^\circ (Y)} L_{\Q, \bbS}(1, \rho \otimes \chi_\Delta)
    =
    \sum_{\Delta \in \mathcal F_\Xi^\circ(Y)} \sumS_{n \leq Z} \frac{\chi_\Delta(n)\psi(n)}{n} + O_\epsilon\left( Y^{\frac{1}{2}+\epsilon}Z^\epsilon + Y^{1+\epsilon}Z^{-\frac{1}{2}+\epsilon}\right).
    \end{equation}
    Exchanging summation in the main term, and observing that we are removing at most finitely many quadratic fields (c.f.\ Lemma \ref{lem:lindisjointL}), we have
    \[
    \sumS_{n \leq Z} \frac{\psi(n)}{n}\sum_{\Delta \in \mathcal F_\Xi^\circ (Y)} \chi_\Delta(n)
    =
    \sumS_{n \leq Z} \frac{\psi(n)}{n}
    \sum_{\Delta \in \mathcal F_\Xi(Y)} \chi_\Delta(n) + O_\epsilon(Z^\epsilon).
    \]
    Here we used $\psi(n) \ll n^\epsilon$ from Lemma~\ref{lem:we need to talk about psi}, which we will apply without mention in the remainder of the argument.

    Applying Proposition~\ref{prop:discsum} with $g=1$, the remaining double sum is equal to
    \[
    \sumS_{\substack{n \leq Z\\ n=\square}} \frac{\psi(n)}{n} 
     \delta_\Xi Y \prod_{p \in \bbS} \left( 1 - \frac 1p\right) \prodS_{p \mid n} \left( 1 - \frac 1p \right)\prodS_{p \nmid n} \left( 1 - \frac 1{p^2} \right)\]
    \[+O(n^\epsilon Y^{\frac12+\epsilon}) + O_\bbS(Z^\epsilon)+ O\left( \sqrt Y \sumS_{\substack{n \leq Z\\ n\ne \square}} \frac{\psi(n)}{n} n^{1/4}\sqrt{\log(4n)}\right).
    \]
     The final error term is easily seen to be of size $O\left( Y^{\frac 12 + \epsilon}Z^{\frac 14 + \epsilon}\right).$
    We extend the sum over $n$ into an infinite sum at the expense of an error term $O_\bbS(YZ^{-\frac 12+ \epsilon})$ since it converges absolutely, and obtain a main term of
    \[
    Y \left[ \delta_\Xi \frac 6{\pi^2} \prod_{p \in \bbS} \left( 1+\frac1p\right)^{-1}\right]  \sumS_{n=\square} \frac{\psi(n)}{n}
    \prod_{p \mid n} \left( 1 + \frac 1p \right)^{-1}.
    \]
    We can write the sum over $n$ as the required Euler product.
    The error term is minimised upon taking $Z=Y^{\frac23}$.

    In the case of general $t$ we note that $L(1,\rho \otimes \chi_\Delta)^t = L(1,(\rho \otimes \chi_\Delta)^{\oplus t})$ and the coefficients of that representation are $\psi^{\ast t}$.
    \end{proof}

To demonstrate the power of this result, we improve the error term for the moments of $L$-functions associated to quadratic fields.

\begin{cor}[Moments of $L(1,\chi)$]
    For $Y \geq 1$ and $t \in \mathbb{Z}_{\geq 1}$, we have
    \[
    \sum_{\Delta \in \calF_{\pm}(Y)} L(1, \chi_\Delta)^t
    =
    Y\frac 3{\pi^2} \prod_p\left(1+\left(1+\frac 1p\right)^{-1} \sum_{k\geq 1} \frac{\tau_t(p^{2k})}{p^{2k}} \right) + O_{t,\epsilon}(Y^{\frac 23 + \epsilon}),
    \]
    where $\calF_{+}$ is the set of positive, and $\calF_{-}$ the set of negative fundamental discriminants.
\end{cor}

\begin{proof}
    This is simply the special case of the previous result where $\rho$ is the trivial representation, $\bbS = \{\infty\}$ and $\xi_\infty = \pm 1$, since the corresponding local factor $\delta_\infty$ equals $\frac12$.
\end{proof}

While the moments of $L\left( \frac{1}{2}, \chi_d\right)$ are mysteries central to prime number theory, the moments of the values at $s=1$ have been more successfully studied. Stepanov~\cite{MR0106214} was the first in 1959 to compute the asymptotic formula for all $k$, followed by Barban~\cite{MR0151441} in 1962. They established the same asymptotic formula as above except with an error term of size $O_\epsilon\left(Ye^{-c(\log Y)^{\frac{1}{2}-\epsilon}}\right)$. Wolke~\cite{MR0252322} reduced the error term to $O_\epsilon\left(Y^{\frac{3}{4} + \epsilon}\right)$. Our result hence bridges the gap between the best previously known error term for this problem and the conjectured truth of $O\left( Y^{\frac 12 + \epsilon}\right)$. One can immediately deduce from this, and the class number formula, Corollary \ref{cor:classnumbers} advertised in the introduction. Namely, if $h(-d)$ denotes the class number of the quadratic field $\Q(\sqrt{-d})$ then
\[
\sum_{d \in \calF_-(Y)} \mu^2(d) h(-d)^t
=
\widetilde{c}_t Y^{\frac t2 +1}
+
O_\epsilon\left(Y^{\frac t2 + \frac{2}{3} + \epsilon}\right).
\]
If one so desired Proposition~\ref{prop:trunc} is flexible enough to allow one to find the moments of the class number among other families of quadratic fields, say thinner families or discriminants in arithmetic progressions and so forth.

Similarly, one may deduce an analogous result for real quadratic fields although in this case, the class number is inseparable from the regulator and one concludes that
\[
    \sum_{d \in \mathcal{F}_+(Y)}  h(d)^tR(d)^t
=
\widetilde{c}_t Y^{\frac{t}{2}+1}
+
O_{t, \epsilon}\left(Y^{\frac{t}{2} + \frac{2}{3} + \epsilon}\right),
\] where $R(d)$ the regulator of the associated real quadratic field.
In the case $t=1$, the main term for the related sum over all integers $d$ was first written down by Gauss in the Disquisitiones Arithmeticae. An asymptotic was established by Siegel \cite{MR12642} and a secondary term was then established by Shintani \cite{MR384717}. Restricting this mean value to fundamental discriminants was first accomplished by Goldfeld--Hoffstein \cite{MR2041614} and was then extended to arbitrary base fields by Datskovsky \cite{MR1210518} and Taniguchi \cite{MR2327040}. Taniguchi \cite{MR2410385} proved an asymptotic with error term $o(Y^2)$ for the case $t=2$. This appears to be the only result for higher moments in this particular problem, however we note that Raulf \cite{MR3498624} established an asymptotic for all $t$ with an error term saving $Y^\epsilon$ when ordering the sum by the size of the regulator. This builds upon work of Sarnak \cite{sarnakI} who performed a similar sum for $t=1$.

\section{Cutoffs}\label{S:fringes}
When counting points of bounded height on $\Sym^2 V$, many of them arise as a quadratic point on $V$ and its conjugate. By the Northcott property, there are only finitely many such points and thus only finitely many quadratic fields will contribute points of bounded height. In this section, we introduce terminology to measure how large the quadratic field can be and still contain a point of small height.

\subsection{Cutoffs}

For this section we fix a projective variety $V$ with an adelically metrised line bundle $\calL$ over a number field $F$, with an induced absolute adelic height $H \colon V(\bar F) \to \R_{> 0}$. Any constants in this section are allowed to depend on this data.

\begin{defn} Let $K$ be a finite extension of the number field $F$.
\begin{enumerate}
\item[(i)] For a subset $\mathcal U \subseteq \Sym^2 V(F)$ define a family of sets $\mathcal U_K=
\{ P \in V(K) \colon [P, \bar P] \in \mathcal U\} \setminus V(F)$, where $K$ ranges over all quadratic extensions of $F$.
\item[(ii)] We say that $\gamma \in \R_{\geq 0}$ is a \emph{quadratic cutoff}, or simply a \emph{cutoff}, for $\calU$ with respect to $\calL$ if there exists some constant $C_\calU> 0$ such that, for any quadratic extension $K/F$ and any $P \in \calU_K$, we have
\begin{equation}\label{eq:cutoff condition}
    |N_{F/\Q} \Delta_{K/F}| \leq C_\calU H_{\calL}(P)^{2\gamma}.
\end{equation}
\item[(iii)] We say $\gamma$ is a cutoff for $V$ if it is a cutoff for $\mathcal U = \Sym^2 V(F)$.
\end{enumerate}
\end{defn}

A priori, cutoffs depend on the adelic metrisation on $\calL$. We now show that it actually depends only on the underlying line bundle.

\begin{prop} Any cutoff $\gamma$ for a subset $\calU \subseteq \Sym^2 V(F)$ with respect to an adelic height depends only on the underlying line bundle. That is, if  $H$ and $H'$ are two heights coming from adelic metrics on the same line bundle $\calL$, then $\gamma$ is a cutoff with respect to the $H$ if and only if it is so with respect to the $H'$.
\end{prop}

\begin{proof}
If $\gamma$ is a cutoff with respect to $H$, then for any quadratic extension $K/F$ and any $P \in \calU_K$ we have
\[
|N_{F/\Q} \Delta_{K/F}| \leq C H(P)^{2\gamma}.\]
It is a basic property of the height machine that for two heights coming from different adelic metrisations on the same line bundle there are positive constants $c_1,c_2 > 0$, such that $c_1 H' \leq H \leq c_2 H'$. By letting $C' := c_2^{2\gamma} C$, we obtain
\[
|N_{F/\Q} \Delta_{K/F}| \leq C H(P)^{2\gamma} \leq C c_2^{2\gamma} H'(P)^{2\gamma} = C' H'(P)^{2\gamma},
\]
meaning that $\gamma$ is also a cutoff with respect to $H'$, as required.
\end{proof}

Our next fundamental fact is that cutoffs are well behaved under decreasing subsets $\calU$ and increasing $\gamma$.

\begin{lemma}
    Let $\calU \subseteq \Sym^2 V(F)$ be a subset.
    \begin{enumerate}
        \item[(a)] A cutoff $\gamma$ for $\calU$ is also a cutoff for any $\calU' \subseteq \calU$.
        \item[(b)] Let $K/F$ be a quadratic extension. A $\gamma \geq 0$ is a cutoff for $\calU$ if and only if it is a cutoff for
        \[
        \calU \setminus \{[P,\bar P] \colon P \in V(K)\}.
        \]
        \item[(c)] If $\gamma$ is a cutoff for $\calU$ with respect to $\calL$, then so is any real number $\gamma' \geq \gamma$.
    \end{enumerate}
\end{lemma}

\begin{proof}
    \begin{enumerate}
        \item[(a)] Clearly if \eqref{eq:cutoff condition} is satisfied for all points $P \in \calU$ then it is for all points in $\calU'$.
        \item[(b)] By (a), if $\gamma$ is a cutoff for $\calU$ then it is a cutoff for
        \[
        \calU' := \calU \setminus \{[P,\bar P] \colon P \in V(K)\}.
        \]

        Now suppose that $\gamma$ is a cutoff for $\calU'$ with constant $C'$. Let $h$ be the minimal height among points in $X(K)$, and define $C=\max\{C', |N_{F/\Q}\Delta_{K/F}|/h^{2\gamma}\}$. Then \eqref{eq:cutoff condition} is satisfied with constant $C$ for $P \in \calU'$ since
        \[
        |N_{F/\Q}\Delta_{F(P)/F}| \leq C' H(P)^{2\gamma} \leq CH(P)^{2\gamma}
        \]
        and for $P \in V(K)$ we have
        \[
        |N_{F/\Q}\Delta_{K/F}| \leq Ch^{2\gamma} \leq C H(P)^{2\gamma}. 
        \]
        \item[(c)] By the Northcott Property there is a minimal height $h>0$ among the points in $V(\bar F)$ of degree $2$. Define $C' = C h^{2(\gamma'-\gamma)}>0$. Now let $K/F$ be a quadratic extension. Then for $P \in \calU_K$ we have
        \[
        | N_{F/\Q} \Delta_{K/F}| \leq C H(P)^{2\gamma} = C' h^{2(\gamma'-\gamma')} H(P)^{2\gamma} \leq  C'  H(P)^{2\gamma'}.
        \]
        It follows that $\gamma'$ satisfies the cutoff condition with the constant $C'$, as required. \qedhere
    \end{enumerate}
\end{proof}

The situation where $\gamma = 0$ is quite special, and demands some attention.

\begin{lemma}\label{lem:cutoff 0}
    The value $\gamma = 0$ is a cutoff for $\calU$ if and only if we have $\calU_K = \emptyset$ for all but finitely many $K$.
\end{lemma}

\begin{proof}
    If $\gamma = 0$ is a cutoff we have found that for all quadratic $K/F$ and $P \in \calU_K$ we have $|N_{F/\Q} \Delta_{K/F}| \leq C$. Hence, either $K/F$ has a small discriminant or $\calU_K = \emptyset$. The reverse implication is immediate.
\end{proof}

\begin{cor}\label{cor:cutoff 0}
    If $\gamma = 0$ is a cutoff for a cothin $\calU$ then $\Sym^2 V$ fails the Hilbert property, that is, $\Sym^2 V(\Q)$ is a thin set.
\end{cor}

\begin{proof}
    By Lemma~\ref{lem:cutoff 0} we see that $\calU_K$ is nonempty for finitely many quadratic fields $K$. Hence $\calU$ is a thin set as it is contained in the union of the images of the finite morphisms
    \[
    \eta_K \colon \Res_{K/\Q}V_K \to \Sym^2 V
    \]
    from Proposition~\ref{prop:map from res to sym}. Since $\calU$ is thin and $\Sym^2 V(\Q)\setminus \calU$ is thin we conclude that $\Sym^2 V$ fails the Hilbert property.
\end{proof}

\begin{remark}
As any rationally connected variety has finite Brauer group modulo constant elements \cite[p.\ 347]{brauerbook}, the conjecture by Colliot-Th{\'e}l{\`e}ne mentioned in Remark~\ref{rem:ctconj} implies that rationally connected varieties satisfy weak weak approximation  \cite[Defn.\ 13.2.4]{brauerbook} whenever they have rational points. Therefore, by a result of Ekedahl \cite[Thm.\ 3.5.7]{MR2363329} their sets of rational points are not thin. In particular, conjecturally $\Hilb^2 V(\Q)$ is not thin if it is not empty, which implies that $\Sym^2 V(\Q)$ is also not thin. Hence, by Corollary~\ref{cor:cutoff 0}, it is conjecturally impossible for 0 to be a cutoff for rationally connected varieties $V$ for which $\Hilb^2 V(\Q) \neq \emptyset$.
\end{remark}

To determine a cutoff in some cases, the following result of Silverman is useful.

\begin{thm}[{\cite[Lem.\ 2]{MR747871}}]\label{thm:silverman}
Let $F$ be a number field and $\delta_F$ the number of archimedean places of $F$. For a point $P \in \PP^n(\overline F)$ we define $H_{\calO(1)}(P)$ as the na\"ive absolute $\calO(1)$-height on $\PP^n(\overline F)$, $F(P)$ the minimal field of definition for $P$, and $m = [F(P)\colon F]$. Then 
we have
\[
|N_{F/\Q}(\Delta_{F(P)/F})| \leq m^{m\delta_F} H_{\calO(1)}(P)^{2m(m-1)[F \colon \Q]}.
\]
\end{thm}

\begin{example}\label{ex:pncutoff}
    Let $F$ be a number field. The value $\frac{2[F \colon \Q]}{k}$ is a cutoff for $\PP_F^n$ with respect to $\calO(k)$. Hence a cutoff for the anticanonical height is given by $\frac{2[F \colon \Q]}{n+1}$.
\end{example}

In Corollary~\ref{cor:optimalp1} we show that these are actually the smallest possible cutoffs for $\PP^n_F$. For now, let us use the result to show that all varieties under consideration at least have a cutoff.

\begin{prop}\label{prop:cutoffs exist}
    Let $V/F$ be a variety with a  adelic height function $H$ on an ample line bundle, then there is a $\gamma\geq 0$ which is a cutoff for $V$.
\end{prop}

\begin{proof}
    Embed $V$ in some projective space $\PP^N_F$ via $\mathcal L^{\otimes k}$ for some $k$. Then $\calO(1)$ on $\PP_F^N$ restricts to $\calL^k$ on $V$. Since $[F \colon \Q]$ is a cutoff for $\calO(1)$, we see that $k[F \colon \Q]$ is a cutoff for $\calL$.
\end{proof}

We will see that in many examples the cutoff produced in this proof is far from optimal. For example, for a smooth quadric surface $X \subseteq \PP^3_\Q$ we will see that $\frac12$ is an anticanonical cutoff, but this result only shows that the $\calO(2)$-cutoff $1$ of $\PP^3_{\Q}$ is also a cutoff for $X$.

\subsection{Cutoffs of product varieties}

Suppose that $V'$ and $V''$ are two projective varieties over $F$ each equipped with an adelic metric on their anticanonical line bundles. Consider $V' \times V''$ with the induced product adelic metric metric on its anticanonical line bundle. Given $\mathcal U' \subseteq \Sym^2 V(F)$ and $\calU'' \subseteq \Sym^2 V''(F)$ subsets with respective cutoffs $\gamma' \geq 0$ and $\gamma'' \geq 0$, we can consider the induced subset of $\Sym^2(V' \times V'')$. In particular, we write $\mathcal U^{\text{prod}}$ for the preimage of $\calU' \times \calU''$  under the natural morphism $\Sym^2(V'\times V'') \to \Sym^2 V' \times \Sym^2 V''$. One can compare $\mathcal U^{\text{prod}}$ with $\calU'$ and $\calU''$ and, in particular, compute a cutoff for such subsets.

\begin{prop}\label{prop:prodsetfacts}
    \begin{enumerate}
        \item[(a)]  If $\, \calU'=\Sym^2 V'(\Q)$ and $\, \calU''=\Sym^2 V''(F)$ then $\, \calU^\text{prod} = \Sym^2(V' \times V'')(F)$.
        \item[(b)] We have $\gamma = \max\{\gamma',\gamma''\}$ is a cutoff for $\calU^\text{prod}$.
    \end{enumerate}
\end{prop}

\begin{proof}
     Let $H', H''$, and $H=H'H''$ be the associated height functions on $V'$, $V''$, and $V=V'\times V''$ respectively. By the definition of $\gamma'$ and $\gamma''$ there are constants $C'>0$ and $C''>0$ such that for every quadratic $K/F$ and $P' \in \calU'_K$ and $P'' \in \calU''_K$ we have
    \[
    |N_{F/\Q}(\Delta_{K/F})| \leq C' H'(P')^{2\gamma'} \quad \text{ and } \quad |N_{F/\Q}(\Delta_{K/F})| \leq C'' H''(P'')^{2\gamma''}.
    \]
    \begin{enumerate}
        \item[(a)] This follows from the definition; the preimage of $\, \calU' \times \calU'' = \Sym^2 V'(F) \times \Sym^2 V''(F)$ under $\Sym^2(V'\times V'')\to \Sym^2 V' \times \Sym^2 V''$ equals $\Sym^2(V'\times V'')(F)$.
        \item[(b)] There is a constant $C_1 >0$ such that any (at most) quadratic point on $V'$ or $V''$ has height at least $C_1$.
        Now consider a point $P=(P',P'') \in \calU^\text{prod}_K \subseteq (V' \times V'')(K) \setminus (V' \times V'')(F)$. Hence at least one of the coordinates must be a pure $K$-point. Assume $P' \in V'(K)\setminus V'(F)$. By construction of $\, \calU^\text{prod}$ we have $P' \in \calU'$. Hence for $\gamma = \max\{\gamma',\gamma''\}$ we find
        \[
         |\Delta_{K/\Q}| C_1^{2\gamma}\leq H'(P')^{2 \gamma} C_1^{2\gamma} \leq H'(P')^{2 \gamma}H''(P'')^{2\gamma} = H(P)^{2\gamma},
        \]
        and similarly if $P''$ is a pure $K$-point. We conclude that $\gamma$ is a cutoff for $\, \calU^\text{prod}$. \qedhere
        \end{enumerate}
\end{proof}

It can be shown under mild assumptions (for example if $V'\times V''$ has pure $K$-points for infinitely many quadratic extensions $K/\Q$) that if $\calU'$ and $\calU''$ are cothin then $\calU^\text{prod}$ is cothin. In the counting examples considered so far (including the one in the next section) this is obvious from the choice of $\mathcal U$.

For now we use $\calU'_F := \{P \in V(F) \colon \exists P' \in V(F) \textrm{ s.t. } [P,P'] \in \calU'\}$ for the set of $F$-points of $V$ that appear as a part of a point in $\calU'$, and we define $\calU''_F$ similarly. For any quadratic field $K$, we have
\[
        \calU'_K \times \calU''_K\; \subseteq \;\calU_K^\text{prod} \; \subseteq \; (\calU'_K \times \calU''_K)\bigcup (\calU'_K \times \calU''_F) \bigcup (\calU'_F \times \calU''_K).
        \]
As we have seen in this section, if a point is quadratic then its height needs to be large.
If we have a point $(P',P'') \in \calU'_K \times \calU''_K \subseteq\mathcal{U}^{prod}$ with both components purely quadratic, the height of this point is $H_{V'}(P')H_{V''}(P'')$, say, and since both heights need to be large, and their product is bounded, neither can be \textit{too} large. 
However for a point $(P', P'') \in (\calU'_F \times \calU''_K)\subseteq\mathcal{U}^{prod}$ the height of $P'$ can be fairly small since it is a rational (not quadratic) point. This means that the height of $P''$ is not as constrained and can be quite large.
This in turn means that the fields $K$ in which this point can live could also possibly be quite large. 

In summary, the fields which contribute pairs of pure quadratic points on $V' \times V''$ might be smaller than those that can contribute pure quadratic points on either $V'$ or $V''$. We introduce now a definition to deal with this phenomenon.

\begin{defn}
    A point on $\Sym^2(V' \times V'')$ defined over $F$ is called \emph{pristine} if it is contained in $\calU'_K \times \calU''_K$ for some quadratic field $K$. That is to say pristine points are those points on $V' \times V''$ which are pure quadratic in both coordinates. We denote the set of all such points in $\mathcal U^{prod}$ by $\mathcal U^{pris}$.
\end{defn}

\begin{remark}\label{rem:pristine}
    In general, one could say a subset $\calU \subseteq \Sym^2 V(F)$ is \emph{pristine} if there are finitely many dominant fibration $\phi_i\colon V \to C_i$ with $0 < \dim C_i < \dim V$, such that $\calU$ is contained in the set
    \[
    \{[P_1,P_2] \in \Sym^2 V(F) \colon \phi_i(P_1),\phi_i(P_2) \text{ are pure quadratic points for all } i\}.
    \]
    This concept behaves well with respect to thin sets; the conjugate points $P_1$ and $P_2$ must lie in conjugate fibres, so either the fibre over a $F$-point or over a pure quadratic point. One can show that the points on $\Sym^2V(F)$ which lie in a fibre over a $F$-point form a thin set of type I. 
    Clearly, the pristine points on the product $V' \times V''$ are precisely the points which do not lie in fibres over $F$-points of either projection.

    For example, the Hirzebruch surfaces $\mathbb F_1$ and $\mathbb F_2$ admit a fibration over $\PP^1$, and the cutoff of the pure quadratic points is strictly larger than the cutoff of the pristine points with respect to this fibration. This is related to the fact that the pure quadratic points which lie over $F$-points on $\PP^1$ dominate the point count on $\Sym^2 \mathbb F_r$, and must be included in a thin set for Manin's conjecture to hold.
\end{remark}

As discussed above the set of fields in which the pristine points of bounded height can arise might not be the same as the set of fields in which we can have quadratic points on $V' \times V''$. In particular, the cutoff may be different. We formulate some results around pristine points.

\begin{prop}\label{prop:pristinefacts}
     \begin{enumerate}
        \item[(a)] If $\gamma'+\gamma'' \ne 0$ then $\gamma = \frac{\gamma'\gamma''}{\gamma'+\gamma''}$ is a cutoff for $\calU^{\text{pris}}$.
        \item[(b)] $\calU^{prod} \setminus \calU^{pris}$ is a thin subset of $\Sym^2(V' \times V'')(F)$.
    \end{enumerate}
\end{prop}

\begin{proof}
    We keep the same conventions as in the proof of Proposition \ref{prop:prodsetfacts}.
     \begin{enumerate}
        \item[(a)] From $\calU_K = \calU'_K \times \calU''_K$ we find for $P=(P',P'') \in \calU_K$ that
        \begin{align*}
        |N_{F/\Q}(\Delta_{K/F})|^{\gamma'+\gamma''} & \leq \left[C' H'(P')^{2\gamma'}\right]^{\gamma''} \left[C'' H''(P'')^{2\gamma''}\right]^{\gamma'} & \\
        & = C'^{\gamma''} C''^{\gamma'} \left[ H'(P')H''(P'')\right]^{2\gamma'\gamma''} = C H(P)^{2\gamma'\gamma''},
        \end{align*}
        where $C = C'^{\gamma''} C''^{\gamma'} >0$. 
        
        \item[(b)] The product points that are not pristine on $V'\times V''$ look like pure $K$-points $(P',P'') \in (V'\times V'')(K)$ where either $P'$ or $P''$ is actually a $F$-point. Consider the case $P' \in V'(K)\setminus V'(F)$ and $P'' \in V''(F)$. Then the corresponding point $[(P',P''),(\bar{P}', P'')] \in \Sym^2(V' \times V'')$ lies in the preimage of the diagonal $\Delta' \subseteq \Sym^2 V'$ under $\Sym^2(V' \times V'') \to \Sym^2V', [(P_1,P_2),(P_3,P_4)] \mapsto [P_1,P_3]$. Hence the product points which are non-pristine lie in a thin subset of type I.
        \qedhere
        \end{enumerate}
\end{proof}

One can show that the product points coming from the subsets $\calU' \setminus \{[P_1,P_2]\ | P_1,P_2 \in V'(F)\}$ and $\calU'' \setminus \{[P_1,P_2]\ | P_1,P_2 \in V''(F)\}$ are precisely the pristine product points coming from $\calU'$ and $\calU''$, hence we conclude from the remark following Proposition~\ref{prop:prodsetfacts} that $\calU^{\text{pris}}$ is also cothin if $\calU'$ and $\calU''$ are cothin.

\begin{example}\label{ex:cutoffs}
    \phantom{.}\begin{enumerate}
        \item The number $1$ is a cutoff for $\PP_\Q^1 \times \PP_\Q^1$ with respect to the anticanonical line bundle.
        \item For the pristine points a cutoff is given by $\tfrac 12$.
    \end{enumerate}
\end{example}

The second cutoff is the one implicitly used by Le Rudulier~\cite[p.\ 93]{lerudulier-thesis}, although the normalisation of her height is different in this section. Specifically, she is using a relative $\mathcal{O}(1)$-height and so replacing $B$ in her conventions by $B^{1/4}$ recovers the cutoff of $\frac{1}{2}$.

\subsection{Permanence of cutoffs under field extensions}

For the next examples we use the following result on norms of relative discriminants under base change.

\begin{prop}\label{prop:discs under base change}
    Let $F'/F$ be an extension of number fields. Then for any $K/F$ linearly disjoint from $F'$ we have
    \[
    |N_{F'/\Q}(\Delta_{KF'/F'})| = |N_{F/\Q}(\Delta_{K/F})|^{[F'\colon F]}.
    \]
\end{prop}

\begin{proof}
    Because $F'$ and $K$ are linearly disjoint we have, by \cite[Prop.\ I.2.11]{Neukirch}, and the remark following it, that
    \[
    \Delta_{KF'/F} = \Delta_{F'/F}^{[K\colon F]} \Delta_{K/F}^{[F'\colon F]}.
    \]
    Computing the discriminant of $KF'/F$ using the intermediate field $F'$ we have, by \cite[Cor.\ III.2.10]{Neukirch}, that
    \[
    \Delta_{KF'/F} = N_{F'/F}(\Delta_{KF'/F'}) \Delta_{F'/F}^{[KF'\colon F']} = N_{F'/F}(\Delta_{KF'/F'}) \Delta_{F'/F}^{[K\colon F]}.
    \]
    We conclude
    \[
    N_{F/\Q}(\Delta_{F'/F}^{[K\colon F]}) = N_{F/\Q}(N_{F'/F}(\Delta_{KF'/F'}) ) = N_{F'/\Q}(\Delta_{KF'/F'}).\qedhere
    \]
\end{proof}

We will now prove a stability result on cutoffs under base change.

\begin{prop}\label{prop:cutoffs under basechange}
    Consider a variety $V$ over a number field $F$, and let $F'/F$ be a field extension of finite degree. Let $\gamma$ be a cutoff for the base change $V_{F'}/F'$. Then $\frac{\gamma}{[F'\colon F]}$ is a cutoff for $V/F$.
\end{prop}

\begin{proof}
Let $C>0$ be such that for every quadratic extension $K'/F'$ and every point $P \in V_{F'}(K')$ we have
\[
|N_{F'/\Q}(\Delta_{K'/F'})| \leq C H(P)^{2\gamma}.
\]
For any quadratic field $K$ there is a bijection between the $K$-points on $V$ and the $K'=KF$-points on $V_{F'}$, which  respects heights. Hence for any $P \in V(K)$ we have
\[
|N_{F/\Q}(\Delta_{K/F})| = |N_{F'/\Q}(\Delta_{K'/F'})|^{\frac1{[F'\colon F]}} \leq C H(P)^{2\frac{\gamma}{[F'\colon F]}}\qedhere
\]
\end{proof}

We can use this result to provide cutoffs for twists of projective spaces.

\begin{prop}\label{prop:cutoffs SB varieties}
    Let $X/F$ be a Severi--Brauer variety of dimension $n$ with an adelic anticanonical height. Then $\frac{2[F\colon \Q]}{n+1}$ is a cutoff for $X$.
\end{prop}

\begin{proof}
    Let $L/\Q$ be such that $X(L) \ne \emptyset$. Since $X_L \cong \PP^n_L$ we see that $\frac{2[L\colon \Q]}{n+1}$ is a cutoff for $X_L$ and by Proposition~\ref{prop:cutoffs SB varieties} $\frac{2[L\colon \Q]}{[L\colon F](n+1)} = \frac{2[F\colon \Q]}{n+1}$ is a cutoff for $X$.
\end{proof}

\begin{remark}\label{rem:cutoffs under base change}
    The converse to Proposition~\ref{prop:cutoffs under basechange} need not hold in general; not all quadratic extensions of $F'$ are of the form $KF'$! Hence a cutoff of $V/F$ cannot possibly say anything about all quadratic extensions of $F'$. This explains why a cutoff of $V$ need not say anything about cutoffs of $V_{F'}$.
    
    For example, consider a variety $V/F$ for which $\Sym^2 V(F) = \emptyset$, but $V_{F'}/F'$ does have pure $K'$-points for infinitely many quadratic extensions $K'/F'$, then by Lemma~\ref{lem:cutoff 0} $0$ is a cutoff for $V$, but not for $V_{F'}$. This already happens in the situation in Proposition~\ref{prop:cutoffs SB varieties}; a Severi--Brauer surface split by a cubic extension will have $0$ as a cutoff, but $0$ is not a cutoff for $\PP^2_\Q$ as $\Q$ is Hilbertian. Other examples among del Pezzo surfaces are quartic del Pezzo surface of index $4$, which were proven to exist in \cite{CreutzViray24}.
\end{remark}

We conclude this subsection by considering another geometric construction which is of critical importance to this paper, the restriction of scalars.

\begin{prop}\label{p:cutoffres}
    Let $V'/F'$ be an anticanonically adelically metrised $F'$-variety, where $F'/F$ is a number field.
    If $\gamma \geq 0$ is a cutoff for $V'$, then $\frac{\gamma}{[F' \colon F]^2}$ is a cutoff for $\Res_{F'/F}V'$.
\end{prop}

\begin{proof}
    All but finitely many quadratic field $K/F$ produce a quadratic field $KF'/F'$. For those $K$ the pure $K$-points in $V(K)$ correspond to the pure $KF'$-points in $\Res_{F'/F}V'$ under the natural bijection
    \[
    \nu_{V'} \colon \Res_{F'/F}V'(K) \xrightarrow{\cong} V(KF').
    \]
    If $P \in V(KF')$ corresponds to a $Q \in \Res_{F'/F}V'(K)$ we have $H_{\Res_{F'/F}\omega}(Q) = H(P)^{[F'\colon F]}$ by Proposition~\ref{prop:height on Weil res}. Then we find for all such $P$ and $Q$
    \[
    |N_{F/\Q}(\Delta_{K/F})|^{[F' \colon F]} = |N_{F'/\Q}(\Delta_{KF'/F'})| \leq C_{V'} H(P)^{2\gamma} = C_{V'}H_{\Res_{F'/F}\omega}(Q)^{\frac{2\gamma}{[F' \colon F]}},
    \]
    where we have used Proposition~\ref{prop:discs under base change}.
\end{proof}

\begin{remark}
    As in Remark~\ref{rem:cutoffs under base change}, the converse of Proposition~\ref{p:cutoffres} need not hold, since there might be infinitely many quadratic extensions of $F'$ which are not of the form $KF'/F'$ for $K/F$ quadratic.
    When we have good bounds for an individual quadratic extension, usually Silverman's result for products of projective space, we do get an equivalence.
\end{remark}

\subsection{Optimal cutoffs and lower bounds}

In Proposition~\ref{prop:cutoffs exist} we saw how to use the Silverman bound to find examples of cutoffs. Often this does not produce the best possible cutoff, but it does provide an upper bound for the smallest possible cutoff. In this section we will concern ourselves with establishing lower bounds for cutoffs, which turns out to be much harder.

\begin{defn}\label{def:cutoff}
    We say $\gamma \geq 0$ is the \emph{optimal cutoff} if it is a cutoff and if there is no smaller cutoff.
\end{defn}

\begin{remark}
    In order to prove the Manin--Peyre conjecture for a surface using the strategy outlined in this paper, it is necessary to compute precisely the optimal cutoff as, thanks to Theorem~\ref{thm:summingc}, this value will appear in the leading constant. However, one can prove lower bounds for the counting function $N_\Q(\calU, B)$ of the correct order of magnitude without needing to compute the cutoff. Indeed, if one simply sums over all fields of discriminant bounded by $B^\epsilon$ for $\epsilon>0$ sufficiently small, one obtains the correct order of magnitude.
\end{remark}

Clearly if the infimum of all cutoffs is a cutoff too, it must be optimal. However, proving that a cutoff is optimal, or even that the optimal cutoff must exist is non-trivial in general.
Here is however a positive result.

\begin{prop}\label{p:p1opt}
    Consider $\PP^1(\Q)$, then $2$ is the optimal $\calO(1)$-cutoff.
\end{prop}

\begin{proof}
    We will find a constant $C'$ such that there are infinitely quadratic fields $K_i/\Q$, and points $P_i \in \PP^1(K_i)$ such that $C' H(P_i)^4 \leq |\Delta_{K_i}|$ for each $i$. Now if $\gamma \geq 0$ is a cutoff for $\PP^1_\Q$ with constant $C>0$ then we have for the point $P_i$
    \[
    C'H(P_i)^4 \leq |\Delta_{K_i}| \leq C H(P_i)^{2\gamma}.
    \]
    Since there are finitely many quadratic points of bounded height, we conclude that from $i\gg 1$ we get $\gamma \geq 2$.

    We will use the approach in \cite{smallgen} to produce infinitely many points of small height, though this paper considers them with a different height.

    Let $m > 0 $ be an integer such that $\Delta_m = 1-4m^2$ is squarefree, and $K_m/\Q$ be the associated quadratic field. We claim that the points $P_m = [2m \colon -1+\sqrt{\Delta_m}]=\left[ 1 \colon \frac{-1+\sqrt{\Delta_m}}{2m} \right] \in \PP^1(K_m)$ satisfy the conditions required above. Indeed, we see that $\frac{-1+\sqrt{\Delta_m}}{2m}$ has minimal polynomial $P(x)=mx^2+x+m=m(x-\alpha_1)(x-\alpha_2)$ in $\mathbb{Z}[x]$ and hence $H(P_m)=M(P)^{\frac{1}{2}}$ where $M(P)$ denotes the Mahler measure (see \cite{Smyth_2008}). The Mahler measure also satisfies $M(P)=m\prod_{i}\max\{1, |\alpha_i|\}=m$, as $|\alpha_1|=1=|\alpha_2|$. Hence we get $M(P)=m$, and now $H(P_m)^4 = m^2 \leq 4m^2-1 = |\Delta_{K_m}|$ follows for infinitely many points in different quadratic fields since there are infinitely many squarefree values of $1-4m^2$.
\end{proof}

\begin{cor}\label{cor:optimalp1}
    For every $k,n\geq 1$ and number field $F/\Q$, the optimal $\calO(k)$-cutoff of $\PP^n_F$ exists and is equal to $\tfrac{2[F \colon \Q]}k$.
\end{cor}

\begin{proof}
    Since cutoffs behave well with respect to powers of line bundles it is enough to prove the statement for $k=1$. Proposition~\ref{prop:cutoffs under basechange} shows that the statement holds for all number fields $F$, if it does so for $F=\Q$.
    
    Since $\PP^n_\Q$ contains a line $L=\PP^1_\Q$ with $\left.\calO(1)\right|_L = \calO_{\PP^1}(1)$ we see from $\PP^1(\Q) \subseteq \PP^n(\Q)$ that any cutoff for $\PP^n_\Q$ must be cutoff for $\PP^1_\Q$. By the previous result, $\PP^1_\Q$ does not have a smaller cutoff than $1$, hence the same holds for $\PP^n_\Q$. Since we saw that $1$ is a $\calO(1)$-cutoff for $\PP^n_\Q$ it is in fact optimal.
\end{proof}

The same proof shows that if $\gamma$ is the optimal cutoff for $V/F$ then $\gamma[F'\colon F]$ is a lower bound for all cutoffs for $V_{F'}$.

In general, proving that optimal cutoffs exist or determining their value appears hard. We offer the following two approaches for producing lower bounds for cutoffs. The first of these approaches was used in the proof of Corollary~\ref{cor:optimalp1}.

\begin{example}
Let a variety $V/\Q$ contain a rational curve $C$ with $C(\Q) \ne \emptyset$, and an ample line bundle $\calL$. Define $d=C \cdot \calL$. Then any cutoff for $V$ with respect to $\calL$ is at least $\tfrac 2d$, since a cutoff for $\Sym^2 V(\Q)$ determines one for the subset $\Sym^2 C(\Q)$.

More generally, for subsets $\calU \subseteq \Sym^2 V(\Q)$ which contain all (or even just many) pure quadratic points of $C$, we can use the same trick. For example, the diagonal $\Delta$ on $\PP_\Q^1 \times \PP_\Q^1$ shows that $\frac 2{\deg \left.\calO(2,2)\right|_\Delta} = \frac 24=\frac 12$ is a lower bound for both the product and pristine points on $\PP^1_\Q \times \PP^1_\Q$. For the latter this is sharp, for the former the actual optimal cutoff is $1$, as seen by the many horizontal and vertical rational curves, which all have anticanonical degree $2$.
\end{example}

In some cases proving that a cutoff $\gamma$ for $X$ is actually the optimal cutoff is a byproduct of establishing Manin's conjecture for $\Sym^2 X$.

\begin{example}\label{ex:Manin implies optimal cutoff}
Consider $X=\PP^2_\Q$ then Manin's conjecture was proven by Schmidt~\cite{MR1330740} by establishing
\[
N_K^*(\PP^2,B^{\frac12}) = c_{\PP^2,K} B + O(c^*_{\PP^2,K} B^{\frac 56}),
\]
and an $\calO(3)$-cutoff for $X=\PP_\Q^2$ is given by $\gamma=\tfrac 23$. Hence for $\calU \subseteq \Sym^2 \PP^2(\Q)$ the cothin set of conjugate pairs of pure quadratic points on $\PP_\Q^2$ he obtains
\begin{align*}
N_\Q(\calU, B) & = \frac 12 \sum_{|\Delta| \leq B^\gamma}N_K^*(\PP^2,B^{\frac12}) = \sum_{|\Delta| \leq B^\gamma} \left[ c_{\PP^2,K} B + O(c^*_{\PP^2,K} B^{\frac 5{6}}) \right] \\
& = \frac 12 \sum_{|\Delta| \leq B^\gamma} \left[ \alpha(\PP_K^2) \tau(\PP_K^2) B + O(c^*_{\PP^2,K} B^{\frac 56}) \right] \\
& =  \frac 12 \alpha(\PP_\Q^2) \tau(\Sym^2 \PP_\Q^2) B \log (B^\gamma) + O\left( \sum_{|\Delta| \leq B^\gamma} c^*_{\PP^2,K} B^{\frac 56} \right).
\end{align*}
Schmidt proved \cite[\textsection 8]{MR1330740} that
\[
\sum_{|\Delta| \leq Y} c^*_{\PP^2,K} \ll Y^{\frac 14} \left( \log^+ Y\right)^{\frac 12},
\]
and hence
\begin{align*}
N_\Q(\calU, B) & =  \frac 12 \gamma \alpha(\PP_\Q^2) \tau(\Sym^2 \PP_\Q^2) B \log B + O\left(B^{\frac 14 \gamma} \left( \log^+ B\right)^{\frac 12} B^{\frac 56} \right)\\ & = \frac 12  \gamma \alpha(\PP_\Q^2) \tau(\Sym^2 \PP_\Q^2) B \log B + O\left(B^{\frac56 + \frac \gamma 4} \left( \log B\right)^{\frac 12} \right).
\end{align*}
Using the cutoff $\gamma = \frac 23$ he obtains the desired result.
The main point of this example is that if there would be a cutoff $\gamma < \frac 23$, then the fields with discriminant between $B^\gamma$ and $B^{\frac{2}{3}}$ should contain no points and thus summing beyond $B^\gamma$ should not change the asymptotic.  
This is contradicted by the formula above; changing $\gamma$ does alter the constant of the main term. Note that there is no such logical inconsistency in choosing $\gamma > \frac{2}{3}$ since then the error term will dominate the main term.
Hence $\frac 23$ is the optimal anticanonical cutoff for $\PP^2_\Q$.
\end{example}

We have thus seen the relevance of the optimal cutoff in the study of Manin's conjecture of $\Sym^2 X$, and the result above shows that for the optimal anticanonical cutoff $\gamma=\frac23$ for $X= \mathbb P^2_{\mathbb Q}$ we have the relation
\[
\frac12 \gamma \alpha(\PP_\Q^2) = \alpha(\Hilb^2 \PP_\Q^2).
\]
One might na{\"i}vely expect something similar for all surfaces $X$.
However, the relation between the cutoff and the $\alpha$-constants of the surface and it's Hilbert scheme can be much more complicated in the presence of fibrations.

\begin{remark}\label{rem:p1p1}
We can even use the trick in Example~\ref{ex:Manin implies optimal cutoff} to prove that the $\calO(2)$-cutoff $1$ for $\PP^1_\Q$ is the optimal cutoff, though Manin's conjecture for $\Sym^2 \PP^1_\Q$ has not been established by summing the individual terms $N'_K(\PP^1,B)$, see \cite[\textsection 9]{MR1330740}.
Le Rudulier \cite{lerudulier-thesis} implicitly used the induced cutoff $\frac12$ for the pristine point on $\PP^1_\Q \times \PP^1_\Q$ to successfully prove Manin's conjecture for $\Sym^2(\PP^1_\Q \times \PP^1_\Q)$. This proves that $\frac12$ is the optimal cutoff for the pristine points on the product, and in turn $1$ must be the optimal anticanonical cutoff for $\PP^1_\Q$.

However, for this surface the relation between the $\alpha$ constants of $X$ and $\Hilb^2 X$ is more mysterious due to the presence of fibrations. Implicitly, Le Rudulier arrives at the following equation for $X= \PP_\Q^1\times \PP_\Q^1$
\[
\alpha\left(\Hilb^2(\PP_\Q^1\times\PP_\Q^1)\right) = \frac12 \alpha(\PP_\Q^1 \times \PP_\Q^1)\gamma \left(1-\frac 1{\rho(\PP_\Q^1 \times \PP_\Q^1)} \left(\frac1{\gamma_1} + \frac1{\gamma_2} \right) \gamma \right),
\]
where $\gamma=\frac12$ is the optimal cutoff for pristine points on $\PP_\Q^1\times \PP_\Q^1$, $\gamma_1=\gamma_2=1$ are the optimal cutoffs for the factors.

We have been able to find a similar expression for $\alpha (\Hilb^2 X)$ for the Hirzebruch surfaces $\mathbb F_1$ and $\mathbb F_2$, but the relation for a general del Pezzo surface seems to be much more complicated.
\end{remark}

\section{Lattices over general rings of integers}\label{s:lattices}

In this section, we collect basic properties of lattices over the ring of integers $\calO_F$ of a number field $F$ following \cite[Ch.\ III]{corpslocal}. In particular, we develop a theory of the geometry of numbers over number fields, estimating the number of lattice points $\Lambda \cap \calS$ of an $\calO_F$-lattice $\Lambda$ in a bounded region $\calS$. The main term is given by some notion of the size of the counting region divided by the determinant of the lattice, and the error term in terms of the successive minima of the $\calO_F$-lattice. The advantage of viewing $\Lambda$ as an $\calO_F$-lattice, as opposed to its $\Z$-structure, is that the dimension of the lattice and hence the number of successive minima are reduced due to the additional symmetry. This yields an improved error term in comparison to the classical Davenport's lemma.

\subsection{The algebra of lattices}
An $\mathbf{\calO}_F$\textbf{-lattice} $\Lambda$ is a finitely generated torsion-free $\calO_F$-module. Then $\Lambda$ is an $\calO_F$-submodule of the $F$-vector space $V = \Lambda \otimes_{\calO_F} F$, so that $\Lambda$ spans $V$ over $F$. We will call $\dim_F V$ the \textbf{rank} of $\Lambda$.

Consider two $\calO_F$-lattices $\Lambda' \subseteq \Lambda$ for which $\Lambda/\Lambda'$ is an $\calO_F$-module of finite length, then there is a composition series
\[
0 \subsetneq M_1 \subsetneq \cdots \subsetneq M_r = \Lambda/\Lambda', 
\]
with simple composition factors of the form $M_i/M_{i-1} \cong \calO_F/\mathfrak p_i$ where $\mathfrak p_i$ is a prime ideal.
We define the \textbf{index ideal} of the two lattices as
\[
(\Lambda : \Lambda') = \prod_{i=1}^r \mathfrak{p}_i.
\]
By the Jordan--H\"older theorem this ideal is well defined, see \cite[I\textsection 5]{corpslocal}.
Moreover, given any two $\calO_F$-lattices $\Lambda$ and $\Lambda'$ which span the same $F$-vector space, their index ideal is given by the quotient
\[
(\Lambda: \Lambda') := \frac{(\Lambda:\Lambda \cap \Lambda')}{(\Lambda' : \Lambda \cap \Lambda')}.
\]
We note that this definition is independent of any choices and does not depend on a basis. Furthermore, the index ideal also does not depend on the base field, in the sense that for a field extension $F'/F$ and $\calO_{F'}$-lattices $\Lambda_1, \Lambda_2 \subseteq V$ we have \cite[Prop.\ I.12]{corpslocal}
\[
N_{F'/F}(\Lambda_1 \colon \Lambda_2)_{F'}  = (\Lambda_1 \colon \Lambda_2)_F.
\]
Therefore we will write $(\Lambda_1 \colon \Lambda_2)$ if the underlying field is unambiguous.

\subsection{The shape of lattices}

We describe the shape of the $\calO_F$-lattice $\Lambda$ in relation to a collection of sets $B_v \subseteq V_v$ indexed by the archimedean places $v$ of $F$. We assume that each $B_v \subseteq V_v$ is open, convex, bounded, and rotationally invariant (in the sense that $uB_v = B_v$ for all $|u|_v = 1$ in $F_v$).

We embed $V$ (and hence $\Lambda$) diagonally in $V_{F, \infty} := \prod_{v\mid \infty} V_v$ and we write $B_{F, \infty}:=\prod_{v \mid \infty} B_v$. We note that  $V_{F, \infty} = V_{\Q, \infty}$, and so in what follows we denote this by $V_\infty$. 
For each $1 \leq i \leq \dim_F V$, the \textbf{$i$-th successive minimum} $\lambda_i$ of $\Lambda$ (with respect to $B_{F, \infty}$) is defined as
\[
\lambda_i := \inf \{ \lambda \in \R_{> 0} \colon \Lambda \cap \lambda B_{F, \infty} \text{ contains } i \text{ $F$-linearly independent vectors} \}.
\]

Suppose further that $V$ is endowed with a non-degenerate $F$-bilinear
pairing $T \colon V \times V \to F$.  
Then the \textbf{discriminant} $\Delta_{\Lambda/F}$ of a lattice $\Lambda$ (with respect to the pairing $T$) is uniquely defined by the following two properties \cite[Prop.\ III.3 and 5]{corpslocal}.
\begin{itemize}
    \item[(1)] If $\Lambda$ is a free $\calO_F$-module with basis $v_i$, then $\Delta_{\Lambda/F}$ is the principal fractional ideal generated by $\det (T(v_i,v_j)_{i,j}) \in F$.
    \item[(2)] For two lattices $\Lambda_i$ in the same vector space we have $\Delta_{\Lambda_1/F} = (\Lambda_2: \Lambda_1)^2 \Delta_{\Lambda_2/F}$.
\end{itemize}

Alternatively, by \cite[Prop.\ III.10]{corpslocal} discriminants are determined locally so one can compute the local discriminants using the above properties then take the product to compute the global discriminant.  Finally, discriminants can also be defined via the dual lattice
\[
\Lambda^\vee := \{x \in V\colon \forall y \in \Lambda,T(x,y)\in \calO_F\},
\]
and then
$\Delta_{\Lambda/F} = (\Lambda^\vee \colon \Lambda)$ \cite[Prop.\ III.4]{corpslocal}. For details, we refer the reader to \cite{corpslocal} and \cite{reiner2003maximal}. 

\subsection{The covolume of a lattice}\label{ss:covolume}

To measure the volume of a lattice with respect to a chosen pairing $T\colon V \times V \to F$, which is usually suppressed in the notation, and relate the successive minima to the discriminant, we first define the \textbf{determinant} of a lattice $\Lambda$ as $${\det}_F \, \Lambda := \sqrt{N_{F/\Q}(\Delta_{\Lambda/F})},$$ or simply $\det \Lambda$ when the base field is understood.

Another invariant is given by Haar measures on $V_v = V \otimes_F F_v$ which is a locally compact group as well as a finite-dimensional $F_v$-vector space for each archimedean place $v$ of $F$. We normalise the Haar measure via the pairing $T$ in the following way. Fix an isomorphism from $V_v$ to its Pontryagin dual
$V_v \to \hat V_v, \eta \mapsto \chi_\eta := \left( x \mapsto \exp(2\pi i \Tr_{F_v/\R} T(x,\eta))\right)$ and let $\mu_v$ be the unique Haar measure which identifies with its dual measure $\hat{\mu}_v$ under this isomorphism. If $V=A$ is a separable $F$-algebra with pairing given by the trace pairing, as it will be in our applications, then under the natural identification $A_v := \prod_{w\mid v} A_w$ is a product of copies of $\R$ and $\C$ as $F_v$-algebras, the measure $\mu_v$ on $A_v$ agrees with the product measure of our normalised Haar measures on $\R$ and $\C$.

We now show that the determinant and the measure on $V_{\infty}$ have a similar dependence on the pairing. 
By our convention on normalisation of norms on local fields, we have $|\alpha|_v^{[F_v \colon \R]} = \mu_v(\alpha S)/\mu_v(S)$ for any non-zero measure set $S\subseteq V_v$.

\begin{prop}
\label{prop:indpair}
On $V_{\infty}$ we consider the measure $\mu_{F, \infty} := \prod_{v \mid \infty} \mu_v$. The ratio 
\[
\frac{\mu_{F, \infty}}{\det_F \Lambda}
\]
is independent of the chosen pairing.
\end{prop}

\begin{proof}
    Let $T$ and $T'$ be two non-degenerate bilinear pairings on $V$, we will write $\mu$ and $\Delta$, and $\mu'$ and $\Delta'$ for the measures and discriminant defined using these respective pairings.
    Define the automorphism $\phi$ as the composition $V \xrightarrow{T} V^\vee \xleftarrow{T'} V$. 
    We will first see how $\det_F \Lambda$ varies under $\phi$ then show that $\mu_{F, \infty}$ changes by the same factor.

    We start by proving $\Delta_{\Lambda/F} = (\det \phi) \Delta'_{\Lambda/F}$. By the defining property (2) of determinants it suffices to prove this for a single lattice spanning $V$, since the index ideal is independent of the pairing. So suppose that $\Lambda$ is free with basis $(e_i)_{i=1}^n$ then by property (1) of determinants we find 
    \[
    \det( T(e_i, e_j)) = \det( T'(e_i, \phi e_j)) = \det \phi \det( T'(e_i,e_j)).
    \]
    We conclude $\det \Lambda = \sqrt{N_{F/\Q}(\Delta_{\Lambda/F})} = \sqrt{N_{F/\Q}(\det \phi)N_{F/\Q}(\Delta'_{\Lambda/F})} = \sqrt{N_{F/\Q}(\det \phi)}\det' \Lambda$.

    We now turn to understanding how $\mu_{F, \infty}$ varies under $\phi$. Let $v$ be an archimedean place and
    let $\mu_v$ be the self-dual measure on $V_v$ under the identification $V_v \to \hat V_v$ above using the pairing $T$. Using the measure $\mu_v$ we compute the double Fourier transform of a suitable function $f$ on $V_v$ with respect to the two identifications $V \to \hat V$ coming from the two pairings $T$ and $T'$ and find 
    \[
    \hat{\hat f}' = |\det \phi|^{[F_v \colon \R]}_v \hat{\hat f}.
    \]
    This implies that the self-dual measures with respect to the pairings $T$ and $T'$ are related by $\mu'_v := \mu_v |\det \phi|^{-\frac{[F_v \colon \R]}2}_v$. We therefore have
    \[
    \prod_{v\mid \infty} \mu_v = \prod_{v\mid \infty} |\det \phi|_v^\frac{[F_v \colon \R]}2 \mu'_v = N_{F/\Q}(\det \phi)^{\frac12} \prod_{v\mid \infty} \mu'_v,\qedhere
    \]
    which concludes the proof.
\end{proof}

\begin{rmk}
    Let $F$ be a number field of degree $d$, and identify $F_\infty$ with $\R^d$ as $\R$-algebras using the identification $\C\cong \R^2$ coming from the basis $\{1,i\}$. The trace pairing on $F_\infty$ and the standard pairing on $\R^d$ differ by a factor $2^s$, where $s$ is the number of complex places of $F$. Under these pairings $\R^d$ is endowed with the Lebesque measure, and $F_\infty$ with $2^s$ times the Lebesque measure. This explains why the determinant of $\calO_F$ in $\R^d$ equals $2^{-s} \sqrt{|\Delta_{F/\Q}|}$.
\end{rmk}

For lattices in the $F$-vector space $F^n$ with the corresponding standard pairing we find that $\sqrt{N_{F/\Q}(\Delta_{\Lambda/F})}= [\calO_F^n \colon \Lambda]$ since $\Delta_{\calO_F^n/F} = 1$. This establishes the following result for Serre's more general notion of lattice, which was proven in the case $V=F^n$ independently by Bombieri--Vaaler \cite[Thm.\ 3\&6]{BV} and McFeat \cite{McFeat}. 

\begin{prop}\label{prop:prod suc min}
    For each $v\mid\infty$ let $B_v \subseteq V_v$ be open, convex, bounded, and rotationally invariant. Then
    \[
    (\lambda_1\cdots \lambda_n)^d \asymp_{F,n} \frac{\det_F \Lambda}{\prod_{v|\infty} \mu_v(B_v)},
    \]
    for $n = \rank \Lambda$ and $d = [F \colon \Q]$. 
\end{prop}

For the lower bound it suffices to assume the regions $B_v$ are just symmetric (i.e. $B_v = -B_v$); the rotationally invariance (i.e. $B_v = uB_v$ for all $|u|_v=1$ in $F_v$) is only required for the upper bound.

\begin{proof}
    Let us fix an $F$-isomorphism $V=F^n$, then $\det_F \Lambda= [\calO_F^n \colon \Lambda]$. The upper bound on $(\lambda_1 \cdots \lambda_n)^d$  follows from \cite[Thm.\ 3]{BV}, and the lower bound from \cite[Thm.\ 6]{BV}. These show that the left hand side is asymptotically equal to an infinite product given by $\prod_{v \mid \infty} \mu_v(B_v) \prod_{v \nmid \infty} \mu_v(\Lambda_v)$.  We now interpret $\prod_{v \nmid \infty} \mu_v(\Lambda_v)$ as the reciprocal index $[\calO_F^n \colon \Lambda]^{-1}$ as in the proof of the Corollary to Theorem~5 in \cite{broberg}. Note here that the implied pairing on $V$ comes from the canonical one on $F^n$ but we conclude using Proposition~\ref{prop:indpair}.
\end{proof}

\subsection{Lattices over different base rings}

A lattice over the ring of integers of a number field is naturally a lattice over a smaller ring of integers. In this section, we will study an $\calO_F$-lattice under its $\Z$-structure as this is the relevant case for our later work. However, the results contained in this section hold for general separable extensions. We write $V_\Q$ for $V$ as a $\Q$-vector space, and as $F/\Q$ is separable $T_\Q := \Tr_{F/\Q} \circ T$ defines a non-degenerate $\Q$-bilinear pairing on $V_\Q$.

The following statement generalises a well-known property for rings of integers which was used in the proof of Proposition~\ref{prop:discs under base change}.

\begin{prop}\label{prop:disc under restriction}
    For a field extension $F/\Q$ and an $\calO_F$-lattice $\Lambda$ we have
    \[
    \Delta_{\Lambda/\Q} = N_{F/\Q}(\Delta_{\Lambda/F}) \Delta_{F/\Q}^{\rank \Lambda}.
    \]
\end{prop}

\begin{proof}
    The $\calO_F$-lattice $\Lambda$ naturally lives in the $F$-vector space $V$, and the $\Z$-lattice $\Lambda_{\Z}$ naturally lives in $V_\Q$. We will compare the dual of $\Lambda$ in $V$ and $\Lambda_\Z$ in $V_\Q$.

    We have
    \begin{align*}
    (\Lambda_\Z)^\vee & := \{x \in V_\Q \colon \forall y \in \Lambda_{\Z},T_\Q(x,y)\in \Z\}\\
    & = \{x \in V \colon \forall y \in \Lambda,T(x,y)\in \delta_{F/\Q}^{-1}\}\\
    & =  \delta_{F/\Q}^{-1} \Lambda^\vee,
    \end{align*}
    where $\delta_{F/\Q}^{-1}$ is the codifferent of $F/\Q$, which consists of the elements of $F$ whose integral multiples have integral traces. Now we conclude
    \begin{align*}
    \Delta_{\Lambda/\Q} & = ((\Lambda_\Z)^\vee \colon \Lambda_\Z)_\Q = N_{F/\Q}(\delta_{F/\Q}^{-1} \Lambda^\vee \colon \Lambda)_F\\
    &= N_{F/\Q}(\delta_{F/\Q}^{-1})^{\rank \Lambda} N_{F/\Q}(\Lambda^\vee \colon \Lambda) \\
    & = \Delta_{F/\Q}^{\rank \Lambda}N_{F/\Q}(\Delta_{\Lambda/F}). \qedhere
    \end{align*}
\end{proof}

Recall that the pairing $T$ induces a particular measure $\mu_{F, \infty}$ on $V_{\infty}= \prod_{v \mid \infty} V_v$, and $T_\Z$ induces a measure $\mu_{\Q, \infty}$ on $V_{\infty} = V \otimes_\Q \R$. These spaces are canonically identified and the following result compares the two measures.

\begin{prop}\label{prop:ratio comparison}
    The measures $\mu_{\Q, \infty}$ and $\mu_{F, \infty}$ on $V_{\infty}$ are the same.
\end{prop}

\begin{proof}
    The pairing $T$ on $V$ induces an identification $V_v \to \hat V_v$ for all archimedean places $v$ for $F$. These give an isomorphism $\prod_{v\mid \infty} V_v \xrightarrow{\cong} \prod_{v\mid \infty} \hat V_v$, hence $V_{\infty} \xrightarrow{\cong} \hat V_{\infty}$.
    By definition of the pairing $T_\Q$ on $V_\Q$, we see that the induced identification $V_{\Q,\infty} \to \hat V_{\Q,\infty}$ agrees with the one above. Hence, the induced measures agree as well.
\end{proof}

Thus, we drop the field from the notation and just write $\mu_\infty$ instead.

As an application, the generalised Davenport's lemma in Proposition~\ref{prop:generalised Davenport} approximates the number of lattice points of the $\calO_F$-lattice $\Lambda$ in a bounded region $\calS \subseteq V_{\infty}$. Its main term
\[
\frac{\mu_{ \infty}(\calS)}{|\Delta_{F/\Q}|^{\frac{\rank_F \Lambda}2}\det_F \Lambda}.
\]
 matches the main term obtained from $\Lambda$ viewed as a $\Z$-lattice, by Proposition~\ref{prop:disc under restriction} and Proposition~\ref{prop:ratio comparison}.

On the other hand, the error term in Davenport's lemma depends on the successive minima, which changes depending on how we view $\Lambda$. 
Indeed, denoting the successive minima of $\Lambda$ by $\lambda_{F,i}$ and $\lambda_{\Q,i}$ for $\Lambda_\Z$, we first note that $\rank_\Z \Lambda = [F \colon \Q] \rank_{\calO_F} \Lambda$, so the number of successive minima differs in these two cases. Also, the $\Z$-minima depend on a region $B_{\Q, \infty} \subseteq V_\infty$ and the $F$-minima come from regions $B_{F,v} \subseteq V_v$. So the successive minima are a priori completely different when viewing the same lattice over both $\Z$ and $\calO_F$. 

In contrast, the combination of Propositions~\ref{prop:prod suc min} and \ref{prop:ratio comparison} shows that their product remains the same. Hence, information about $\Lambda$ which could be spread out over $n$ successive minima is instead consolidated into the remaining $\frac{n}{[F \colon \Q]}$ minima. That is, over $F$, there will be fewer factors to control, but each will have a stronger dependence on $\Lambda$, yielding a stronger error term in Proposition~\ref{prop:generalised Davenport} than when using just the $\Z$-structure.

\begin{prop}\label{prop:comparing succ min}
    Let $F$ be a number field of degree $d$. For an $\calO_F$-lattice $\Lambda \subseteq V$ of rank $n$, suppose that the $n$ successive minima $\lambda_{F,i}$ are computed with respect to open, convex, bounded and rotationally invariant $B_v \subseteq V_v$.

    Let $\lambda_{\Q,i}$ for $1 \leq i \leq nd$ be the successive minima of $\Lambda_\Z$ computed with respect to $B_\infty := \prod_{v \mid \infty} B_v \subseteq V_\infty$. Then they agree with 
    \[
    \underbrace{\lambda_{F,1}, \lambda_{F,1}, \ldots, \lambda_{F,1}}_{d \text{ times}}, \underbrace{\lambda_{F,2}, \lambda_{F,2}, \ldots, \lambda_{F,2}}_{d \text{ times}}, \ldots, \underbrace{\lambda_{F,n}, \lambda_{F,n}, \ldots, \lambda_{F,n}}_{d \text{ times}}
    \]
    up to an absolute constant only depending on $F$ and $n$. That is, $\lambda_{\Q, (j-1)d + k} \asymp_{n, F} \lambda_{F, j}$ for all $1 \leq j \leq n$ and $1 \leq k\leq d$.
\end{prop}

We remark that under this choice of unit balls $B_v$ and $B_\infty$ we trivially have $\lambda_{F,1} = \lambda_{\Q,1}$, but a precise relation for the higher successive minima without the implicit constant depending on $F$ and $n$ seems not to exist in general.

\begin{proof}
    Let $x_1,\dots,x_n$ be $F$-linearly independent elements in $\Lambda \subseteq A$ such that $x_j \in \lambda_{F,j}B_v$ for all $v$.
    Fix an integral basis $e_1,\dots,e_d$ for $\calO_F$. Now $\Lambda_\Z$ is the free $\Z$-lattice $\Lambda_\Z \subseteq A$ spanned by the $e_k x_j$. For some constant $C$ depending on $F$ only, we have $|e_k|_v \leq C$ for all $k$ and $v$.
We then can diagonally embed $e_k x_j$ into $\prod_{v \mid \infty} (\lambda_{F,j} C B_v) = C \lambda_{F,j} B_{\infty}$ using the identification of real vector spaces $\prod_{v \mid \infty} V_v = V_\infty$. This provides an upper bound $\lambda_{\Q,(j-1)d+k} \ll_F \lambda_{F,j}$ for all $1 \leq j \leq n$ and $1 \leq k\leq d$.

By Proposition~\ref{prop:prod suc min} we have
    \[
    \frac{\det_\Q \Lambda}{\mu_\infty(B_\infty)} \ll_{dn} \lambda_{\Q,1}\cdots \lambda_{\Q,dn} \ll_{F,n} (\lambda_{F,1}\cdots \lambda_{F,n})^d \ll_{F,n} \frac{\det_F \Lambda}{\mu_\infty(\prod_{v \mid \infty} B_v)}.
    \]
    By Proposition~\ref{prop:disc under restriction} and the choice of unit ball for $\Lambda_\Z$ we see that the lower and upper bound agree up to a factor depending only on $F$ and $n$, which establishes
    \[
    \lambda_{\Q,1}\cdots \lambda_{\Q,dn} \asymp_{F,n} (\lambda_{F,1}\cdots \lambda_{F,n})^d.
    \]
    As we can upper bound each of the $dn$ factors on the left by the corresponding factor on the right, we obtain the reverse bound up to a factor only depending on $F$ and $n$ from the commensurability of the product. This proves the claim.
\end{proof}

\subsection{Ideals as lattices}
Now we specialise to ideals of number fields or separable algebras, which can be viewed as lattices over the ring of integers of any subfield. We establish some general facts about these  lattices. Let $A$ be a separable algebra of rank $n$ over $F$, with the nondegenerate trace pairing $T:=\text{Tr}_{A/F}$.

\begin{lemma}\label{lem:determinant of an ideal}
    A fractional ideal $\frakd \subseteq A$ is a rank $n$ lattice over $\calO_F$ with determinant $\det_F \frakd = N_{A/\Q}(\frakd)\vert\Delta_{A/\Q}\vert^{1/2}$.
\end{lemma}

\begin{proof}
    A fractional ideal of $A$ is a finitely generated $\calO_F$-submodule of $A$ and satisfies $\frakd \otimes_{\calO_F} F = A$. Hence $\frakd$ is a lattice of rank $\dim_FA =n$. We compute its discriminant. From $\Delta_{\frakd/F}=(\calO_A:\frakd)^2\Delta_{\calO_A/F}$ it suffices to compute the discriminant of $\calO_A$ and the index ideal $(\calO_A:\frakd)$. However, by definition, $(\calO_A:\frakd) = N_{A/F}(\frakd)$ and by \cite[Ch.\ III, Prop.\ 6]{corpslocal} we have $\Delta_{\calO_A/F} = \Delta_{A/F}$. In this last equation, the first $\Delta$ refers to the discriminant of the lattice and the second to the discriminant of the extension as is standard in algebraic number theory. In conclusion, we  see that
    \[
    {\det}_F \frakd = \sqrt{N_{F/\Q}(\Delta_{\frakd/F})}
    = \sqrt{N_{F/\Q}(N_{A/F}(\frakd))^2 \Delta_{A/F}}
    = N_{A/\Q}(\frakd) \sqrt{\Delta_{A/F}},
    \] as claimed.
\end{proof}

To study the successive minima of an ideal $\fraka \subseteq A$ we work with the archimedean places of $A$, which are the ring homomorphisms $A \to \C$ up to conjugation. We say that an archimedean place $w$ extends, or divides, a place $v$ of $F$ if the homomorphism factors as $F \to A \to \C$. If we write $A= \oplus A_i$ as the product of fields, one shows that a place $w$ of $A$ correspond to the places of the $A_i$ through the factorisation $A \to A_i \to \C$. In particular, we have that $A_v := A \otimes F_v$ is naturally isomorphic to $\prod_{w \mid v} A_w$. The $F_v$-algebra $A_v := A \otimes_F F_v$ for an archimedean place $v$ of $F$ is naturally isomorphic to $A_v = \prod_{w \mid v} A_w$ and can be identified as $F_v$-algebras with $\R^{r_v} \oplus \C^{s_v}$ where $r_v$ and $s_v$ are the number of real and complex archimedean places of $A$ dividing $v$. By choice of normalisation, the local measure $\mu_v$ on $A_v$ is then just the product measure of the usual Haar measures on $\R$ and $\C$. For a region $\calS \subseteq A_\infty$, we write $\vol(\calS) := \mu_\infty(\calS).$

From now on we write $B_v := \{(x_w)_{w \mid v} \colon \max_{w\mid v} \vert x_w \vert_w \leq 1\}$. We compute the successive minima of an ideal with respect to these unit balls or small variants, usually denoted $B'_v$, thereof. We have made this choice such that $B_\infty := \prod_{v \mid \infty} B_v \subseteq A_\infty$ is the unit ball under the maximum norm $\max_{w\mid \infty} \vert x_w \vert_w$. This independence of the base field will simplify notation and some arguments when comparing an $\calO_F$-lattice to itself as a $\Z$-lattice.

\begin{remark}
    The reader is warned that this choice of unit balls is not typical! Usually one considers unit balls in $A_v$ under the Euclidean norm, that is, $\{(x_w)_{w \mid v} : \sum_{w\mid v} \vert x_w \vert_w^2 \leq 1\}.$
    However, the two choices are commensurable, indeed
    \[
    \left\{(x_w)_{w \mid v} : \max_{w\mid v} \vert x_w \vert_w \leq \frac{1}{\sqrt{n}}\right\} \subseteq
    \left\{(x_w)_{w \mid v} : \sum_{w\mid v} \vert x_w \vert_w^2 \leq 1\right\} \subseteq  \left\{(x_w)_{w \mid v} : \max_{w\mid v} \vert x_w \vert_w \leq 1\right\},
    \]
    and yield the same successive minima up to an absolute constant which depends only on the rank $n$. The choice for the Euclidean norm would not have changed the application, but would have complicated some arguments later in the paper. The ability to designate successive minima in the way most suited to the problem at hand is yet another advantage of setting up our lattices in the level of generality which we have in this section.
\end{remark}

We now turn to our chief tool for counting lattice points in $A$ in a bounded region, which can be seen as a version of Davenport's lemma for $\calO_F$-lattices and as a generalisation of \cite[Lem.\ 3.3]{GH}. Indeed, we are able to produce an asymptotic for the number of points of an $\mathcal{O}_F$-lattice lying in some bounded region, without requiring that the lattice lives inside $F^{{\rm{rk}} \Lambda}$. 
Some constraint on the shape of the region is required and we here make a natural generalisation of the standard condition that the boundary of the region can be parametrised by a controlled number of Lipschitz maps with bounded Lipschitz constant.

\begin{defn}
    Given a map $\phi:[0,1)^m \rightarrow A_\infty$ for some $m \leq \deg_\R A_\infty$, we say that $\phi$ is $(m,R)$-Lipschitz with respect to $B'_\infty$ if
    \[
    \phi(\mathbf x) - \phi (\mathbf y) \in R\vert \mathbf x - \mathbf y \vert B'_\infty,
    \] for any $\mathbf x, \mathbf y \in [0,1)^m$.
\end{defn}
Recall that $B'_\infty$ is a variant of the unit ball $B_\infty \subseteq A_\infty$ under maximum norm. We recall the familiar situation of $\Z$-lattices where the unit ball is instead given by the unit ball in $\R^{dn}$ under the Euclidean distance and $A_\infty$ is identified  with $\R^d$.

\begin{defn}
    A subset $S \subseteq A_\infty$ is said to be in $\text{Lip}_{A}(R, N)$ (with respect to $B'_\infty$) if there exist $N$ maps $\phi$ which are each $(nd-1,R)$-Lipschitz (with respect to $B'_\infty$) and whose images together cover $S$.
\end{defn}

This definition has been made to introduce a notion of Lipschitz-parametrisability which mirrors the $\text{Lip}(n,1,L,M)$ constraint in Widmer's work \cite{MR2645051}, but also takes account of the different structure of an $\calO_F$-lattice and that makes use of the notion of length induced by the choice of local unit balls $B'_v$. Indeed when one considers the image of $A_\infty$ in $\R^{dn}$, the two notions are the same except that the Lipschitz parameter may differ by a constant depending at most on $F$ and $n$.

\begin{prop}\label{prop:generalised Davenport}
Let $F$ be a number field, $\Lambda$ an $\calO_F$-lattice in $A$ of rank $n$, and $\calS \subseteq A_\infty$ a bounded counting region. Let $\tau =(\tau_w)_w \colon A_\infty \to A_\infty$ be an $F_\infty$-linear map with $\det_\R \tau = 1$, and write $\lambda_1,\dots,\lambda_n$ for the successive minima of $\Lambda$ with respect to the local regions $B'_v := \tau_v^{-1} B_v \subseteq A_v$.

Suppose that the boundary of $\calS \subseteq A_\infty$ is in $\text{Lip}_A(R,N)$ with respect to $B'_\infty$
then there exists a constant $C_{F,n}$, depending at most on $F$ and $n$, such that  
    \[
    \left \vert
    \#(\Lambda \cap \calS) - \frac{\vol(\calS)}{|\Delta_{F/\Q}|^{\frac{\rank_F \Lambda}2}\det_F \Lambda}
    \right \vert
    \leq
    C_{F,n}N \left(\frac{R^{nd-1}}{(\lambda_1\cdots \lambda_{n-1})^d \lambda_n^{d-1}} + \sum_{i=0}^{n-1} \frac{R^{id}}{(\lambda_1\cdots \lambda_i)^d} \right),
    \]
    where we take $\lambda_1\cdots \lambda_i = 1$ when when $i=0$.
\end{prop}

Notice that $\#(\Lambda \cap \calS)$ does not depend on the base ring of $\Lambda$. That is, in estimating the cardinality of this set, we are free to pick the lattice structure on $\Lambda$. The result says we get a counting result with generally acceptable error terms if we pick the unit balls $B'_v$ with a similar relative shape as the counting region.

\begin{remark}
    One immediate consequence of the result is the useful upper bound
    \[
    \#(\Lambda \cap S) \ll \sum_{i=0}^n \frac{R^{id}}{(\lambda_1 \cdots \lambda_i)^d},
    \] which removes the top error term since it is dominated by the size of the main term.
\end{remark}

\begin{proof}[Proof of Proposition \ref{prop:generalised Davenport}]
    We naturally identify $A_\infty$ with $\R^{nd}$. Our assumption on the boundary of $\calS$ means that the region  $\tau \calS$ is contained within a Euclidean ball of radius $R$ in $\R^{nd}$. Applying a traditional lattice point counting result (such as \cite[Lem.\ 5.4]{MR2645051}) to the torsion-free $\Z$-module $\tau \Lambda_\Z \subseteq \R^{nd}$, we have that
    $$
    \begin{aligned}
    \left \vert
    \#(\tau \Lambda \cap \tau \calS) - \frac{\vol_{\R^{nd}}(\tau \calS)}{\det \tau\Lambda_\Z} \right \vert
    &\leq C(d,n)
    \sum_{i=0}^{nd-1} \frac{R^{i}}{\lambda_{\Q,1}\cdots \lambda_{\Q,i}}\\
    &\leq C'(F,n)\sum_{i=0}^{n-1}\sum_{k=0}^{d-1} \frac{R^{id+k}}{(\lambda_{1}\cdots \lambda_{i})^d\lambda_{i+1}^{k}},
    \end{aligned}
    $$
    where we applied Proposition~\ref{prop:comparing succ min} in the last step. 
    We may simplify the main term since $\det \tau \Lambda_\Z = \det \Lambda_\Z$ and $\vol_{\R^{nd}}(\tau S) = \mu_\infty(S)$.
    
    Finally, we observe that we can dominate each term $\frac{R^{id+k}}{(\lambda_{1}\cdots \lambda_{i})^d\lambda_{i+1}^{k}}$ independently of $k$ according to the size of $R$. Indeed if $R \geq \lambda_n$ then the term $\frac{R^{nd-1}}{(\lambda_1\cdots \lambda_{n-1})^d \lambda_n^{d-1}}$ dominates. If $\lambda_j > R \geq \lambda_{j-1}$ for some $2 \leq j  \leq n$, then the term $\frac{R^{(j-1)d}}{(\lambda_1\cdots \lambda_{j-1})^d}$ dominates. Finally, if $R <  \lambda_1$, then each term  $\frac{R^{id+k}}{(\lambda_{1}\cdots \lambda_{i})^d\lambda_{i+1}^{k}}$ is at most $O_{F,n}(1)$. This concludes the proof.
\end{proof}

In \S\ref{s:nonsplit} we will apply the counting result to lattices of rank at most $8$ over a fixed quadratic base field, so the implicit dependence on $n$ and $F$ is good enough for our applications.

\begin{rmk}
Traditionally, symmetries in lattices have played major roles in understanding structures of algebraic objects -- for example, the theory of cyclotomic fields and the theory complex multiplication. So it is not surprising that studying $\calO_F$-lattices, which implies extra symmetry over $\Z$-lattices, gives us additional information.

For example, the $\Z[i]$-lattices in $\C$ must be invariant under the rotation by $\pi/2$ (which corresponds to multiplication by $i$), so they must be square lattices. In particular, we get that $\lambda_{\Q,1} = \lambda_{\Q,2}$ when the successive minima are computed using a rotationally invariant region.

Thus, the theory of $\calO_F$-lattices should not be viewed as the simple generalisation of Davenport's lemma to number fields; rather, this should be seen as a strengthening of Davenport's lemma in the cases where the underlying lattice has extra symmetry. 
\end{rmk}

\section{A family of examples: non-split quadrics}\label{s:nonsplit}

Finally, we provide a new family of examples of symmetric squares of surfaces for which our strategy applies and on which we can explicitly count rational points in a cothin subset. The surfaces that we consider are of the form $X_d=\{ x^2 - dy^2 = zw\} \subseteq \PP_\Q^3$, where $d$ is a squarefree integer. The first key idea in our method is to use the fact that $X_d$ is the restriction of scalars of $\PP^1_{\Q(\sqrt d)}$ from $\Q(\sqrt d)$ to $\Q$.

\noindent\textbf{Notation.} We denote by $L = \Q(\sqrt d)$ the splitting field of $\Pic X_d$, and $K = \Q(\sqrt{\Delta})$ will denote the various quadratic extensions over which we consider the points of $X_d$ (which yield rational points of $\Sym^2 X_d$). Their compositum will be denoted by $M := KL = \Q(\sqrt d, \sqrt{\Delta})$.\\

Our principal aim in this section is to prove the following, which immediately gives assertion (a) in Theorem~\ref{thm:mainthm1}.

\begin{theorem}\label{thm:quadxd}
    Let $\calU \subseteq \text{Sym}^2X_d(\Q)$ denote the set of points which do not arise as a pair of $\Q$-points or a conjugate pair of $L$-points on $X_d$. Then for the absolute anticanonical height~$H^{(2)}$ from Proposition~\ref{prop:height on Sym} induced from the height $H$ defined in Definition~\ref{defn:heightonXd},
    \[
    \#\{ Q \in \calU \colon H^{(2)}(Q) \leq B\} = c_{\Hilb^2 X_d, \Q} B \log B + O_d\left(B(\log B)^{3/4}\right).
    \]
\end{theorem}

We remark that the points in $\calZ := \Sym^2 X_d(\Q) \setminus \calU$ contribute $\asymp B \log B$, but they need to be removed to obtain the correct leading constant. Indeed,
\[
\#\{P \in X_d(L) \colon H(P) \leq B^{\frac12}\} \asymp B \log B \quad \mbox{and} \quad \#\{P \in X_d(\Q) \colon H(P) \leq B\} \asymp B
\]
and from partial summation, it follows that the contributions to $\#\{ Q \in \Sym^2 X_d(\Q) \colon H(Q) \leq B\}$ from points of the form
$Q=[P_1,P_2]$ with $P_1,P_2 \in X_d(\Q)$, and of the form $Q=[P, \overline{P}]$ for $P \in X_d(L)$ are also of size $\asymp B\log B$. By Lemma~\ref{lem:kptsthin}, these points form a thin subset of $\text{Sym}^2 X_d(\Q)$.

For the remainder of the section, we focus on proving Theorem \ref{thm:quadxd}. In the first subsection, we will identify the quadric surface of $X_d$ with the restriction of scalars of $\PP^1_{L}$ from $L$ to $\Q$. This allows us to compute the various Peyre constant factors associated to $\Sym^2 X_d$ as well as the optimal cutoff of $X_d$. Moreover, in \textsection\ref{s:widmer} we identify the pure quadratic points on the surface $X_d$ with quartic points on $\PP_\Q^1$ defined over biquadratic extensions containing $L$. 
Our tool for counting these points comes from the theory of $\calO_L$-lattices which we have just developed in \textsection\ref{s:lattices}.
In \textsection\ref{s:widmerproof}, we apply this theory to produce an asymptotic for these quartic points on $\PP^1_\Q$, generalising Schmidt's~\cite{MR1330740} arguments and drawing ideas from Widmer \cite{MR2645051}. In particular, we extend all of Schmidt's approach to a quadratic base field, making very specific use of the structure of biquadratic points over a fixed quadratic subfield to gain an advantage over what can be achieved by considering just quartic points over $\Q$.

\subsection{Geometry of quadric surfaces}\label{s:geomxd}

The following result seems well known, but we add a proof for completeness.

\begin{lemma}\label{lem:XdisRes}
The surface $X_d/\mathbb Q$ is isomorphic to $\Res_{L/\mathbb Q} \mathbb P^1_{L}$.
\end{lemma}

\begin{proof}
We first construct an $L$-isomorphism
$$
\phi \colon \begin{array}{ccc}\PP^1_{L} \times \PP^1_{L} &\longrightarrow & (X_d)_{L}\\
([s:t], [u:v]) &\longmapsto &
 \left[\frac{sv + tu}{2}: \frac{sv - tu}{2\sqrt{d}}: tv: su\right].
 \end{array}
$$
which admits an inverse constructed from the two projections
\[
(X_d)_{L} \to \PP^1_{L}, 
[x:y:z:w] \mapsto 
[x \pm \sqrt{d}y:z] = [w:x \mp \sqrt{d}y].
\]
Note that the action of the non-trivial element of order two $\sigma \in \Gal(L/\Q)$ swaps the two factors, and hence $\phi$ respects the descent datum on $\PP^1_L \times \PP^1_L$, defined in Proposition~\ref{prop:weilrestriction}, and the effective descent datum on $(X_d)_L$. Hence $\phi$ descends along $L/\Q$ and yields an isomorphism $\Res_{L/\Q}\PP^1_L \xrightarrow{\isom} X_d$.
\end{proof}

From now on we identify $X_d$ with $\Res_{L/\Q} \PP^1_{L}$. Using this identification we can compute many geometric invariants of $X_d$ and its Hilbert scheme of two points directly from understanding the invariants of $\PP^1$ and $\PP^1 \times \PP^1$.

\begin{lemma}\label{lem:constsxd}
    For any non-square $d$, we have
    \begin{enumerate}
        \item[(a)] $\beta(X_d) = 1 = \beta( \Hilb^2 X_d)$;
        \item[(b)] $\rho(X_d) = 1$, $\rho(\Hilb^2 X_d) =2$;
        \item[(c)] $\alpha(X_d) = \frac{1}{2}$, $\alpha(\Hilb^2 X_d) = \frac{1}{8}$;
        \item[(d)] equality of the Peyre constants $c_{X_d} = c_{\PP^1,L}$, and
        \item[(e)] The optimal cutoff (in the sense of Definition~\ref{def:cutoff}) for $X_d$ is $\frac{1}{2}$.
    \end{enumerate}
\end{lemma}

\begin{proof}
    \begin{enumerate}
        \item[(a)] By \cite[Lem.\ 4.2]{MR3430268} we have $\beta(X_d) = \beta(\PP_L^1) = 1$. The second part follows from Proposition \ref{prop:hilbbasics}(c).
        \item[(b)] Geometrically $X_d$ is isomorphic to $\PP^1 \times \PP^1$. Hence $\Pic \Xbar_d = \Z H_1 \oplus \Z H_2$, corresponding to the two rulings on $\Xbar_d$. The Galois action factors through $\Gal(L/\Q)$ and is generated by the element exchanging $H_1$ and $H_2$. Since $d$ is non-square, $\Pic(\Xbar_d)^\Gamma$ is generated by $H := H_1 + H_2$. This establishes that the rank of $\Pic X_d$ is 1, and the rank of the Picard group of the Hilbert scheme is one greater by Proposition \ref{prop:hilbbasics}(b).
        \item[(c)] The effective cone of $\PP^1 \times \PP^1$ is spanned by $H_1$ and $H_2$, hence the Galois invariant part of the effective cone is spanned by $H = H_1+H_2$ (Lemma~\ref{lem:rankseq}). From $-K_X = 2H$ we deduce $\alpha(X_d) = \frac12$.
        From \cite[\textsection 3.1]{lerudulier-thesis} we see that the effective cone of $\Hilb^2(\PP^1 \times \PP^1)$ is spanned by $H^{[2]}_1-\Delta^{[2]}, H^{[2]}_2-\Delta^{[2]}$ and $\Delta^{[2]}$ (see \S\ref{ss:hilb2}), so the Galois invariant part is generated by $H^{[2]} - 2\Delta^{[2]}$. The dual cone is spanned by $h$ and $h+\frac12 \delta$, so the part given by $\langle \cdot, -K_X\rangle \leq 1$ equals $\{ah+b\delta \colon 0\leq a \leq \frac12, 0 \leq b \leq \frac14\}$. We conclude that $\alpha(\Hilb^2 X_d) = \frac{\alpha(X)}{4} = \frac 18$.
        \item[(d)] We can apply the results from \cite[\textsection 4.2]{MR3430268}, since $X_d = \Res_{L/\Q} \PP^1_L$.
        \item[(e)] Since 2 is an $\calO(1)$-cutoff for $\PP^1_{L}$ by Proposition \ref{ex:pncutoff}, $\frac{1}{2}$ is a cutoff for $X_d =\Res_{L/\Q} \PP^1_{L}$ by Proposition \ref{p:cutoffres}.
        Consider the conic $C \subseteq X_d$ given by $x^2-zw=0=y$ which satisfies $\deg \left.\omega_{X_d}\right|_C=4$, whose optimal cutoff is $\frac24=\frac12$. Since a cutoff for $X_d$ would also be a cutoff for $C$ we conclude that $\frac12$ is the optimal cutoff for $X_d$. \qedhere
    \end{enumerate}
\end{proof}

We note that for this surface, as for $\mathbb P^2$, we have the relation $\alpha(\Hilb^2 X) = \tfrac12 \alpha(X) \gamma$, where~$\gamma$ is the optimal anticanonical cutoff for~$X$. For the approach of counting the points from each quadratic field $K/\Q$ separately to work for a del Pezzo surface $X$ with $\rank \Pic X=1$, this relation between the $\alpha$-constants and the optimal cutoff will need to hold.

\subsection{Counting pure quadratic points}\label{s:widmer}
From the functorial definition of the  restriction of scalars $X_d = \Res_{L/\Q} \PP^1_{L}$, we have for any extension $K$ of $\Q$ the bijection $\nu_{\PP^1} \colon  \PP^1(K \otimes_{\Q} L ) \xrightarrow{\isom} X_d(K)$ and $\PP^1(L) \xrightarrow{\isom} X_d(\Q)$. For instance, instead of counting pure quadratic points in $X_d$ whose field of definition is $K = \Q(\sqrt{\Delta})$ for some $\Delta$, we may count the points of $\PP_\Q^1$  which are defined over the biquadratic extension $M := \Q(\sqrt{d}, \sqrt{\Delta})$ but which are not defined over~$L$. Let us define the heights that we are considering on either side of this identification between $X_d$ and $\Res_{L/\Q}\PP^1_L$. We will denote the absolute heights on $X_d$ and $\PP^1_L$ by $H$ and $\resH$, respectively. If they are decorated by the subscript $\calO(1)$, then they will denote the $\calO(1)$-heights; otherwise, they will denote the anticanonical heights on the appropriate varieties.

We start off with the na\"ive absolute $\calO(1)$-height on $\PP^1_L$ given by
\[
\resH_{\calO(1)} \colon \PP^1_{L}(\overline{L}) \to \R_{>0}, \quad [s:t] \mapsto \prod_{v \in \Omega_L} \max\{|s|_v,|t|_v\}^{\frac{[L_v \colon \Q_p]}{2}}
\]
and the anticanonical height $\resH := \resH_{\calO(1)}^2$. By Proposition~\ref{prop:height on Weil res} these induce heights on $\Res_{L/\Q}\PP^1_L$.

\begin{defn}\label{defn:heightonXd}
    Let $H_{\calO(1)}$ be the $\calO(1)$-height on $X_d$ induced from the $\calO(1)$-height $\resH_{\calO(1)}$ on $\PP^1_L$ under the identification $X_d = \Res_{L/\Q}\PP^1_L$. We will use the anticanonical height $H := H^2_{\calO(1)}$ on $X_d$.
\end{defn}

In particular, by Proposition~\ref{prop:height on Weil res}, for a point $P$ on $X_d$ and its corresponding point $Q$ on $\PP^1_{L}$ we have $H(P) = \resH(Q)^2$.

\begin{remark}
Note that $H_{\calO(1)}$ is an adelic height defined by the adelic metric of Proposition~\ref{prop:adelic metric on Weil res}(a), so we have $H_{\calO(1)}(P) = \prod_v H_{\calO(1),v}(P)$ for $P=[x\colon y \colon z \colon w] \in X_d(\Qbar)$. By construction we have
\[
H_{\calO(1),v}(P) = \prod_{\nu \mid v}\max\{|x+\sqrt dy|_{\nu},|x-\sqrt dy|_{\nu}, |w|_{\nu}, |z|_{\nu}\}^\frac{[L_{\nu}:\Q_p]}{[L:\Q]}
\]
for $x,y,z,w \in L$. We can simplify these local heights easily in many cases.
\begin{enumerate}
    \item At all primes but the finitely many $p \mid 2d$ we have
    \[
    H_{\calO(1),p}(P) = \max\{|x|_{p},|y|_{p}, |w|_{p}, |z|_{p}\},
    \]
    as it should be for all but finitely many places, as $H_{\calO(1)}$ is an adelic height.
    \item If $p \mid d$ is odd, then we find that
    \[
    H_{\calO(1),p}(P) = \max\{|x|_p,|d|_p^{\frac12}|y|_p,|z|_p,|w|_p\}.
    \]
\end{enumerate}
\end{remark}

Instead of counting $K$-points on $X_d$ we count $M$-points on $\PP^1_L$. We define
\[
Z(\PP^1, K,d,Y):=\mathcal\#\left\{P \in \PP^1(M) \, : \, \Q\left(\sqrt{d},P\right)=M,\, \resH_{\calO(1)}(P)\leqslant Y\right\}.
\]
Providing a suitably uniform count of such quartic points is the subject of the following theorem and the major ingredient in the proof of Manin's conjecture for $\Sym^2 X_d$.

\begin{theorem}\label{thm:widmer+}
    For any quadratic field $K/\Q$ disjoint from $L$, any $\epsilon>0$ and any $Y \geq 1$, we have
        \[
 Z(\PP^1, K, d, Y)  =
c_{\PP^1, M}Y^8+O_{d,\epsilon}\left(Y^7\frac{(h_MR_M)^{3/4} (1+\res_{s=1} \zeta_M(s))}{|\Delta_{K/\Q}|^{3/2}}\right),
    \]
where $h_M$ and $R_M$ are the class number and regulator of the compositum $M=KL$ respectively, and $c_{\PP^1, M}$ is Schanuel's constant for $\PP^1$ over $M$ \cite[Cor., p.\ 447]{MR0557080}. The implicit constant does not depend on $K$.
\end{theorem}

\begin{remark}
    This result should be compared to the works of Gao \cite[Thm.\ 3.1]{MR2693933} and Widmer \cite[Thm.\ 3.1]{MR2645051}. In our particular setting, the dependence on $\Delta_{K/\Q}$ in the error term of Theorem~\ref{thm:widmer+} improves on both results, neither of which would have been sufficient for our application.
\end{remark}

Theorem~\ref{thm:widmer+} is proven by adapting Schmidt's proof to a quadratic number field base, namely $L$, and applying the theory of $\calO_L$-lattice counting, rather than the traditional $\Z$-lattice counting. Much of Schmidt's argument was already generalised to a number field base by Widmer in the proof of \cite[Thm.\ 3.1]{MR2645051}, however one particular argument could not be since Widmer handles higher degree extensions than the quadratic extensions considered in this paper and by Schmidt. In particular, we make the error terms in the geometry of numbers explicit in a particular parameter (namely, the partition parameter $\mathbf{i}$, see \S\ref{ss:fundunits}) which Widmer treats uniformly. Using a trick of Schmidt,  we are able to achieve a saving in the sum over this parameter.
Moreover, further savings are accrued by using the $\calO_L$-lattice perspective, which provides sharper error terms as we have improved control of the successive minima. The union of these savings means that the error term in Theorem \ref{thm:widmer+} only just suffices to establish Theorem \ref{thm:mainthm1}, as we shall see later.

\subsection{Reinterpreting the counting problem}\label{s:widmerproof}
 We are interested in the counting function, using the notation from the beginning of \S\ref{s:nonsplit},
\[
Z(\PP^1, K,  d, Y) =
\#\{[x_0:x_1] \in \PP^1(M): \resH_{\calO(1)}([x_0:x_1]) \leq Y \text{ and } L(x_0 \colon x_1) = M\},
\]
where $L(x_0 \colon x_1)$ is the residue field of $[x_0 \colon x_1] \in \PP^1_L$.

We begin with basic manipulations to move the counting problem to $M^2$ instead of $\PP^1(M)$. Since rescaling $(x_0,x_1)$ corresponds to the same projective point, we first remove rescaling by non-units, and reduce the problem into a lattice point counting problem at the same time.

\subsubsection{Reduction to counting in ideals}\label{sss:reduction to ideals}

Fix $\frak a_1, \ldots, \frak a_{h_M}$ integral ideals representing the classes of $\Cl M$, the class group of $M$. Now represent a projective point by coordinates $(x_0,x_1)$ such that the fractional ideal $\langle x_0,x_1\rangle$ of $M$ generated by the coordinates equals one of the $\fraka_i$. This representation is unique up to rescaling by a unit in $\calO_M^\times$, which is a finitely generated abelian group of rank $r+s-1$, where $r = r_M$ denotes the number of real embeddings of $M$ and $s = s_M$ the number of its conjugate pairs of complex embeddings. Let us decompose this group into two parts.

\begin{defn}
    Consider the three quadratic subfields $K_j$ of $M$. For the real fields $K_j$ we define $\epsilon_j$ as the unique fundamental unit in $\calO_{K_j}$ with $\epsilon_j > 1$. We let $\mathbb U \subseteq \calO^{\times}_M$ be the subgroup generated by the $\epsilon_j$, and $\bar{\mathbb U}$ the saturation of $\mathbb U$ in the full unit group.
\end{defn}

To simplify the exposition we will occasionally use $\epsilon_j=1$ when $K_j$ is imaginary.

\begin{prop}
    \begin{itemize}
    \item[(a)] The abelian group $\mathbb U$, and hence $\bar{\mathbb U}$, is torsion-free and has rank $r+s-1$.
    \item[(b)] The index $q_M$ of $\mathbb U$ in $\bar{\mathbb U}$ satisfies $q_M \mid 8$.
    \item[(c)] The index of $\mathbb U$ in $\calO_M^\times$ equals $q_Mw_M$, for $w_M$ the number of torsion units in $\calO_M$.
    \end{itemize}
\end{prop}

\begin{proof}
    In this proof we restrict to the $j$ for which $K_j$ is real.
    \begin{itemize}
            \item[(a)] Suppose that $\eta := \prod_j \epsilon_j^{c_j}$ for $c_j \in \mathbb Z$ is a torsion unit. Let $\sigma_j \in \Aut_\Q M$ be the unique non-trivial element fixing $K_j$. Since $\sigma_j$ preserves $\epsilon_j$ and sends the other $\epsilon_{j'}$ to $\pm \epsilon_{j'}^{-1}$ we see that $\eta \sigma_j(\eta)=\pm \epsilon_j^{2c_j}$ is also a torsion unit. Hence $c_j=0$ for all $j$. This also proves that the $\epsilon_j$ for $K_j$ real form a basis for $\mathbb U$. Hence the rank equals either $1$ or $3$ depending on whether $(r,s)=(0,2)$ or $(4,0)$.
    \item[(b)] From the definition of the $\epsilon_j$ we have an isomorphism $\calO_{K_j}^\times/\{\pm 1\} \cong \langle \epsilon_j\}$. Consider the composition $\prod_j \langle \epsilon_j \rangle \to \bar{\mathbb U} \to \prod_j \calO_{K_j}^\times/\{\pm 1\} \cong \prod_j \langle \epsilon_j \rangle$  of the inclusion and relative norms, which equals squaring. The first homomorphism is injective by (a). We prove that the second homomorphism is injective too. Suppose that $x \in M$ has norm $1$ in all three quadratic subfields. This means that $\sigma(x)=x^{-1}$ for all non-trivial $\sigma \in \Aut_\Q M$, hence for $j \ne j'$ we have $x^{-1} = \sigma_j \circ \sigma_{j'} (x) = \sigma_j( \sigma_{j'} x) = \sigma_j(x^{-1}) = x$. Hence $x=\pm 1$ and it represents the unit of $\bar{\mathbb U}$.
    
    Hence the composite homomorphism is injective and the rank of the free group is at most 3, the result follows.
    \item[(c)] As $\bar{\mathbb U}$ is a maximal free subgroup of $\calO_M^\times$ we see that the index of $\bar{\mathbb U}$ in the unit group is $w_M$. The result follows. \qedhere
    \end{itemize}
\end{proof}

\begin{rmk}
    In \cite{MR1330740} and \cite{MR2645051} the unit action is split into a maximal free subgroup, $\bar{\mathbb U}$ in our case, and the finite quotient which is isomorphic to the torsion subgroup of $\calO_M^{\times}$. The approach works for any subgroup of finite index, that is of maximal rank $r+s-1$. For our application this particular choice $\mathbb U$ will be important.
\end{rmk}

Note that an element of $M^2 \setminus (0,0)$ up to the action by $\mathbb U$ has $q_M w_M$ representatives which produce the same projective point. We find that for all $Y \geq 1$ the function $Z(\PP^1, K, d, Y)$ equals
\[
\frac{1}{q_Mw_M}\sum_{i = 1}^{h_M} \#\{(x_0, x_1) \in M^2\setminus(0,0): \resH_{\calO(1)}([x_0:x_1]) \leq Y, \langle x_0, x_1 \rangle = \fraka_i, L(x_0\colon x_1) = M\}/\mathbb U.
\]
The definition of the absolute $\calO(1)$-height for the points in the inner cardinality above evaluates to
$$
\resH_{\calO(1)}([x_0:x_1])=N_{M/\Q}(\fraka_i)^{-1/4}\prod_{w \in \Omega_M^\infty} \max \{ \vert x_0 \vert_w, \vert x_1 \vert_w \}^{\frac{d_w}{4}}.
$$
Let $\mu_M$ denotes the generalised M\"obius function on the ideals of $\calO_M$. Using M\"obius inversion we arrive at the expression
\[
\frac{1}{w_Mq_M}\sum_{i=1}^{h_M}\sum_{0 \neq \frakb \subseteq \calO_M} \!\!\mu_M(\frakb)\, \#\{(x_0, x_1) \colon\!  \prod_{w \mid \infty}\resH_{\calO(1),w}(x_0,x_1)^{\frac{d_w}4} \leq T, x_0, x_1 \in \fraka_i\frakb,L(x_0\colon x_1) = M\}/\mathbb U
\]
with $\resH_{\calO(1),w}(x_0,x_1)=\max\{|x_0|_w,|x_1|_w\}$ and $T=N_{M/\Q}(\fraka_i)^{\frac14}Y$.

So, it suffices to count pairs $x_0,x_1 \in \fraka_i\frakb$ such that $\prod_{w \mid \infty} \calH_{\calO(1),w}(x_0,x_1)^{\frac{d_w}4}\leq T$. We now rewrite this condition while accounting for the action of $\mathbb U$.

\subsubsection{Accounting for the action of $\mathbb U$}\label{ss:fundunits}

For a number field $F$, the unit group modulo torsion units  embeds as a lattice in the hyperplane in $\mathbb R^{r_F+s_F}$ given by $\sum d_v x_v = 0$ under the map $\ell \colon F^\times \to \R^{r_F+s_F}, x \mapsto (\log |x|_v)_{v \mid \infty}$. By definition of the regulator $R_F$ of $F$ the determinant of this lattice equals $\sqrt{r_F+s_F} R_F$. For the biquadratic field $F=M$ we will suppress the subscript in the case of $r$ and $s$.

We diagonally embed $M$ in $M_\infty = \prod_{w \mid \infty} M_w$ via the natural diagonal embedding $\alpha\mapsto (\sigma_w(\alpha))_{w \mid \infty}$. We extend $M \hookrightarrow M_\infty$ componentwise to embed $M^2$ into $M_\infty^2$. Given an ideal $\fraka$ in $M$, we write $\Lambda(\fraka)$ for $\fraka \times \fraka \subseteq M^2$ which naturally lies in $M_\infty^2$. General pairs of coordinates in $M^2$ map to $\R^{r+s}_{\geq 0}$ by
\begin{equation}\label{eq:ideallattice}
M^2 \hookrightarrow M^2_\infty \xrightarrow{(\calH_{\calO(1),w})_w}{\R^{r+s}_{\geq 0}},
\end{equation}
where the local height function $\calH_{\calO(1),w} \colon M^2_w \to \mathbb R_{\geq 0}$ collects its arguments from the two factors of $M^2_\infty$ indexed by the infinite place $w$ of $M/\Q$.

To account for the action of $\mathbb U$ we will find a $\mathbb U$-fundamental region. For every $j$ with $K_j$ a real quadratic field let us write $v_j$ for the place of $K_j$ for which $|\epsilon_j|_{v_j} > 1$ and $v'_j$ for the conjugate place. We define for $x=(x_1,\dots,x_n) \in M^n_\infty$  with $x_i = (x_{i,w})_w \in \prod_{w\mid \infty} M_w$ the function
\[
\psi_j(x) = \frac{\prod_{w \mid v_j} \max_i \{|x_{i,w}|_w\}^{d_w}}{\prod_{w \mid v'_j} \max_i\{|x_{i,w}|_w\}^{d_w}},
\]
which we use for both $n=1$ and $n=2$. Note that for $n=1$ the function $\psi_j$ simplifies to
\[
\psi_j(x) = \frac{|N_{M/K_j}(x)|_{v_j}}{|N_{M/K_j}(x)|_{v'_j}} = \frac{|N_{M/K_j}(x)|^2_{v_j}}{|N_{M/\Q}(x)|}
\]
where $v'_j$ is the conjugate place of $v_j$. Here we have used that $d_{v_j}=d_{v'_j}=1$ as the two places of $K_j$ are real.

\begin{lemma}\label{lem:unit action on psis}
\begin{itemize}
    \item[(a)] Let $v$ be a place of one of the $K_j$. We have
    \[
    \prod_{w \mid v} |\epsilon_{j'}|_w^{d_w} = 
    \begin{cases}
        1 & \text{ if } j \ne j';\\
        |\epsilon_j|_v^2 & \text{ if } j=j'.
    \end{cases}
    \]
    \item[(b)] For $x \in M^n_\infty$ we have
    \[
    \psi_j(\epsilon_{j'} x) =
    \begin{cases}
        \psi_j(x) & \text{ if } j \ne j';\\
        \epsilon_j^4 \psi_j(x) & \text{ if } j=j'.
    \end{cases}
    \]
\end{itemize}
\end{lemma}

\begin{proof}
\begin{itemize}
    \item[(a)] We find
    \[
    \prod_{w \mid v} |\epsilon_{j'}|_w^{d_w} = |N_{M/K_j}(\epsilon_{j'})|^{d_v}_v.
    \]
    If $j=j'$ then $N_{M/K_j}(\epsilon_{j})=\epsilon^2_j$, and either $\epsilon_j=1$ in the complex case or $d_v=1$ in the real case, and if $j\neq j'$ then $N_{M/K_j}(\epsilon_{j'})=1$.
    \item[(b)] By definition and the discussion above we have 
\[
\psi_j(\epsilon_{j'}x) = \psi_j(\epsilon_{j'})\psi_j(x)
= \frac{\prod_{w \mid v} |\epsilon_{j'}|_w^{2d_w}}{|N_{M/\Q}(\epsilon_{j'})|} \psi_j(x). 
\]
The result now follows from part (a). \qedhere
\end{itemize}
\end{proof}

\begin{cor}\label{cor:quasi-fundamental domain}
        An element $(x_0,x_1)\in M^n \setminus (0,0)$ has a unique representative under the action of $\mathbb U$ for which
    \[
    \psi_j(x) \in [\epsilon_j^{-2}, \epsilon_j^2)
    \]
    for all $j$, for which $K_j$ is real.
\end{cor}

Let us write $S^{(n)}(T)$ for the union of $(0,0)$ and 
\[
\left\{(x_i)_i \in M_\infty^n \setminus (0,0)\colon \prod_{w \mid \infty}\max_i\{|x_{i,w}|_w\}^{\frac{d_w}4} \leq T, \,\,\psi_j(x)\in [\epsilon_j^{-2}, \epsilon_j^2)\, \text{ for all } j \text{ with } K_j \text{ real} \right\},
\]
so that $S^{(n)}(T)$ is cut out by the conditions in Corollary~\ref{cor:quasi-fundamental domain} and contains all points in $M^n\setminus (0,0)$ of height at most $T$.

\begin{prop}\label{prop:volume of quasifundamental domain}
    We have $S^{(n)}(T) = T S^{(n)}(1)$ for any $T>0$, and 
    \[
    \vol S^{(n)}(T) = q_M \left(2^r \pi^s\right)^{n+1}(n+1)^{r+s-1}R_M T^n.
    \]
\end{prop}

\begin{proof}
    It is clear that the condition on $S^{(n)}(T)$ imposed by the $\psi_j$ is invariant under multiplying an element of $M_\infty^n$ by a positive real number. On the other hand, from $\prod_{w\mid \infty} T^{d_w/4} = T$ we see that the height condition scales linearly with $T$.

    For the second part, the volume of a $\mathbb U$-fundamental domain is $q_M$ times the volume of a $\bar{\mathbb U}$-fundamental region. Such a volume was computed by Widmer \cite[p.\ 33]{MR2645051} to be
    \[
    \left(2^r \pi^s\right)^{n+1}(n+1)^{r+s-1}R_M.
    \qedhere
    \]
\end{proof}

Now we can write the counting function at the end of \textsection \ref{sss:reduction to ideals} as 
\[
\frac{1}{q_M w_M}\sum_{i=1}^{h_M}\sum_{0 \neq \frakb \subseteq \calO_M} \!\!\mu_M(\frakb)\, \#\{x_0,x_1 \in \fraka_i\frakb \colon\!  (x_0,x_1) \in S^{(2)}(T),L(x_0\colon x_1) = M\}
\]
with $T=N_{M/\Q}(\fraka_i)^{\frac14}Y$, because of the local height function $\resH_{\calO(1),w}(x_0,x_1)=\max\{|x_0|_w,|x_1|_w\}$.

\subsubsection{The partition of the fundamental region}\label{sss:partition}
We would like to count the appropriate lattice points inside $S(T) := S^{(2)}(T)$. However, a priori, this region might be contained within a ball with radius of order $e^{R_M}T$ which would give catastrophically large dependence of the error term on the regulator in an application of Proposition~\ref{prop:generalised Davenport}. To handle this, we follow Schmidt~\cite[\textsection 4]{MR1330740} to partition the counting region $S(T)$ into approximately $R_M$ subsets of radius approximately $T$ so that this dependence is reduced significantly.
When $K_j$ is real, we define $t_j = \lfloor R_{K_j} \rfloor + 1$ and $u_j = e^{R_{K_j}/t_j}$. Hence $u_j^{t_j} = \epsilon_j$, and Schmidt shows
\begin{equation}\label{eq:bounding uj}
1< u_j \ll 1.
\end{equation}
For $K_j/\Q$ an imaginary quadratic extension, we have $R_{K_j} = 1$ and define $t_j = 1$ and $u_j = 1$.

We take $\mathbf i = (i_j)_j$ with $i_j$ an integer in $[-t_j,t_j)$. We then define
$
S_{\mathbf i}(T)
$
as the region given by the height condition and $\psi_j(x) \in [u_j^{2i_j},u_j^{2i_j+2})$ for all $j$. We now have the partition $S(T) = \bigcup_{\mathbf{i} \in \mathcal I} S_{\mathbf{i}}(T)$. We will write $S_0(T)$ for the part corresponding to $\mathbf i=(0,0,0)$. Note that if $K_j$ is imaginary, we have $i_j \in \{-1,0\}$, but $S_{\mathbf i}(T)$ for any $\mathbf i$ with $i_j = -1$ will be empty. In what follows we ignore those parts, and we define $\calI$ as the collection of such triples $\mathbf i$ as above with $i_j=0$ whenever $K_j$ is imaginary. Let us write $t := \# \calI$.

\begin{prop}\label{prop:t asymp qMhM}
    We have $t \asymp q_M R_M$ with an absolute implicit constant.
\end{prop}

\begin{proof}
    By definition $t_j \asymp R_{K_j}$ with absolute implicit constant. The equality $q_M R_M \asymp \prod_j R_{K_j}$ follows from (7.26) and (7.27) in \cite[Ch.\ VIII]{FT}.
\end{proof}

For every $j$ we have an $\R$-linear morphism $\tau_j \colon M_w \to M_w$ given by $x \mapsto |\epsilon_j|_w^{-1/2t_j} x$, which gives a linear morphism $\tau_j \colon M_\infty \to M_\infty$ of determinant $1$. We now define $\tau_{\mathbf i} = \prod_j \tau_j^{i_j}$. We often abuse notation and write $\tau_{\mathbf i}$ also for the induced linear automorphism of $M_\infty^2$.

\begin{lemma}\label{lem:tau maps Si to S0}
    The linear automorphism $\tau_{\mathbf i}$ of $M^2_\infty$ maps $S_{\mathbf i}(T)$ to $S_0(T)$.    
\end{lemma}

\begin{proof}
    It is immediate by the product formula that the height condition
    \[
    \prod_{w \mid \infty}\max_i\{|x_{i,w}|_w\}^{\frac{d_w}4} \leq T
    \]
    is satisfied for $x=(x_0,x_1) \in M^2_\infty \setminus (0,0)$ if and only if it satisfied for $\tau_{\mathbf i} x$.

    We consider the point $\tau_{\mathbf i}x \in M^2_\infty$ with coordinates $x_i  \prod_{j'} |\epsilon_{j'}|_w^{-i_{j'}/2t_{j'}} \in M_w$ indexed by $i \in \{0,1\}$ and $w \in \Omega^\infty_M$. The function $\psi_j$ only detects the factor with $j'=j$, and we have
    \[
    \psi_j(\tau_{\mathbf i} x) = u_j^{-2 i_j} \psi_j(x).
    \]
    Hence $\psi_j(x) \in [u_j^{2i_j}, u_j^{2i_j+2})$ for all $j$ if and only if $\psi_j(\tau_{\mathbf i} x) \in [1, u_j^2)$ for all $j$.
\end{proof}

\begin{remark}
     Our description of the action of the unit group and the partitioning of the fundamental domain is rather different from that of Gao and Widmer and has more in common with Schmidt's approach. This is so that we can use explicitly the way in which the unit structure in $M$ is more or less inherited from the units in its quadratic subfields. In more general extensions, there is not such a clear relationship and this perspective is no longer practical or useful.
\end{remark}

\begin{prop}\label{prop:Lipschitzity of counting region}
    The region $S_{\mathbf i}(T)$ is $\text{Lip}_{M^2}(R,N)$ with respect to $B'_{\infty, \mathbf i} := \tau_{\mathbf i} B_\infty$, where $R \leq \kappa T$ for a constant $\kappa$ depending at most on $d$ and $N \ll 1$.
\end{prop}

\begin{proof}
    We only need to prove $S_0(T)$ is $\text{Lip}_{M^2}$ with respect to $B_{\infty}$. This follows from \cite[Lem.\ A.1]{MR2645051}.
\end{proof}

\subsubsection{Ideals under different lattice structures}

We now consider the ideal $\frakd \subseteq \calO_M$ with lattice structure inherited from the partition function. To that end, define $B_v\subseteq M_v$ to be the unit region under the maximum norm $\{(x_w)_{w \mid v} \colon \vert x_w\vert_w\leq 1\}$. We consider $\frakd$ as lattice with respect to the local regions $B'_{\mathbf i, v} := \tau_{\mathbf i,v}^{-1} B_v = \{(x_w)_{w \mid v} \colon \vert x_w\vert_w\prod_j|\epsilon_j|_w^{-i_j/2t_j}\leq 1\}$.

\begin{defn}\label{defn:succ min for different regions}
    Let $\lambda_i(\frakd,\mathbf i)$ denote the successive minimum of the $\mathcal{O}_L$-lattice $\frakd \subseteq M$ with respect to $B'_{v,\mathbf i}$.
\end{defn}

To study how the successive minima of $\frakd$, especially the first, vary with $\mathbf i$ we introduce the following functions.

\begin{defn}
    We define for $K_j$ a real quadratic field
    \[
    m^{(j)} \colon M_\infty \to \R_{\geq 0}, (x_w)_w \mapsto \max_{v \in \Omega_{K_j}^\infty} \left\{ \prod_{w \mid v} |x_w|^{d_w/2}_w\right\},
    \quad \text{ and } \quad m := \left[ \prod_{j: K_j \text{ real}} m^{(j)}\right]^{1/\#\{j: K_j \text{ real}\}}.
    \]
    We also define
     $m^{(j)}_{\mathbf i}(x) := m^{(j)}(\tau_{\mathbf i} x)$ and
     $m_{\mathbf i}(x) := m(\tau_{\mathbf i} x)$.
\end{defn}

Note that $m^{(j)}(x)$ is the maximum among the numerator and denominator of $\psi_j(x)$ for $x \in M_\infty$.

Using Proposition~\ref{lem:unit action on psis}(a) we see that
    \begin{align*}
    m_{\textbf i}^{(j)}(x) & = \max_{v \in \Omega_{K_j}^\infty} \left\{ \prod_{w \mid v} \left[ (\prod_{j'} \vert\epsilon_{j'}\vert_w^{-d_wi_{j'}/4t_{j'}}) |x_w|^{d_w/2}_w \right] \right\}\\
    & = \max_{v \in \Omega_{K_j}^\infty} \left\{\left[ \prod_{j'} \left( \prod_{w \mid v}\vert\epsilon_{j'}\vert_w^{d_w/2} \right)^{-i_{j'}/2t_{j'}} \right] \prod_{w \mid v} |x|_w^{d_w/2}\right\}\\
    & = \max_{v \in \Omega_{K_j}^\infty} \left(\vert\epsilon_{j}\vert_v^{-i_{j}/2t_{j}} \prod_{w \mid v} |x_w|^{d_w/2}_w \right)
    = \max_{v \in \Omega_{K_j}^\infty} \left(\vert\epsilon_{j}\vert_v^{-i_{j}/2t_{j}} |N_{M/K_j}(x)|^{d_v/2}_v\right)
    \end{align*}
which simplifies further as $d_v=1$ for $v$ a place of a real quadratic field $K_j$.

\begin{lemma}
    The function $m_{\mathbf i}$ assumes a positive minimum on the non-zero elements of an ideal $\frakd \subseteq M$.
\end{lemma}

We will write $\mu_1(\frakd, \mathbf i)>0$ for this minimum.

\begin{proof}
    It is clear that $m_{\textbf i}(x)=0$ if and only if $x=0$, hence the minimum is positive if it exists.
    
    Now note that for a fractional ideal $\frakd$ of a number field $F$, and a sequence of numbers $\lambda_v \in \mathbb R_{>0}$ indexed by the archimedean places of $F$ the function
    \[
    x \mapsto \max_v (\lambda_v |x|_v)
    \]
    on $x \in \frakd$ takes its values in a discrete subset of $\mathbb R_{\geq 0}$ as there are only finitely many elements of $x \in \frakd$ with $\max_v (\lambda_v |x|_v)\leq X$ for every $X > 0$ by Proposition~\ref{prop:generalised Davenport}. 
    
    For each $j$ with $K_j$ a real quadratic field, we have
    \[
    m_{\textbf i}^{(j)}(x)^2 = \max_{v \in \Omega_{K_j}^\infty} \left(\vert\epsilon_{j}\vert_v^{-i_{j}/t_{j}} |N_{M/K_j}(x)|_v\right).
    \]
    Hence by the result applied to a real $K_j$ with the ideal $N_{M/K_j}(\frakd)$ we see that $m_{\textbf i}^{(j)}$ takes its value in a discrete set.

    We conclude that the image of $m_{\textbf i}(\frakd)$ is a discrete subset of $\R_{\geq 0}$, which proves the claim.
\end{proof}

The definition of $\mu_1(\frakd, \mathbf i)$ allows us to prove the bound in Proposition~\ref{prop:secondschmidttrick} which will be crucial to handle the error terms in our count after the application of Proposition~\ref{prop:generalised Davenport}.

Let $\mathcal{C}$ be an ideal class in $\mathcal{O}_M$. Recalling that $t$ denotes the number of partitions of the $\mathbb{U}$-fundamental domain of the unit lattice, we define
    $$
    N(\mathcal{C}):=\sum_{n=1}^t\frac{1}{N_{M/\Q}(\frakc_n)^{1/4}},
    $$

    where the ideals $\frakc_n$ belong to the ideal class $\mathcal{C}$, and they are ordered by increasing norms $N_{M/\Q}(\frakc_1)\leqslant N_{M/\Q}(\frakc_2)\leqslant \cdots$. We will relate the sum of the reciprocals of the first successive minima to this norm, thus allowing for some non-trivial savings in the sum of the error term over the different parts of the partition.

\begin{prop}\label{prop:secondschmidttrick}
    We have
     \[
    \sum_{\mathbf i \in \calI} \frac 1{\lambda_1(\frakd, \mathbf i)} \ll N_{M/\Q}(\frakd)^{-1/4} N(\calD^{-1}),
    \]
     where $\calD$ is the ideal class of $\frakd$.
\end{prop}

\begin{remark}
    The above result is more or less the only facet of Schmidt's paper which was not generalised previously to  base fields other than $\Q$, although a version was established for ideals viewed as $\Z$-lattices by Debaene~\cite[Thm 11]{debaene}. In Widmer's approach this sum is estimated by $N_{M/\Q}(\frakd)^{-1/4} R_M$ which would not suffice for our purposes.
\end{remark}

\begin{proof}[Proof of Proposition \ref{prop:secondschmidttrick}]
    Let $\beta \in \frakd$ achieve $\lambda_1(\frakd, \mathbf i)$, that is $\max_{w\mid \infty} \left( |\beta|_w \prod_j |\epsilon_j|_w^{-i_j/2t_j}\right)$ is minimal. Firstly, observe that for each $j$ we have
    \begin{align*}
    m_{\textbf i}^{(j)}(\beta) = \max_{v \in \Omega_{K_j}^\infty} \left\{\prod_{w \mid v}(\prod_j \vert\epsilon_j\vert_w^{-d_wi_j/4t_j}) \vert \beta \vert_w^{d_w/2}\right\}
   & \leq
     \max_{v \in \Omega_{K_j}^\infty} \max_{w \mid v}(\prod_j \vert\epsilon_j\vert_w^{-i_j/2t_j}) \vert \beta \vert_w\\
     &=
     \max_{w\mid \infty} (\prod_j \vert\epsilon_j\vert_w^{-i_j/2t_j})\vert \beta\vert_w = \lambda_1(\frakd, \mathbf i).
    \end{align*}
    To deduce the inequality above we use $\prod_w x_w^{d_w/2} \leq \max_w x_w$ since $\sum_{w\mid v} d_w = 2d_v=2$. Hence, we have $m_{\mathbf i}(\beta) \leq \lambda_1(\frakd,\mathbf i)$. Thus it suffices to bound
    \[
    \sum_{\mathbf i \in \calI} \frac{1}{\mu_1(\frakd, \mathbf i)}.
    \]

    Let $\alpha := \alpha(\frakd, \mathbf i)$ be an element of $\frakd$ achieving $\mu_1(\frakd, \mathbf i)$. Then we claim that $\psi_{j}(\alpha) \in[\epsilon_j^{-4},\epsilon_j^4)$ for each $j$ such that $K_j$ is real.
    
    Indeed suppose that $\psi_1(\alpha)>\epsilon_1^4$, noting that this is only possible if $K_j$ is real. We will prove that $m_{\mathbf i}(\alpha) > m_{\mathbf i}(\epsilon_1^{-1}\alpha)$, contradicting the minimality of $m_{\mathbf i}$ at $\alpha$.

    First note that the terms in $m^{(2)}_j$ and $m^{(3)}_j$ have been set up so they are unaffected by multiplying by $\epsilon_1$ as per Proposition \ref{lem:unit action on psis}(a). To prove $m^{(1)}_{\mathbf i}(\alpha) > m^{(1)}_{\mathbf i}(\epsilon_1^{-1}\alpha)$ we recall $m^{(1)}_{\mathbf i}(x) = \max_{v \in \Omega_{K_1}^\infty} \left(\vert\epsilon_1\vert_v^{-i_1/2t_1} |N_{M/K_1}(x)|^{1/2}_v\right)$ and first prove that the maximum in $m^{(1)}_{\mathbf i}(\epsilon_1^{-1} \alpha)$ is attained for $v=v_1$ the place $K_1$ for which $v_1(\epsilon_1) > 1$ as in \S\ref{ss:fundunits}. Indeed,
    \[
\frac{\vert\epsilon_1\vert_{v_1}^{-i_1/t_1} |N_{M/K_1}(\epsilon_1^{-1} \alpha)|_{v_1}}{\vert\epsilon_1\vert_{v'_1}^{-i_1/t_1} |N_{M/K_1}( \epsilon_1^{-1}\alpha)|_{v'_1}} = \epsilon_1^{-2i_1/t_1} \psi_1(\epsilon_1^{-1} \alpha) \geq \epsilon_1^{-2} \epsilon_1^{-2} \epsilon_1^4 = 1
    \]
    since $i_1 \in [-t_1,t_1)$ and $\psi_1(\epsilon_1^{-1} \alpha) = \epsilon_1^{-2} \psi_1(\alpha)$ by Proposition \ref{lem:unit action on psis}(b).
    We now get
    \begin{align*}
    m^{(1)}_{\mathbf i}(\alpha)^2 & \geq \vert\epsilon_1\vert_{v_1}^{-i_1/t_1} |N_{M/K_1}( \alpha)|_{v_1} > \epsilon_1^{-2} \left( \vert\epsilon_1\vert_{v_1}^{-i_1/t_1} |N_{M/K_1}( \alpha)|_{v_1} \right)\\
    & = \vert\epsilon_1\vert_{v_1}^{-i_1/t_1} |N_{M/K_1}( \epsilon_1^{-1}\alpha)|_{v_1} = m^{(1)}_{\mathbf i}(\epsilon_1^{-1} \alpha)^2.
    \end{align*}
    Thus we conclude that
   \[
    m_{\mathbf i}(\alpha) = \prod_j m_{\mathbf i}^{(j)}(\alpha) ^{1/\#\{j:K_j \text{ real}\}} > \prod_j m^{(j)}_{\mathbf i}(\epsilon_1^{-1}\alpha)^{1/\#\{j:K_j \text{ real}\}} = m_{\mathbf i}(\epsilon_1^{-1}\alpha),
    \] which contradicts the minimality of $\alpha$. The proof in the case $\psi_1(\alpha) < \epsilon_1^{-4}$ is  exactly the same just with all the exponents reversed.

For any $\alpha \in \frakd$ with $\psi_j(\alpha) \in [\epsilon_j^{-4},\epsilon_j^4)$, we have
    \[
    \sum_{\substack{\mathbf i \in \calI \\ \alpha(\frakd, \mathbf i) = \alpha}} \frac{1}{\mu_1(\frakd, \mathbf i)}
    = \sum_{\substack{\mathbf i \in \calI\\ \alpha(\frakd,\mathbf i)=\alpha}} \left(\prod_{j: K_j \text{ real}} \min_{ v \in \Omega^\infty_{K_j}} \vert\epsilon_j\vert_v^{i_j/2t_j} |N_{M/K_j}(\alpha)|_v^{-1/2} \right)^{1/\#\{j : K_j \text{ real}\}}
    \]
Since $\prod_w |\alpha|^{d_w}_w=|N_{M/\Q}(\alpha)|$ and the $\epsilon_j$ are linearly independent in the unit lattice of $M$, we can solve for $\mathbf c_j \in \mathbb R$ such that
    \[
    |\alpha|_w= |N_{M/\Q}(\alpha)|^{1/4} \prod_{j'} |\epsilon_{j'}|_w^{c_{j'}},
    \]
    and hence
    \begin{align*}
    |N_{M/K_j}(\alpha)|_v &= |N_{M/\Q}(\alpha)|^{1/2} \prod_{w \mid v} \prod_{j'} |\epsilon_{j'}|_w^{d_wc_{j'}}\\
    &= |N_{M/\Q}(\alpha)|^{1/2} \prod_{j'} |N_{M/K_j}(\epsilon_{j'})|_v^{c_{j'}} = |N_{M/\Q}(\alpha)|^{1/2} |\epsilon_j|_v^{2c_j}.
    \end{align*}
Therefore, we conclude that 
\begin{align*}
\sum_{\substack{\mathbf i \in \calI \\ \alpha(\frakd, \mathbf i) = \alpha}} \frac{1}{\mu_1(\frakd, \mathbf i)}
     &=
      |N_{M/\Q}(\alpha)|^{-1/4} \sum_{\substack{\mathbf i \in \calI \\ \alpha(\frakd, \mathbf i) = \alpha}} \left( \prod_{j : K_j \text{ real}}\min_{v \in \Omega_{K_j}^\infty   } \vert \epsilon_j \vert_v^{\frac{\mathbf i_j}{2t_j}-c_j} \right)^{1/\#\{j: K_j \text{ real}\}}.
\end{align*}
The parameters $i_j$ go up to $t_j$ but for the purposes of an upper bound we may extend the sum to infinity for those $j$ such that $K_j$ is real. This shows that 
\[
\sum_{\substack{\mathbf i \in \calI \\ \alpha(\frakd, \mathbf i) = \alpha}} \frac{1}{\mu_1(\frakd, \mathbf i)}
     \leq
      |N_{M/\Q}(\alpha)|^{-1/4}
      \sum_{\substack{j: K_j \text{ real}\\ k_j \in \Z}} \min_{v \in \Omega_{K_j}^\infty}
      \vert \epsilon_j \vert_v^{\frac{k_j/2t_j-c_j}{\#\{j:K_j \text{ real}\}}}.
\]
When $\frac{k_j}{2t_j} - c_j \leq 0$ the minimum is attained at $v=v_j$, conversely when $\frac{k_j}{2t_j} - c_j \geq 0$ the minimum is attained at $v=v_j'$. Since $\vert \epsilon_j \vert_{v_j} \vert \epsilon_j \vert_{v_j'} = 1$, we have that 
\[
 \vert \epsilon_j \vert_{v_j'}^{\frac{k_j/2t_j-c_j}{\#\{j:K_j \text{ real}\}}}= \vert \epsilon_j \vert_{v_j}^{\frac{c_j-k_j/2t_j}{\#\{j:K_j \text{ real}\}}}.
\]
Therefore we see that each range corresponds to summing the term $ \vert \epsilon_j \vert_{v_j}^{\frac{\ell_j}{2t_j\#\{j:K_j \text{ real}\}}}=u_j ^{\frac{\ell_j}{2\#\{j:K_j \text{ real}\}}}$ as $\ell_j$ runs over non-positive integers. In summary, we have
\[
\sum_{\substack{\mathbf i \in \calI \\ \alpha(\frakd, \mathbf i) = \alpha}} \frac{1}{\mu_1(\frakd, \mathbf i)}
     \leq
      |N_{M/\Q}(\alpha)|^{-1/4}
      \prod_{j:K_j \text{ real}} \frac{2}{u_j^{1/2 \#\{j: K_j \text{ real}\}}-1}
      \ll |N_{M/\Q}(\alpha)|^{-1/4},
\]
where we have used $u_j > 1$ from \eqref{eq:bounding uj}.

    For each $\mathbf i \in \calI$, there exists an integral ideal $\frakd_{\mathbf i}'$ in the ideal class $\calD^{-1}$ such that $\alpha(\frakd, \mathbf i) \calO_M = \frakd \frakd_{\mathbf i}'$. Conversely, suppose we are given an ideal $\frak c$ in $\calD^{-1}$ and we want to find an element $\alpha \in \calO_M$ such that $\alpha \calO_M = \frakd \frakc$ and such that $\psi_j(\alpha) \in [\epsilon_j^{-4}, \epsilon_j^4)$. Since $\psi_j(\pm\epsilon_1^{k_1}\epsilon_2^{k_2}\epsilon_3^{k_3}\alpha) = \epsilon_j^{4k_j}\psi_j(\alpha)$ by Lemma \ref{lem:unit action on psis}, it follows that there are at most $16q_M\ll 1$ ways to choose $\alpha$. 
    
    Now supposing that $s\leq t$ distinct numbers $\alpha_1, \ldots, \alpha_s$ arise among the $\alpha(\frakd, \mathbf i)$ as $\mathbf i$ varies, we deduce that 
    \[
    \sum_{\mathbf i \in \calI} \frac{1}{\mu_1(\frakd, \mathbf i)} \ll \sum_{n=1}^s |N_{M/\Q}(\alpha_n)|^{-1/4}.
    \]
    Hence for $t$ certain distinct integral ideals $\frakc_1, \ldots, \frakc_t$, we conclude
    \[
    \sum_{\mathbf i \in \calI} \frac{1}{\mu_1(\frakd, \mathbf i)} \ll |N_{M/\Q}(\frakd)|^{-1/4}\sum_{n=1}^t |N_{M/\Q}(\frakc_n)|^{-1/4} \leq |N_{M/\Q}(\frakd)|^{-1/4}N(\calD^{-1}),
    \] as desired.
\end{proof}

\subsubsection{Counting pure points in ideals}
For a finitely generated $\calO_L$-submodule $\Lambda \subseteq M^2$ and an open bounded $\calS \subseteq M^2_\infty$ (or the closure thereof), let
\[
Z^*(\Lambda, \calS) := 
\#\{ \omega \in \Lambda \cap \calS \colon \omega \neq \boldsymbol{0}, L(\omega_0\colon \omega_1) = M\}
\]
Recall that any fractional ideal $\frakd$ of $M$ defines a lattice $\Lambda(\frakd) \subseteq M^2$. In general, we will be using $\Lambda(\frakd)$ for $\Lambda$ and the sets $S(T)$ or $S_{\mathbf{i}}(T)$ for $\calS$.
 
Over the last sections we have recovered the following decomposition of the counting function, analogous to that from \cite{MR2645051}.

\begin{lemma}\label{lem:widmerdecomp}
Let $w_M$ denote the number of roots of unity in $M$. Then for $Y>0$, we have
\[
Z(\PP^1,K,d,Y) = \frac1{q_Mw_M}
\sum_{\fraka} \sum_{0 \neq \frakb \subseteq \mathcal{O}_M} \mu_M(\frakb) \sum_{\mathbf i \in \mathcal I} Z^*(\Lambda(\fraka\frakb), S_{\mathbf i}(N(\fraka)^{1/4}Y)),
\]
where the $\fraka$ sum runs over any set of ideal class representatives of $\Cl(M)$ and $\mu_M$ is the number field M{\"o}bius function. 
\end{lemma}

Our aim will be to apply the counting result in Proposition~\ref{prop:generalised Davenport} to estimate the various terms of the form $Z^*(\Lambda(\frakd), S_{\mathbf i}(T))$. To that end let us remark that $B_v \times B_v \subseteq M_v^2$ is the unit region under the maximum of the $w$-adic valuations for all $w \mid v$ of both coordinates, similarly $B_\infty \times B_\infty \subseteq M_\infty^2$ is the maximum among all $w$-adic valuations. Recall the linear automorphism $\tau_{\mathbf i}$ of $M_\infty$ and the induced $\tau_{\mathbf i}$ on $M^2_{\infty}$ from \textsection\ref{sss:partition}, and the notation for successive minima $\lambda_i(\frakd,\mathbf i)$ with respect to the transformed unit balls $B'_{v,\mathbf i} := \tau_{\mathbf i}^{-1} B_v$ from Definition~\ref{defn:succ min for different regions}.

\begin{lemma}
    The successive minima of $\Lambda(\frakd):=\frakd \times \frakd \subseteq M^2$ with respect to the balls $B'_{v,\mathbf i} \times B'_{v,\mathbf i}$ are
    \[
    \lambda_1(\frakd,\mathbf i),\; \lambda_1(\frakd,\mathbf i), \;\lambda_2(\frakd,\mathbf i) \quad \text{ and } \quad \lambda_2(\frakd,\mathbf i).
    \]
\end{lemma}

\begin{proof}
    The proof is identical to the proof of \cite[Lem.\ 4.2]{MR2645051}.
\end{proof}
By Lemma~\ref{lem:tau maps Si to S0} and Proposition~\ref{prop:Lipschitzity of counting region} we have
\[
S_{\mathbf i}(T) = \tau^{-1}_{\mathbf i} S_0(T) \subseteq \tau^{-1}_{\mathbf i}\kappa T (B_\infty \times B_\infty) = \kappa T (B'_{\infty,\mathbf i} \times B'_{\infty,\mathbf i}),
\]
which shows that in the application of Proposition~\ref{prop:generalised Davenport} to $Z^*(\Lambda(\frakd), S_{\mathbf i}(T))$ it makes sense to compute the successive minima with respect to these regions.

\begin{lemma}
\label{lem:lambda2ub}
Let $\frakd$ be a fractional ideal of $M$ and $T\geqslant 1$. Consider the absolute constant~$\kappa$ from Proposition~\ref{prop:Lipschitzity of counting region}. If $\lambda_2(\frakd,\mathbf i)>\kappa T$, then we have
\[
Z^{\ast}\big(\Lambda(\frakd),S_{\mathbf i}(T)\big)=0.
\]
\end{lemma}

\begin{proof}
    Suppose $(x_0,x_1) \in \Lambda(\frakd) \setminus (0,0)$ is counted in $Z^{\ast}\big(\Lambda(\frakd),S_{\mathbf i}(T)\big)$. The condition $L(x_0 \colon x_1) = M$ implies that $x_0,x_1 \in \frakd$ are $L$-linearly independent, hence by definition of $\lambda_2(\frakd, \mathfrak i)$ for every $\lambda < \lambda_2(\frakd, \mathfrak i)$ there is an $i$ with $x_i \not\in \lambda B'_{\infty, \mathbf i}$. However, from $(x_0,x_1) \in S_{\mathbf i}(T)$, we get $x_i \in \kappa T B'_{v,\mathbf i} \subseteq (\lambda_2(\frakd, \mathfrak i)-\epsilon) B'_{v,\mathbf i}$ for some $\epsilon > 0$. This is a contradiction.
\end{proof}

With this result in hand, we have the following point counting estimate.

\begin{prop}\label{prop:counting}
    \[
    Z^*(\Lambda(\frakd), S_{\mathbf i}(T))
    =
    \frac{\vol S_{\mathbf i}(T)}{\det \Lambda(\frakd)}
+ O_L\left( \frac{T^7}{\lambda_1(\frakd,\mathbf i) \det(\frakd)^{3/2} } \right)
    \]
\end{prop}

\begin{proof}
    For simplicity let us write $\lambda_i$ for $\lambda_i(\frakd,\mathbf i)$.
    
    The proof is split into two cases. Firstly, suppose that $\lambda_2 > \kappa T$. In this case, we know from Lemma~\ref{lem:lambda2ub} that $Z^*(\Lambda(\frakd), S_{\mathbf i}(T)) = 0$. Moreover, we have
    \[
    \frac{\vol S_{\mathbf i}(T)}{\det \Lambda(\frakd)}
    \ll \frac{T^8}{\lambda_1 \lambda_2 \det(\frakd)^{3/2}} \ll \frac{T^7}{\lambda_1 \det(\frak d)^{3/2}},
    \] which verifies that the claimed bound holds.

    Otherwise, suppose that $\lambda_2 \ll T$. By Lemma \ref{prop:Lipschitzity of counting region}, the region $\tau_{\mathbf i}S_{\mathbf i}(T)$ is in $\text{Lip}(\kappa T,N)$, for some constant $\kappa$ depending at most on $d$ and some absolute constant  $N$. Hence we may apply Proposition~\ref{prop:generalised Davenport}, which gives
    \[
    \left \vert \#\big(\Lambda(\frakd) \cap S_{\mathbf i}(T)\big)
    -
    \frac{\vol S_{\mathbf i}(T)}{\det \Lambda(\frakd)} \right \vert\ll_L 1+ \frac{T^2}{\lambda^2_1} + \frac{T^4}{\lambda^4_1} + \frac{T^6}{\lambda^4_1\lambda^2_2} + \frac{T^7}{\lambda^4_1\lambda^3_2}.
    \]
    But since $T \gg \lambda_2 \geq \lambda_1$, we have 
    \[
    \frac{T}{\lambda_1} \geq \frac{T}{\lambda_2} \gg 1,
    \]
    and hence the $T^7$ term in the sum dominates the others.  To turn this into an estimate for $Z^*(\Lambda(\frakd), S_{\mathbf i}(T))$, we need to remove the set of lattice points in $\Lambda(\frakd)$ which do not generate~$M$ over $L$. Such points are given by pairs $(\alpha_0,\alpha_1)$ with $\alpha_i \in \frakd$ which are linearly dependent over $L$. Without loss of generality, assume that $\alpha_0 \neq 0$. Then $\alpha_0$ lies in the $\mathbb U$-fundamental region $S^{(1)}_{\mathbf i}(T) \subseteq M_\infty$ and hence we may apply Proposition~\ref{prop:generalised Davenport} to the $\calO_L$-lattice $\frakd$ of rank~2 to deduce that the number of choices for $\alpha_0$ is
    \[
    \frac{\vol S^{(1)}_{\mathbf i}(T)}{\det\frakd}
    + O_L\left(1 + \frac{T^2}{\lambda_1^2} + \frac{T^3}{\lambda_1^2 \lambda_2} \right)
    \ll \frac{T^4}{\det \frakd}.
    \]
    Moreover, given a choice of $\alpha_0$, the $L$-linearly independent choices for $\alpha_1$ lie in $S^{(1)}_{\mathbf i}(T) \subseteq M_\infty$. We apply Proposition~\ref{prop:generalised Davenport} once more to the lattice $\frakd$ and the counting region $L_\infty \alpha_0 \cap S^{(1)}_{\mathbf i}(T) \subseteq M_\infty$. As the volume of this region is $0$, we can bound the number of choices of $\alpha_1$ by $1 + \frac{T^2}{\lambda_1^2} +  \frac{T^3}{\lambda_1^2 \lambda_2} \ll_L \frac{T^3}{\lambda_1^2 \lambda_2} \ll \frac{T^3}{\lambda_1 (\det \frak d)^{1/2}}$. Combining these two  bounds we see that the contribution from these linearly dependent pairs is contained within the claimed error term.
\end{proof}

\subsection{Proof of Theorem~\ref{thm:widmer+}}
We are now almost ready to conclude our main counting result. Before doing so we require one final ingredient, an upper bound in the vein of Landau's theorem on ideals of bounded norm. This will allow us to estimate the norm $N(\calD^{-1})$ which arises in Proposition \ref{prop:secondschmidttrick}.
\begin{lemma}\label{lem:landau}
With $K$ and $L$ as before and $M=KL$, we have
\[
\#\{ \fraka \subseteq \calO_M : N_{M/\Q}(\fraka) \leq X\}
\ll_{L, \epsilon}
X \res_{s=1} \zeta_M(s) + X^{1/2}\vert \Delta_{K/\Q} \vert^{1/3 + \epsilon}.
\]
\end{lemma}

\begin{remark}
    A field uniform version of Landau's theorem was established by Lowry-Duda--Taniguchi--Thorne~\cite[Thm.\ 3]{MR4381213}. Unfortunately their result contains a $\log^{d-1} X$ term in general which is too large for our purposes. Instead we are able to use the fact that we are working with a biquadratic field, that we are only concerned with the dependence on the discriminant of two out of the three quadratic subfields,
    and that we only require an upper bound to remove this term. Pointwise the residue appearing in the main term can only be bounded by something like $\log^3 \vert \Delta_{K/\Q}\vert$ which can be as large as $\log^{3} X$ again in our application but the advantage that we gain is that on average this residue is constant.
\end{remark}

\begin{proof}[Proof of Lemma \ref{lem:landau}]
Let $a_n$ denote the number of ideals of norm $n$ and $X\geqslant 1$ be fixed. We start by smoothing the sum
    \[
    \#\{ \fraka: N_{M/\Q}(\fraka) \leq X\}
    =
    \sum_{n \leq X} a_n 
    \leq \sum_{n \geq 1} a_n e^{1-n/X}.
    \]
    Performing a Mellin inversion, we have
    \[
    \#\{ \fraka: N_{M/\Q}(\fraka) \leq X\}
    \leq
    \frac{1}{2\pi i}
    \int_{(2)} \zeta_M(s) X^s \Gamma(s) \mathrm{d}s.
    \]
    Let $T>0$ be a parameter that we will ultimately take going to infinity. We first truncate the integral to
    $$
    \int_{(2)} \zeta_M(s) X^s \Gamma(s) \mathrm{d}s=\int_{2-iT}^{2+iT} \zeta_M(s) X^s \Gamma(s) \mathrm{d}s+\int_{|{\rm{Im}}(s)|>T}\zeta_M(s) X^s \Gamma(s) \mathrm{d}s
    $$
    and note that for $s=2+it$ and $|t|\geqslant T$
    $$
    \zeta_M(s) X^s \Gamma(s) \ll \zeta(2)^4 X^2 |t|^{3/2}e^{-\frac{\pi |t|}{2}}
    $$
    by Stirling's formula. Hence,
    $$
    \int_{|{\rm{Im}}(s)|>T}\zeta_M(s) X^s \Gamma(s) \mathrm{d}s \ll X^2 \int_{|t|\geqslant T} |t|^{3/2} e^{-\frac{\pi |t|}{2}}\mbox{d}t=o_{T\to +\infty} (1),
    $$
    since this is the remainder of a convergent integral.
    We now move the contour back to the half line, in the course of which we acquire the contribution of the residue at the simple pole $s=1$. Thus
    \[
    \#\{ \fraka: N_{M/\Q}(\fraka) \leq X\}
    \leq
    X \res_{s=1} \zeta_M(s) +
    \frac{1}{2\pi i}
    \int_{\frac{1}{2}-iT}^{\frac{1}{2}+iT} \zeta_M(s) X^s \Gamma(s) \mathrm{d}s+|I_{-}|+|I_{+}|+o_{T\to +\infty} (1)
    \]
    where 
    $$
    I_{\pm}=\int_{\frac{1}{2} \pm iT}^{2\pm iT} \zeta_M(s) X^s \Gamma(s) \mathrm{d}s.
    $$
     Note that $\zeta_M(s) = \zeta(s) L(s, \chi_K) L(s, \chi_L) L(s, \chi_K\chi_L)$ and that we have available the following Weyl-type subconvexity bounds
    \[
        \left \vert \zeta\left( \frac{1}{2} +it \right) \right \vert \ll (1 + \vert t \vert)^{1/6  + \epsilon} \text{ and }
        \left \vert L\left( \frac{1}{2} +it, \chi \right) \right \vert \ll (1 + q\vert t \vert)^{1/6  + \epsilon},
    \] for $\chi$ a Dirichlet character of modulus $q$.
    The first bound is classical and due to Weyl, the second is work of Petrow--Young \cite{MR4624371}.
    Moreover, from Stirling's formula we infer the bound
    \[
    \left \vert \Gamma \left( \sigma + i t \right) \right \vert
    \ll e^{-\frac{\pi \vert t \vert}{2}}\vert t \vert^{\sigma - \frac{1}{2}}.
    \]
In combination, we have for $s=\sigma \pm iT$ with $\sigma \leq 1$,
$$
\zeta_M(s) X^s \Gamma(s)\ll e^{-\frac{\pi T}{2}} X^{\sigma} T^{\sigma - \frac{1}{2} + \frac{4(1 - \sigma)}{3} + \epsilon}|\Delta_{K/\Q}|^{\frac{2(1-\sigma)}{3}},
$$
for $\sigma>1$, we use instead the bound $\zeta_M(s) \leq \zeta(s)^4.$ 
Together this implies that
$$
I_{\pm} \ll e^{-\frac{\pi T}{2}}T^{\frac{2}{3} + \epsilon}|\Delta_{K/\Q}|^{1/3}\int_{\frac{1}{2}}^2 X^{\sigma} \mbox{d}\sigma=o_{T\to +\infty} (1).
$$
   We therefore have
   $$
   \#\{ \fraka: N_{M/\Q}(\fraka) \leq X\}
    \leq
    X \res_{s=1} \zeta_M(s) +
    \frac{1}{2\pi i}
    \int_{\frac{1}{2}-iT}^{\frac{1}{2}+iT} \zeta_M(s) X^s \Gamma(s) \mathrm{d}s+o_{T\to +\infty} (1).
   $$
    Another application of the subconvexity bounds and Stirling's formula produces the bound 
    \[
    \left \vert
    \zeta_M\left( \frac{1}{2} + it\right)
    X^{\frac{1}{2}+it}
    \Gamma \left( \frac{1}{2} + it\right)
    \right \vert
    \ll X^{\frac{1}{2}}
    \vert \Delta_{K/\Q} \vert^{\frac{1}{3} + \epsilon}
    \vert t \vert^{\frac{2}{3} + \epsilon}
    e^{-\frac{\pi \vert t \vert}{2}}.
    \]
    The remaining $t$ integral is then convergent and it suffices to bound $\int_{-T}^T$ by $\int_{\R}$ and to let $T$ go to (positive) infinity.
\end{proof}

We are now ready to prove the main theorem of this section. 

\begin{proof}[Proof of Theorem \ref{thm:widmer+}] 
Fix $\epsilon>0$. Throughout all implicit constants are allowed to depend on $d$, hence $L$, and on $\epsilon$.
For any ideal $\frakd$ and parameter $T \geq 1$, Proposition~\ref{prop:counting} gives the inequality 
    \[
    \left\vert
Z^*(\Lambda(\frakd), S_{\mathbf i}(T) )  - \frac{\vol S_{\mathbf i}(T)}{\det  \Lambda(\frakd)} \right \vert\ll
\frac{T^7}{\lambda_1(\frakd, \mathbf i) \det(\frakd)^{3/2}}
.
    \]
    Applying Proposition~\ref{prop:prod suc min} and Lemma \ref{lem:determinant of an ideal}, we may rewrite the error so that 
    \[
    Z^*(\Lambda(\frakd), S_{\mathbf i}(T) )
=
\frac{\vol S_{\mathbf i}(T)}{\det \Lambda(\frakd)} 
+
O \left(
\frac{T^7}{\lambda_1(\mathbf i, \frakd) \vert\Delta_{K/\Q}\vert^{3/2} N_{M/\Q}(\frakd)^{3/2}}
\right).
\]
Upon taking $\frakd =\fraka\frakb$, $T=N_{M/\Q}(\fraka)^{1/4}Y$ and applying Lemma \ref{lem:widmerdecomp} we see that  
\[
\begin{aligned}
 Z(\PP^1,K,d,Y) =& \frac{1}{q_Mw_M}
\sum_{\fraka} \sum_{0 \neq \frakb \subseteq \mathcal{O}_M} \mu_M(\frakb) \sum_{\mathbf i \in \mathcal I}
\left[\frac{\vol S_{\mathbf i}( N_{M/\Q}(\fraka)^{1/4}Y )}{\det \Lambda(\fraka\frakb)}\right. \\
&+
\left. O \left(
\frac{(N_{M/\Q}(\fraka)^{1/4}Y)^7}{\lambda_1(\mathbf i, \fraka \frakb) \vert\Delta_{K/\Q}\vert^{3/2} N_{M/\Q}(\fraka \frakb)^{3/2}}
\right) \right].
\end{aligned}
    \]
    By Proposition~\ref{prop:volume of quasifundamental domain}, we have that
    \[
    \sum_{\mathbf i \in \calI} \vol S_{\mathbf i }(N_{M/\Q}(\fraka)^{1/4}Y)
    =
    N_{M/\Q}(\fraka)^2 Y^8 \sum_{\mathbf i \in \calI} \vol S_{\mathbf i} (1) 
    =
    N_{M/\Q}(\fraka)^2 Y^8 2^{3r+ s -1}\pi^{2s}q_MR_M.
    \]
    The remaining sums in the main term are precisely as in \cite[p.\ 33]{MR2645051}, hence we conclude in the same manner that the main term is $c_{\PP^1, M}Y^8$. 
    We turn to the error term.

    Let $\calA$ and $\calB$ be the ideal classes of $\fraka$ and $\frakb$ respectively. 
   Then it follows from Proposition \ref{prop:secondschmidttrick} that  
   
    \[
        \sum_ {\mathbf i \in \mathcal I} \frac{(N_{M/\Q}(\fraka)^{1/4}Y)^7}{\lambda_1(\mathbf i, \fraka \frakb) \vert\Delta_{K/\Q}\vert^{3/2} N_{M/\Q}(\fraka \frakb)^{3/2}} \ll \frac{Y^7 N((\mathcal{AB})^{-1})}{\vert\Delta_{K/\Q}\vert^{3/2} N_{M/\Q}(\frakb)^{7/4}}.
    \]
   The $\frakb$ sum converges absolutely and both the sum over ideal classes and that defining $N((\calA\calB)^{-1})$ are finite so we can exchange summations. The final computation remaining is therefore a sum over $th_M$ ideals of the reciprocal of the norms. We cannot handle this sum in the elementary manner of the proof of \cite[Thm.\ 2]{MR1330740} as this would result in an outsized power of $\log Y$. Instead, we will use partial summation and Lemma~\ref{lem:landau}. Before this, we must handle the unusual summation condition. In order to do so we split the sum into ideals of norm less than $th_M$ and those of larger norm. Let us suppose that there are $I$ ideals of norm less than $th_M$. Then
\[
\sum_{i=1}^{th_M} N(\frakc_i)^{-1/4}
\ll
\sum_{N_{M/\Q}(\frakc) \leq th_M} N(\frakc)^{-1/4}
+
\sum_{i = I+1}^{th_M} N(\frakc_i)^{-1/4}.
\]
Each summand in the second sum is at most $(th_M)^{-1/4}$ and there are at most $th_M$ terms so the sum is bounded by $(th_M)^{3/4}$. Applying partial summation to Lemma~\ref{lem:landau}, and writing $X=th_M$, the former sum is bounded by
\[
 \frac{X \res_{s=1} \zeta_M(s) + X^{1/2}\vert \Delta_{K/\Q} \vert^{1/3 + \epsilon}}{X^{1/4}}
+
\int_1^{X} \frac{u \res_{s=1} \zeta_M(s) + u^{1/2}\vert \Delta_{K/\Q} \vert^{1/3 + \epsilon}}{u^{5/4}}\mathrm{d}u
\]
\[
\ll X^{3/4} \res_{s=1} \zeta_M(s) + X^{1/4} \vert \Delta_{K/\Q} \vert^{1/3 + \epsilon }.
\]
Since $th_M \ll q_MR_Mh_M \ll h_MR_M$ by Proposition~\ref{prop:t asymp qMhM} with an implicit constant independent of $K$, we find the a total error term of order
\[\frac{(h_M R_M)^{3/4} + (h_M R_M)^{3/4}  \res_{s=1} \zeta_M(s) + (h_M R_M)^{1/4} \vert \Delta_{K/\Q} \vert^{1/3 + \epsilon }}{|\Delta_{K/\Q}|^{3/2}}.
\]
By the Brauer--Siegel theorem the first term dominates the third, and we arrive at the required bound.
\end{proof}

\begin{remark}
    Our passage to the $\calO_L$-lattice structure was critical in providing the error term which we have achieved in Theorem \ref{thm:widmer+}. Without doing so, merely treating $\Lambda(\frakd)$ as a $\Z$-lattice one would need control of four successive minima. Using Widmer's lower bounds on the successive minima (\cite[Lem.\ 9.7]{MR2645051}) one could achieve a result of the shape
    \[
    \left\vert
Z^*(\Lambda(\frakd), S_{\mathbf i}(T) )  - \frac{\vol S_{\mathbf i}(T)}{\det  \Lambda(\frakd)} \right \vert\ll
\frac{T^7}{\lambda_{\Q,1}(\frakd, \mathbf i)^4\lambda_{\Q,3}(\frakd, \mathbf i)^3} \ll \frac{T^7}{N_{M/\Q}(\frakd)^7 \vert \Delta_{M/\Q}\vert^{3/8}},
    \]
    ultimately culminating (assuming equivalent savings in the $\mathbf i$ and class group sums) in an estimate of the form
    \[
    Z(\PP^1,K,d,Y)
=
c_{\PP^1, M}Y^8+O_d\left(Y^7\frac{\log^4(h_MR_M)}{|\Delta_{K/\Q}|^{3/4}}\right).
    \]
    This is just barely insufficient for Theorem \ref{thm:mainthm1} since the error term overwhelms the main term once $\vert \Delta_{K/\Q} \vert \geq  B^{1/2}(\log B)^{-3}$.
\end{remark}

\subsection{Proof of Theorem \ref{thm:quadxd}}\label{finalsection}
We now conclude the proof of Theorem~\ref{thm:quadxd}. For the optimal cutoff $\gamma=1/2$ for $X_d$ thanks to Lemma~\ref{lem:constsxd}, we have
\begin{align*}
\#\{ x \in \mathcal U: H(x) \leq B\}
&= \frac12
\sum_{\substack{[K:\Q]=2\\ |\Delta_{K/\Q}| \ll B^\gamma}} \#\{P \in X_d(K) : \Q(P)=K, H(P)H(\overline{P}) \leq B\}\\
&=\frac12 \sum_{\substack{[K:\Q]=2\\ |\Delta_{K/\Q}| \ll B^\gamma}} \#\{P \in X_d(K) : \Q(P)=K, H(P) \leq B^{1/2}\}.
\end{align*}

By virtue of the height which we have chosen, the one induced by the na\"ive relative $\mathcal{O}(1)$-height $\resH_{\calO(1)}$ on $\PP^1$, the inner cardinality is precisely
\[
\#\{ Q \in \PP^1(M) : L(Q) = M, \resH_{\calO(1)}(Q) \leq B^{1/2}\}
=
Z(\PP^1, K, d, B^{1/8}).
\]
Applying Theorem \ref{thm:widmer+}, the cardinality $\#\{ x \in \mathcal U: H(x) \leq B\}$ is given by 
    \[
    \frac12 \sum_{\substack{[K:\Q]=2\\ \vert \Delta_{K/\Q} \vert\ll B^\gamma}} \left[ c_{\PP^1, M} B + O_d\left(B^{7/8}\frac{(h_MR_M)^{3/4} (1+\res_{s=1} \zeta_M(s))}{|\Delta_{K/\Q}|^{3/2}}\right)\right].
    \]
We first perform the sum of the main terms.

\begin{lemma}\label{lem:constsumxd}
    We have
    \[
    \frac12 \sum_{\substack{[K:\Q]=2\\ |\Delta_{K/\Q}| \ll B^\gamma}} c_{\PP^1, M}  = \frac{\gamma}4  \tau(\Hilb^2 X_d) \log B +O( 1).
    \]
    When summing up to the optimal cutoff $\gamma=\frac12$ the leading constant equals $c_{\Hilb^2 X_d, \Q}$.
\end{lemma}

\begin{proof}
  We have $c_{\PP^1, M} = \frac12 \tau(\PP^1_M)$, and from \cite[Thm.\ 4.3]{MR3430268} we obtain
  \[
  \tau(\PP^1_M) = \tau(\Res_{M/K}\PP^1_M) = \tau(X_{d,K}),
  \]
  since $K$ and $L$ are linearly disjoint.
  Applying Theorem~\ref{thm:summingc}, we have
  \[
  \sum_{\substack{[K:\Q]=2\\ |\Delta_{K/\Q}| \ll B^\gamma}} \tau(X_{d, K})
  =
  \tau(\Hilb^2 X_d) \log(B^\gamma) +O(1),
  \]
  which proves the first result. The last statement follows from the values $\alpha(\Hilb^2 X_d)=\tfrac18$ and $\beta(\Hilb^2 X_d)=1$ computed in Lemma~\ref{lem:constsxd}.
\end{proof}
Next, we gather estimates on the sums of the error terms.

\begin{lemma}\label{lem:errsumxd}
Fix $\epsilon >0$ and $\gamma >0$. Recall that $M=KL= K(\sqrt{d})$. We have 
%
\[
\sum \limits_{\substack{[K: \Q]=2 \\ |\Delta_K| \ll B^\gamma}}
\frac{(h_MR_M)^{3/4} \left(1+\res_{s=1} \zeta_M(s)\right)}{\vert\Delta_{K/\Q}\vert^{3/2}} \ll B^{\gamma/4} (\log B)^{3/4}.
\]
where the implicit constant is allowed to depend on $\epsilon, \gamma$ and $L$.
\end{lemma}

\begin{proof}

        We start by splitting the sum into two separate sums
        \[
        \displaystyle\sum \limits_{\substack{[K: \Q]=2 \\ |\Delta_K| \ll B^\gamma}}
\frac{(h_MR_M)^{3/4} }{\vert\Delta_{K/\Q}\vert^{3/2}}+\displaystyle\sum \limits_{\substack{[K: \Q]=2 \\ |\Delta_K| \ll B^\gamma}}
\frac{(h_MR_M)^{3/4} \res_{s=1} \zeta_M(s)}{\vert\Delta_{K/\Q}\vert^{3/2}}.
        \]
We start with the first term. Applying H\"older's inequality, we may bound this by
\[
\ll B^{\gamma/4}
\left(\displaystyle\sum \limits_{\substack{[K: \Q]=2 \\ |\Delta_K| \ll B^\gamma}}
\frac{h_MR_M }{\vert\Delta_{K/\Q}\vert^2}\right)^{3/4}.
\]
Using $|\Delta_{M/\Q}|  = |\Delta_{K/\Q}|^2|\Delta_{L/\Q}|^2$ and the class number formula, we have
    \[
        \frac{h_MR_M}{\vert \Delta_{K/\Q} \vert^2}
        \ll_L \frac1{\vert \Delta_{K/\Q} \vert} \frac{h_MR_M}{\vert \Delta_{M/\Q} \vert^{1/2}} =
        \frac{2^{-r_M} (2\pi)^{-s_M}w_M\res_{s=1} \zeta_M(s)}{\vert \Delta_{K/\Q} \vert}.
    \] 
    Since $M$ is a biquadratic field, the Dedekind zeta function can be  decomposed and we see that 
    \[
        \res_{s=1} \zeta_M(s)
        = L(1, \chi_K) L(1, \chi_L) L(1,\chi_K\chi_L).
    \]
    The value $L(1, \chi_L)$ is a constant which does not depend on $K$ and hence can be incorporated into the implicit constant. 
    As in the proof of Theorem \ref{thm:summingc}, we first handle the sum without the term $\vert \Delta_{K/\Q} \vert$. Our starting point then is the sum for $Y\geqslant 1$
    \[
     \sum \limits_{\substack{[K: \Q]=2 \\  |\Delta_K| \leq Y}} 2^{-r_M} (2\pi)^{-s_M} L(1, \chi_K)L(1,\chi_K \chi_L).
    \]
    The remaining $L$-function $L(s, \chi_K)L(s, \chi_K \chi_L)$ is precisely the one attached to the Galois representation $\mathbf{1} \oplus \chi_L$ twisted by $\chi_K$. Thus we may separate the fields according to their embedding type and number of roots of unity and
    appeal to Corollary~\ref{cor:galtwists} with the conditions incorporated into the sequence $c_\Delta$ and $t=1$ to deduce the upper bound
    \[
        \sum_{\substack{(r,s)\\
    w \in \{2,3,4,6,8,12\}}} \hspace{-0.6cm} 2^{-r}(2\pi)^{-s}w
    \hspace{-0.5cm} \sum \limits_{\substack{[K: \Q]=2 \\ |\Delta_K| \leq Y\\ (r_M, s_M) = (r,s)\\ w_M = w}}
    L(1, \chi_K)L(1,\chi_{K}\chi_L)
    \ll Y.
    \]
    Thus by partial summation, we arrive at
    \[
    \displaystyle\sum \limits_{\substack{[K: \Q]=2 \\ |\Delta_K| \ll B^\gamma}}
\frac{h_MR_M }{\vert\Delta_{K/\Q}\vert^2}\ll \gamma \log B,
    \] as desired.

    The second sum is handled similarly. Again, one may apply H\"older to arrive at the bound
    \[
    \ll
    \left( \displaystyle\sum \limits_{\substack{[K: \Q]=2 \\ |\Delta_K| \ll B^\gamma}}
\frac{h_MR_M }{\vert\Delta_{K/\Q}\vert^2}\right)^{3/4}\left(\displaystyle\sum \limits_{\substack{[K: \Q]=2 \\ |\Delta_K| \ll B^\gamma}}
(\res_{s=1} \zeta_M(s))^4 \right)^{1/4}
.
    \]
    We know already that the first term is $O((\log B)^{3/4})$, hence it suffices to produce the bound $O(B^{\gamma/4})$ for the second term. This is accomplished in the same way by considering
    \[
      \sum \limits_{\substack{[K: \Q]=2 \\  |\Delta_K| \leq B^\gamma}}  L(1, \chi_K)^4L(1,\chi_K \chi_L)^4,
    \]
    which is bounded by $O_d(B^\gamma)$ by the $t=4$ case of Corollary~\ref{cor:galtwists}.
\end{proof}
 In conclusion, upon summing up to the optimal cutoff $\gamma=\frac12$ we find that
\begin{align*}
   \#\{ x \in \mathcal U: H(x) \leq B\}
   &= \frac{\gamma}4  \tau(\Hilb^2 X_d) B \log B + O(B) + O_d(B^{7/8+\gamma/4}(\log B)^{3/4} )\\
   &=c_{\Hilb^2 X_d, \Q} B\log B + O_d(B(\log B)^{3/4}),
\end{align*}
establishing Theorem~\ref{thm:quadxd} and with it completing the proof of Theorem~\ref{thm:mainthm1}.

From the success of the approach we rediscover, as in Example~\ref{ex:Manin implies optimal cutoff} and Remark~\ref{rem:p1p1}, that the cutoff $\gamma=\frac12$ for $X_d$ is indeed optimal.

\begin{remark}\label{rem:infinitely many heights}
    As Widmer has shown the geometry of numbers approach is compatible with any adelic height which satisfies the Lipschitz parametrisability conditions required for counting. While we have only established our results for the height induced by the standard height on $\PP^1$, the arguments do not use this explicitly and we could have chosen any anticanonical height on $X_d$ induced from an adelic Lipschitz system on $\PP^1$. 
\end{remark}

\bibliographystyle{alphaabbr}
\bibliography{bibliography.bib}

\end{document}